\documentclass[11pt,reqno]{amsart}
\usepackage{amsmath,amsfonts,amsthm,amssymb}
\usepackage[colorlinks,citecolor=blue,linkcolor=red]{hyperref}
\usepackage{graphicx}
\usepackage{color}
\usepackage{tikz-cd}

\catcode`\@=11

\long\def\@savemarbox#1#2{\global\setbox#1\vtop{\hsize\marginparwidth 
  \@parboxrestore\tiny\raggedright #2}}
\newcommand\lref[1]{\ref{#1}%
\@ifundefined{r@DisplaY #1}{}{ (#1)}}

\newcommand\fakelabel[2]{\@bsphack\if@filesw {\let\thepage\relax
   \newcommand\protect{\noexpand\noexpand\noexpand}%
\xdef\@gtempa{\write\@auxout{\string
      \newlabel{#1}{{#2}{\thepage}}}}}\@gtempa
   \if@nobreak \ifvmode\nobreak\fi\fi\fi\@esphack}

\def\SL@margintext#1{{\showlabelsetlabel{\tiny\{\SL@prlabelname{#1}\}}}}
\catcode`\@=12

\def\Empty{}
\newcommand\oplabel[1]{
  \def\OpArg{#1} \ifx \OpArg\Empty {} \else
        \label{#1}
  \fi}
\newtheorem{theoremSt}{Theorem}[section]

\newtheorem{exampleSt}[theoremSt]{Example}
\newtheorem{exerciseSt}[theoremSt]{Exercise}

\newcommand\MakeStEnv[1]{
  \newenvironment{#1}[1]{
  \begin{#1St} \oplabel{##1}%
  \global\def\CrntSt{\thetheoremSt}%
}{ 
  \end{#1St} }
  \newenvironment{#1+}[1]{
  \begin{#1St} \label{##1}%
  \label{DisplaY ##1}%
  \global\def\CrntSt{\thetheoremSt}%
  \def\Labl{##1}\ifx\Labl\Empty{} \else {\em (\Labl)\,}\fi%
}{ 
  \end{#1St} }
}
\MakeStEnv{theorem}
\MakeStEnv{corollary}
\MakeStEnv{proposition}
\MakeStEnv{lemma}
\MakeStEnv{definition}
\MakeStEnv{conjecture}
\MakeStEnv{problem}
\MakeStEnv{question}

\long\def\realfig#1#2{
\begin{figure}[htbp]
\includegraphics{#1}
\caption[#1]{#2}
\oplabel{#1}
\end{figure}}

\newlength{\saveu}

\newenvironment{pf*}[1]{%
 \begin{proof}[#1]%
}{ 
 \end{proof}
}

\newcommand{\finishproof}[1]{ 
  \def\FPArg{#1}
  \ifx\FPArg\Empty
        \newcommand\FPArg{\CrntSt}  \fi
  \smallbreak\noindent\makebox[\textwidth]{\hfill\fbox{\FPArg}}
  \medbreak\noindent
}

\newcommand\BB{{\mathcal B}}
\newcommand\CC{{\mathcal C}}

\newcommand\EE{{\mathcal E}}
\newcommand\FF{{\mathcal F}}
\newcommand\GG{{\mathcal G}}
\newcommand\HH{{\mathcal H}}

\newcommand\LL{{\mathcal L}}
\newcommand\MM{{\mathcal M}}
\newcommand\NN{{\mathcal N}}

\newcommand\PP{{\mathcal P}}
\newcommand\QQ{{\mathcal Q}}
\newcommand\RR{{\mathcal R}}

\newcommand\TT{{\mathcal T}}
\newcommand\UU{{\mathcal U}}
\newcommand\VV{{\mathcal V}}
\newcommand\WW{{\mathcal W}}
\newcommand\XX{{\mathcal X}}
\newcommand\YY{{\mathcal Y}}
\newcommand\ZZ{{\mathcal Z}}

\newcommand\PMF{{\PP\kern-2pt\MM\FF}}

\newcommand\PML{{\PP\kern-2pt\MM\LL}}

\newcommand\ep{\epsilon}

\newcommand\union{\cup}
\newcommand\intersect{\cap}
\newcommand\bbR{{\mathord{\text{I\kern-2pt R}}}}        
\newcommand\bbH{{\mathord{\text{I\kern-2pt H}}}}        

\newcommand\R{{\mathbb R}}

\newcommand\bigrightarrow[1]{\hbox to #1{\rightarrowfill}}
\newcommand\bigleftarrow[1]{\hbox to #1{\leftarrowfill}}

\newcommand\boundary{\partial}
\newcommand\semidir{\mathrel{\hbox{\vrule depth-.03ex height1.1ex\kern-0.15em$\times$}}}

\newcommand{\ssm}{\smallsetminus}
\newcommand{\diam}{\operatorname{diam}}

\numberwithin{equation}{section}

\catcode`\@=11

\def\subsection{\@startsection{subsection}{2}%
  \z@{.5\linespacing\@plus.7\linespacing}{.5em}%
  {\normalfont\bfseries\centering}}

\def\section{\@startsection{section}{1}%
  \z@{.7\linespacing\@plus\linespacing}{.5\linespacing}%
  {\normalfont\large\bfseries\centering}}

\def\subsubsection{\@startsection{subsubsection}{3}%
  \z@{.5\linespacing\@plus.7\linespacing}{-.5em}%
  {\normalfont\bfseries}}

\catcode`\@=12

\newcommand{\dist}{\operatorname{dist}}

\newcommand{\fsubd}{\mathrel{{\scriptstyle\searrow}\kern-1ex^d\kern0.5ex}}
\newcommand{\bsubd}{\mathrel{{\scriptstyle\swarrow}\kern-1.6ex^d\kern0.8ex}}
\newcommand{\fsubeq}{\mathrel{\raise-.7ex\hbox{$\overset{\searrow}{=}$}}}
\newcommand{\bsubeq}{\mathrel{\raise-.7ex\hbox{$\overset{\swarrow}{=}$}}}

\newcommand{\tsh}[1]{\left\{\kern-.9ex\left\{#1\right\}\kern-.9ex\right\}}
\newcommand{\Tsh}[2]{\tsh{#2}_{#1}}

\newcommand\Teich{{\mathcal T}}

\usepackage[cmtip,arrow]{xy}
\usepackage{pb-diagram, pb-xy}
\usepackage{overpic}
\usepackage{multicol}

\newcommand{\transverse}{\pitchfork}
\newcommand{\orth}{\perp}
\newcommand{\nest}{\sqsubset}
\newcommand{\move}{\mathrm{Move}}
\newcommand\Ext{\operatorname{{\mathbf E}}}
\newcommand{\Trans}{\mathrm{{\mathbf T}}}
\newcommand{\delete}{\mathrm{del}}

\newcommand{\Trim}{\mathrm{Trim}}
\newcommand{\res}{\mathrm{Res}}
\newcommand{\hd}{\widehat{d}}
\newcommand{\hpi}{\widehat{\pi}}
\newcommand{\hrho}{\widehat{\rho}}

\newcommand{\symdiff}{\bigtriangleup}

\newcommand{\MCG}{\mathrm{MCG}}

\newcommand{\rel}{\mathrm{Rel}}

\newcommand{\cuco}[1]{{\mathcal #1}}
\newcommand{\OL}{\overleftarrow}
\newcommand{\OR}{\overrightarrow}
\newcommand{\pt}{{\mathbf p}}

\newcommand{\tup}[1]{\vec{#1}}
\newcommand{\ignore}[2]{\Tsh{#2}{#1}}

 \usepackage{mathabx}
\newcommand{\propnest}{\sqsubsetneq}

\usepackage[shortlabels]{enumitem}

\newtheorem{thmi}{Theorem}
\newtheorem{cori}[thmi]{Corollary}

\newtheorem{rem}{Remark}
\newtheorem{claim}{Claim}

\begin{document}

\title{Stable cubulations, bicombings and barycenters}
\author[M. G. Durham]{Matthew G. Durham}
	\address{Department of Mathematics, University of California, Riverside, CA}
	\email{mdurham@ucr.edu}
\author[Y. N. Minsky]{Yair N. Minsky}
	\address{Department of Mathematics, Yale University, New Haven, CT}
	\email{yair.minsky@yale.edu}
\author[A. Sisto]{Alessandro Sisto}
	\address{Department of Mathematics, Heriot-Watt University, Edinburgh, UK}
	\email{a.sisto@hw.ac.uk}
	
\date{\today}

\begin{abstract}
We prove that the hierarchical hulls of finite sets of points in mapping class groups and Teichm\"uller spaces are stably approximated by a CAT(0) cube complexes, strengthening a result of Behrstock-Hagen-Sisto.  As applications, we prove that mapping class groups are semihyperbolic and Teichm\"uller spaces are coarsely equivariantly bicombable, and both admit stable coarse barycenters.  Our results apply to the broader class of ``colorable" hierarchically hyperbolic spaces and groups.  
\end{abstract}

\maketitle



%
%
%
%
%
%
%
%
%
%

\section{Introduction}

Much of the coarse structure of mapping class groups has the flavor of CAT(0) geometry, in
spite of the fact that mapping class groups have no geometric actions on CAT(0) spaces
(Bridson \cite{bridson:noCAT0}). Manifestations of this include the relatively hyperbolic structure associated to curve complexes (Masur-Minsky \cite{MM_I}), and the equivariant embedding into
finite products of quasi-trees found by Bestvina-Bromberg-Fujiwara
\cite{BBF:qtree-products}.

A notion of ``hulls'' of finite sets in mapping class groups was
introduced in \cite{BKMM:qirigid}, and these were more recently shown in \cite{BHS:quasi} to be approximated in a
uniform way by finite CAT(0) cube complexes, see also the alternative proof given in \cite{Bow:cubulation}.  Our goal in this paper is to refine this
construction to make it {\em stable}, in the sense that perturbation of the input data
gives rise to bounded change in the cubical structure.  As initial applications, we give a
construction for equivariant barycenters and a proof that mapping class groups are bicombable. 

As in \cite{BHS:quasi}, the proof works in a more general context of {\em hierarchically
  hyperbolic groups}, a class of groups (and spaces) introduced by Behrstock-Hagen-Sisto
\cite{HHS_I,HHS_II} which are endowed with a structure similar to the hierarchical family
of curve complexes associated to a surface \cite{MM_II}. See Subsection \ref{subsec:HHSery} below for the definition of a hierarchically hyperbolic space (HHS).
 
week.

Our main result, stated informally, is the following: 
\begin{thmi}{}\label{thm:stable cubulation informal} 
  In a colorable HHS $(\XX,\mathfrak S)$, the coarse hull $H_\theta(F)$ of any finite set 
  $F$ can be approximated by a finite CAT(0) cube complex whose dimension is bounded by
  the complexity of $(\XX,\mathfrak S)$, in such a way that a bounded change in $F$ corresponds to a change
  of the cubical structure by a bounded number of hyperplane deletions and insertions.
\end{thmi}

The colorability assumption in Theorem \ref{thm:stable cubulation informal} is apparently quite weak and excludes none of the key examples of HHSes.  In fact, we are not aware of any noncolorable HHGs; see Definition \ref{defn:colorable}.

For the general context of this result, see the discussion in Subsection \ref{subsec:coarse hulls} below, where we also give a more precise statement in Theorem \ref{thm:stable cubulations mid-level}.  See Theorem \ref{thm:stable_cubulations} for the strongest version. Besides mapping class groups, there are several other classes of spaces and groups that are colorably hierarchically hyperbolic, including:

\begin{itemize}
 \item many cubical groups including all right-angled Artin and Coxeter groups \cite{HHS_I,HS:cubical},
 \item Teichm\"uller spaces with either the Teichm\"uller or the Weil-Petersson metric \cite{Rafi:combo,Dur:augmented, EMR:rank},
 \item fundamental groups of closed 3-manifolds without Nil or Sol summands \cite{HHS_II},
 \item groups resulting from various combination and small-cancellation-type theorems \cite{HHS_II,BR:combination,Spriano:combo,RS:combination,HHS:asdim,robbio2020hierarchical},
 \item quotients of mapping class groups by suitable large powers of Dehn twists, and other related quotients \cite{BHMS:kill_twists},
\item extensions of Veech subgroups of mapping class groups \cite{DDLS},
\item the genus-$2$ handlebody group \cite{miller2020stable}. 
\end{itemize}

With the exception of any hyperbolic and cubical examples from above, our main results and its applications are novel for this wide class of objects. 

\subsection{Applications}

We now discuss our two main applications of Theorem \ref{thm:stable cubulation informal}, namely that mapping class groups and Teichm\"uller spaces are bicombable (Corollary \ref{cor:MCG bicombing}) and admit stable barycenters (Corollary \ref{cor:MCG barycenter}).

\subsubsection*{Bicombings and semihyperbolicity}

In CAT(0) spaces, geodesics are unique.  In geodesic Gromov hyperbolic spaces, all geodesics between any pair of points fellow-travel.  In fact, in both of these classes of spaces geodesics are stable under perturbation of their endpoints in the following sense:
\begin{itemize}
\item Given points $x,x',y,y'$ with $d(x,y), d(x',y') \leq 1$, all geodesics between $x,y$ and $x',y'$ fellow-travel in a parametrized fashion.
\end{itemize}

The notion of \textbf{bicombing} of a metric space $X$, introduced by Thurston, generalizes this stability property.  Roughly speaking, a bicombing is a transitive family of uniform quasigeodesics with the above parametrized fellow-traveling property under perturbation of endpoints.  See Subsection \ref{subsec:bicombings thm} for a precise definition.

Bicombability is a quasi-isometry invariant which imposes strong constraints on groups, such as property $FP_{\infty}$, a quadratic isoperimetric inequality, and the Novikov conjecture \cite{AB95, BGS, ECH, GerSh, Storm:Novikov}.  Moreover, bicombings are the key geometric feature of biautomatic structures on groups (where one requires that the bicombing is constructible by a finite state automaton), thereby playing an important role in computational group theory.  It is worth noting that bicombability is decidedly a feature of nonpositive curvature, as amenable groups such as the 3-dimensional Heisenberg group are not bicombable.

The power of our stable cubical models is that they allow us to stably and hierarchically import geometric features of cube complexes into HHSes.  In particular, $\ell^1$-geodesics in the cubical models map to \emph{hierarchy paths} (Definition \ref{defn:hierarchy path}), which are quasigeodesics that are finely attuned to the HHS structure, in that they project to uniform, unparametrized quasigeodesics in every hyperbolic space in the hierarchical structure.  The stability property of the cubulation then implies that carefully chosen $\ell^1$-geodesics give a bicombing:

\begin{thmi}{}\label{thm:bicombing main}
Any colorable HHS $(\cuco X, \mathfrak S)$ admits a coarsely $\mathrm{Aut}(\cuco X,\mathfrak S)$-equivariant, discrete, bounded, quasi-geodesic bicombing by hierarchy paths with uniform constants.
\end{thmi}

If the action by automorphisms is free, then coarse equivariance can be upgraded to equivariance.  By the definition of semihyperbolicity \cite{AB95}, we obtain:

\begin{cori}{}\label{cor:semihyp}
Colorable hierarchically hyperbolic groups are semihyperbolic.
\end{cori}

Note that semihyperbolicity has several novel consequences for HHGs.  Besides novel consequences of bicombability, such as property $FP_{\infty}$,  these include solvability of the conjugacy problem and the fact that abelian subgroups are undistorted \cite{AB95}.

While many HHSes were known to be bicombable for other reasons, e.g. many are CAT(0), this produces bicombings for many new examples, such as extensions of Veech subgroups of mapping class groups.

Our main application is: 

\begin{cori}{}\label{cor:MCG bicombing}
For any finite type surface $\Sigma$, its mapping class group $\MCG(\Sigma)$ is semihyperbolic and its Teichm\"uller space $\mathrm{Teich}(\Sigma)$ with either the Teichm\"uller metric or the Weil-Petersson metric is coarsely $\MCG(\Sigma)$-equivariantly bicombable by hierarchy paths with uniform constants.
\end{cori}

Note that our notion of hierarchy path here is more general than the hierarchy paths produced in \cite{MM_II, Dur:augmented}.

We remark that semihyperbolicity of $\MCG(\Sigma)$
follows from work in a preprint of Hamenst\"adt \cite{Ham09}.  The result for $\Teich(\Sigma)$ is new,
though we were informed by M. Kapovich and K. Rafi that they know of a different
construction for bicombing  $\Teich(\Sigma)$.  Note that $\Teich(\Sigma)$ with the
Weil-Petersson metric is bicombable since its completion is CAT(0)
\cite{Wolpert:curvature, Tromba:curvature, bridson-haefliger}, though we note that it is unknown whether Weil-Petersson geodesics are hierarchy paths.  Combability of $\MCG(\Sigma)$ follows from work of Mosher \cite{Mosher:automatic}.

 Notably, our bicombing construction applies to both mapping class groups and Teichm\"uller spaces simultaneously.  Moreover, our bicombings are relatively straight-forward applications of our more powerful stable cubulation construction.  See Subsection \ref{subsec:proof discussion} below for a discussion.

\subsubsection*{Stable Barycenters}

Another key feature of nonpositively curved spaces is that bounded sets admit (coarse)
\emph{barycenters}. Here, we think of barycenters simply as maps assigning a point to any
finite subset. Some more properties are required to make this notion meaningful, such as \emph{stability}, which requires the barycenter to vary a bounded amount when the finite set varies a bounded amount, and \emph{coarse equivariance} when a group action is present; see Subsection \ref{subsec:barycenters thm}.

In CAT(0) spaces there are a number of useful notions of barycenter which are equivariant
and stable, for example center-of-mass constructions and circumcenters. 
Coarse barycenters are useful in the context of groups for understanding centralizers and
solving the conjugacy problem for torsion elements and subgroups.  Notably, 
Gromov hyperbolic spaces admit (coarse) barycenters:
a coarse barycenter of a finite set $F$ in a hyperbolic space $X$ can be taken
to be one of the standard CAT(0) barycenters in the CAT(0) space which models the hull of $F$ in $X$, i.e., a simplicial tree.  See Subsection \ref{subsec:coarse hulls} for a discussion of these ideas in the context of the this paper.

We should mention that coarse barycenters for triples of points are used to define coarse medians in the sense of Bowditch \cite{Bow:coarse_median}, thus playing a central role in the theory of coarse median spaces and its many applications.  However it is unclear how to construct barycenters even for pairs of points in a coarse median space, and stability properties appear just as difficult to obtain.

Barycenters in CAT(0) spaces are not in general well-behaved under quasi-isometries. Using Theorem \ref{thm:stable cubulation informal} and a construction reminiscent of Niblo-Reeves' normal paths \cite{NibloReeves}, we are able to prove that most HHSes admit equivariant coarse barycenters, which are coarsely invariant under HHS automorphisms:

\begin{thmi}{}\label{thm:barycenter main}
Let $(\cuco X, \mathfrak S)$ be a colorable HHS.  Then $\cuco X$ admits coarsely
$\mathrm{Aut}(\cuco X,\mathfrak S)$-equivariant stable barycenters for $k$ points, for any $k\ge 1$. 
\end{thmi}

We remark that the coarse barycenter we produce for a set $F$ is contained in the hull of $F$.

As with Theorem \ref{thm:bicombing main}, Theorem \ref{thm:barycenter main} can be applied
to mapping class groups and Teichm\"uller spaces:

\begin{cori}{}\label{cor:MCG barycenter}
For any finite type surface $\Sigma$, its mapping class group $\MCG(\Sigma)$ and
Teichm\"uller space $\Teich(\Sigma)$ admit coarsely $\MCG(\Sigma)$-equivariant stable barycenters for $k$ points, for any
$k\ge 1$. 
\end{cori}

Corollary \ref{cor:MCG barycenter} is new for arbitrary finite sets of points in $\MCG(\Sigma)$ and $\Teich(\Sigma)$ with the Teichm\"uller metric, even without the stability property.  The corresponding statement for $\Teich(\Sigma)$ with the Weil-Petersson metric is an easy consequence of the fact that its completion is CAT(0).  Corollary \ref{cor:MCG barycenter}, without the stability property, was proven for triples of points in $\MCG(\Sigma)$ by Behrstock-Minsky \cite{BM:centroids}, for orbits of finite order elements of $\MCG(\Sigma)$ in $\MCG(\Sigma)$ by Tao \cite{Tao:conj}, and more generally for orbits of finite subgroups of $\MCG(\Sigma)$ in both $\MCG(\Sigma)$ and $\Teich(\Sigma)$ with the Teichm\"uller metric in \cite{Dur:bary}.

As we were completing this paper, we learned that Haettel-Hoda-Petyt \cite{HHP} have simultaneously and independently proven that HHSes are coarse Helly spaces, in the sense of \cite{CCGHO}.  This property has a number of strong consequences, many of which overlap with the results in this paper.  In particular, they obtain versions of Theorems \ref{thm:bicombing main} and \ref{thm:barycenter main} along with their corollaries, without the colorability assumption and the hierarchy path conclusion.

Their approach and constructions are very different from ours, using results from the theory of coarse Helly and injective metric spaces, whereas our work relies mostly on hyperbolic and cubical geometry.

\subsection{Coarse hulls and their models}\label{subsec:coarse hulls}
Given the technical nature of many of the proofs in this paper, 
we include here an extended but simplified discussion of the ideas that go into our
constructions.  The propositions stated in this section will not, however, be used
elsewhere in the paper. 

Consider first the notion of a {\em convex hull} in a CAT(0) space. The convex hull of a
finite set $F$ is well-controlled, and in particular the map $F \mapsto hull(F)$ is
1-lipschitz with respect to the Hausdorff metric on sets. We are interested in
generalizing this notion to more coarse hulls (which we will just denote by $hull(F)$ in
each case) in more general spaces.

As a first motivating example, consider the Euclidean plane,  $X=\R^2$ with the $\ell^2$
metric. The convex hull of two points $hull(\{x,y\})$ is just the unique geodesic between
them. If, on the other hand,  we endow $\R^2$ with the $\ell^1$ metric than the convex hull is the axis-parallel
rectangle spanned by $x$ and $y$.  Note that $(\R^2,\ell^1)$ is not CAT(0) but is
a product of CAT(0) spaces, and this hull is a product of hulls in the CAT(0) factors. 
 See Figure \ref{fig:planehull}. This simple idea is a model for a useful construction in
the HHS context. 

\begin{figure}
\includegraphics[width=0.5\textwidth]{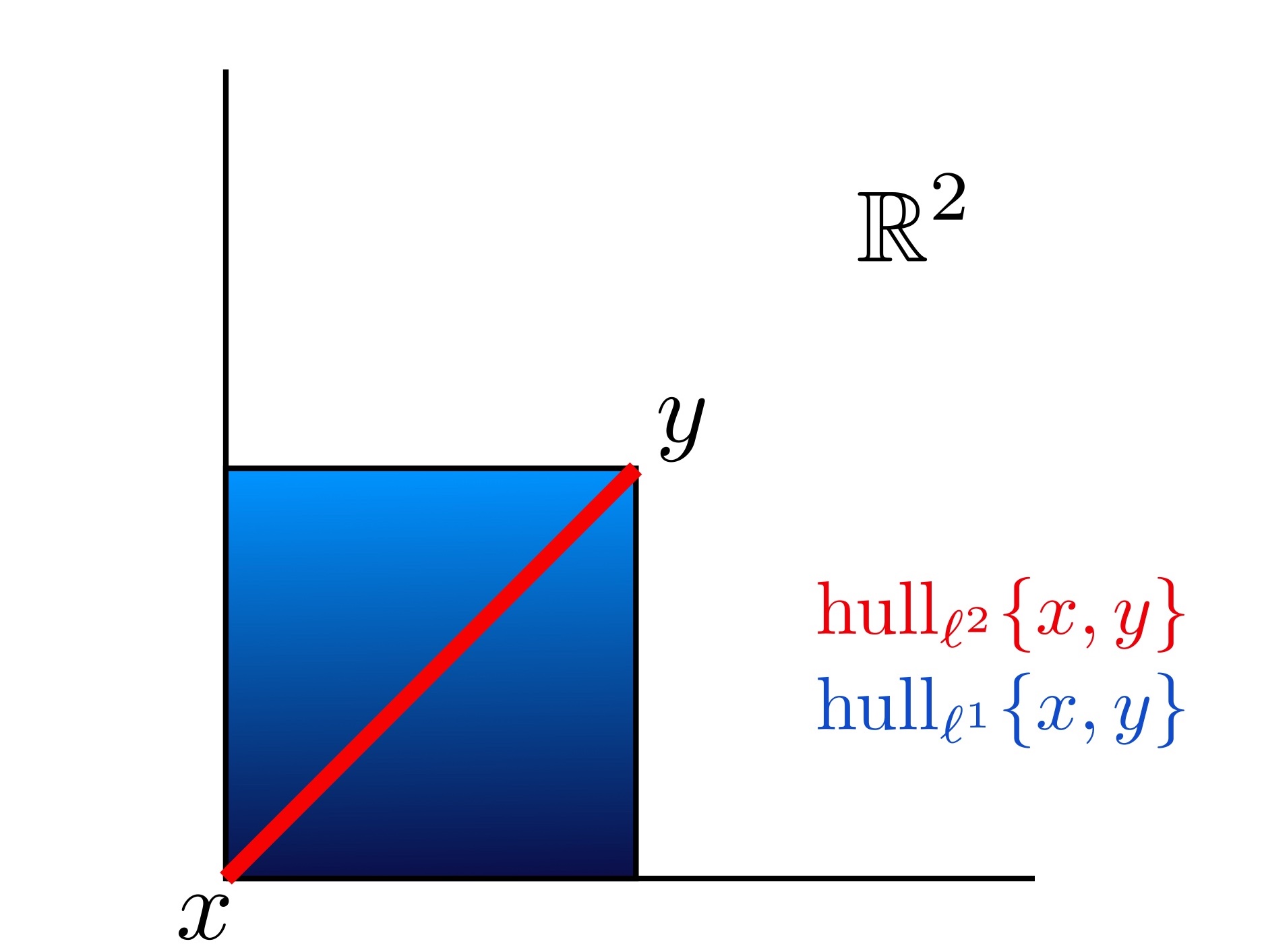}
\caption{A cartoon of the $\ell^2$-hull (red) and $\ell^1$-hull (blue) of two points in
  $\mathbb{R}^2$.  The $\ell^1$-hull reflects the intrinsic product structure of the space.}\label{fig:planehull}
\end{figure}

We can think of an HHS as (coarsely) embedded in a product of hyperbolic spaces, in such a way
that it is composed of products of certain of these factors, intersecting and nesting in a
complicated fashion. The reader familiar with the foundational example, namely
Masur-Minsky's hierarchy of curve graphs for mapping class groups \cite{MM_I, MM_II}, will
lose nothing by keeping it in mind during the ensuing discussion. In that setting,
Behrstock-Kleiner-Minsky-Mosher \cite{BKMM:qirigid} introduced a notion of hull which is
essentially a coarse pullback of convex hulls in each hyperbolic factor (see Subsection
\ref{subsec:HHSery}). Behrstock-Hagen-Sisto \cite{BHS:quasi} proved, in the general HHS setting,
that these hulls are quasi-isometrically modeled by finite CAT(0) cubical complexes.

Their result is a partial generalization of the situation in Gromov hyperbolic spaces, where Gromov proved that hulls of finite sets of points are quasi-isometrically modeled by finite simplicial trees \cite{Gromov87}.  However, in the setting of hyperbolic spaces, the modeling trees satisfy additional strong stability properties under perturbation of the set of input points; see Proposition \ref{prop:hyperhull} below.

Our main theorem (in increasing specificity, Theorems \ref{thm:stable cubulation informal}, \ref{thm:stable cubulations mid-level} and \ref{thm:stable_cubulations}) endows the modeling cube complexes from \cite{BHS:quasi} with a generalization of the stability properties that Gromov's modeling trees enjoy.

Before giving a full account of our results and an overview of their proofs, it will be beneficial to discuss the situation in hyperbolic spaces and cubical complexes.  We will see that our results are a common generalization of the situations from these motivating examples.  

\subsubsection*{Hulls in trees and cube complexes}

Let $X$ be a simplicial tree.  Then the convex hull of any finite set of vertices $F \subset X^0$ is the subtree $T_F$ of $X$ spanned by $F$.  Moreover, the subtree $T_F$ is stable under small perturbations of $F$, in the following sense (see Figure \ref{fig:treehull}):
\begin{proposition}{prop:treehull}
Let $X$ be a simplicial tree.  If $F,F' \subset X^0$ satisfy $\# F' = \# F = k$ and $d_{\mathrm{Haus}}(F,F') \leq 1$, then the intersection of their hulls $T_0 = T_F \cap T_{F'}$ is itself a subtree with both $T_F \ssm T_0$ and $T_{F'} \ssm T_0$ a union of at most $k$ subtrees each of diameter at most 1.
\end{proposition}
We will not use this fact, so we leave its proof to the interested reader.

\begin{figure}
\includegraphics[width=0.6\textwidth]{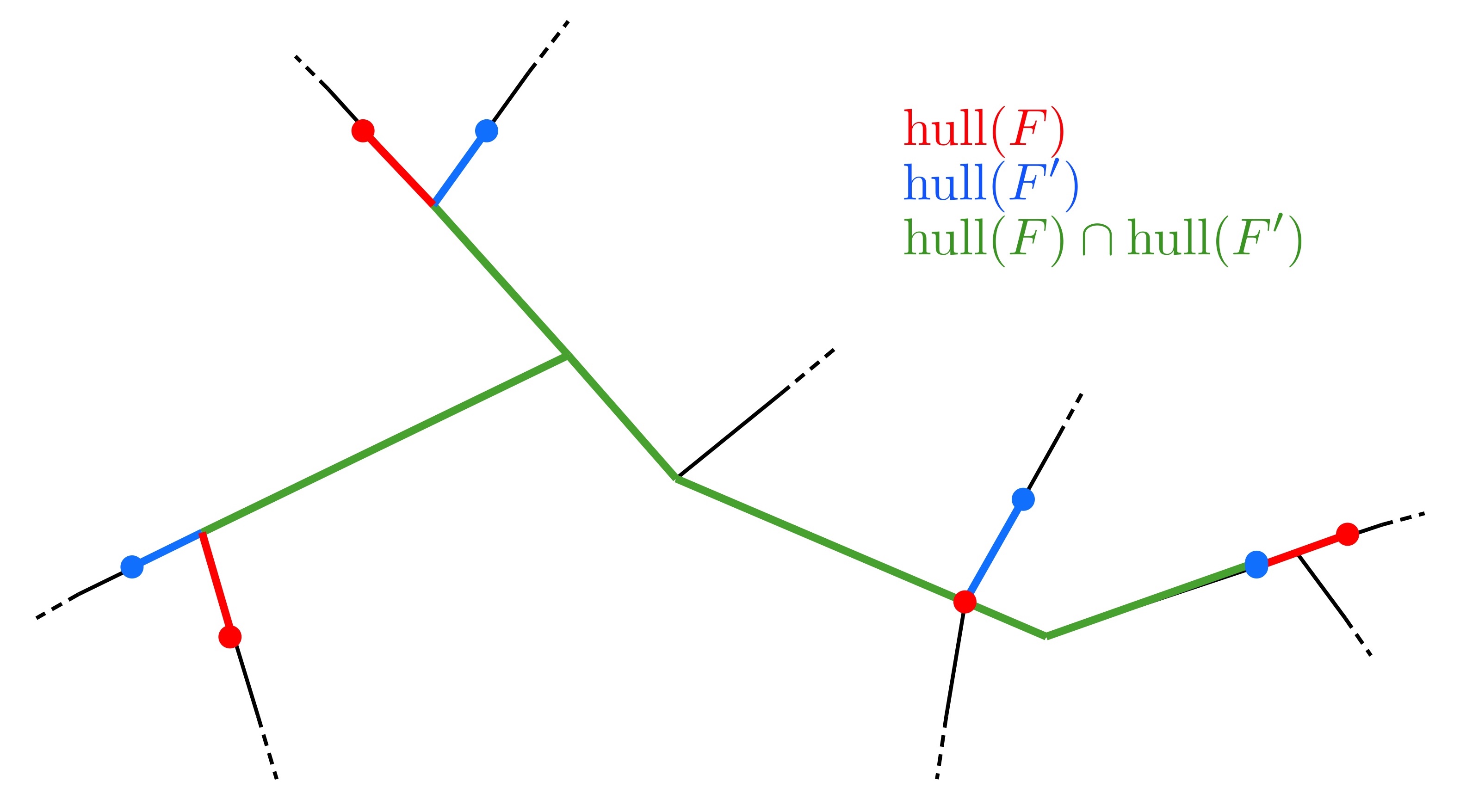}
\caption{Stability of hulls in a tree:  The intersection of $\mathrm{hull}(F)$ (red) and $\mathrm{hull}(F')$ (blue) subtrees is a (green) subtree which can be obtained by deleting the boundedly many complementary red and blue subtrees.}\label{fig:treehull}
\end{figure}

This situation generalizes to when $X$ is a CAT(0) cube complex endowed with the $l^1$-metric (see Subsection \ref{subsec:cube complexes} for the relevant definitions).  Recall that the $l^1$-metric on $X$ is completely determined by a special collection $\HH_X$ of codimension-1 subspaces called \textbf{hyperplanes} (Subsection \ref{subsec:cube complexes}), in the sense that $X$ is precisely the dual cube complex arising from Sageev's cubulation construction \cite{Sageev:cubulation} applied to $\HH_X$ as a wallspace (see Subsection \ref{subsec:walls}).

In the cubical context, the $\ell^1$ convex hull of any finite set of vertices $F \subset X^0$ is the cubical subcomplex $\QQ_F \subset X$ realized as the dual to the hyperplanes $\HH_F$ separating the points in $F$.  In addition, these cubical hulls satisfy the following strong stability property (see Figure \ref{fig:cubehull}): 

\begin{figure}
\includegraphics[width=0.75\textwidth]{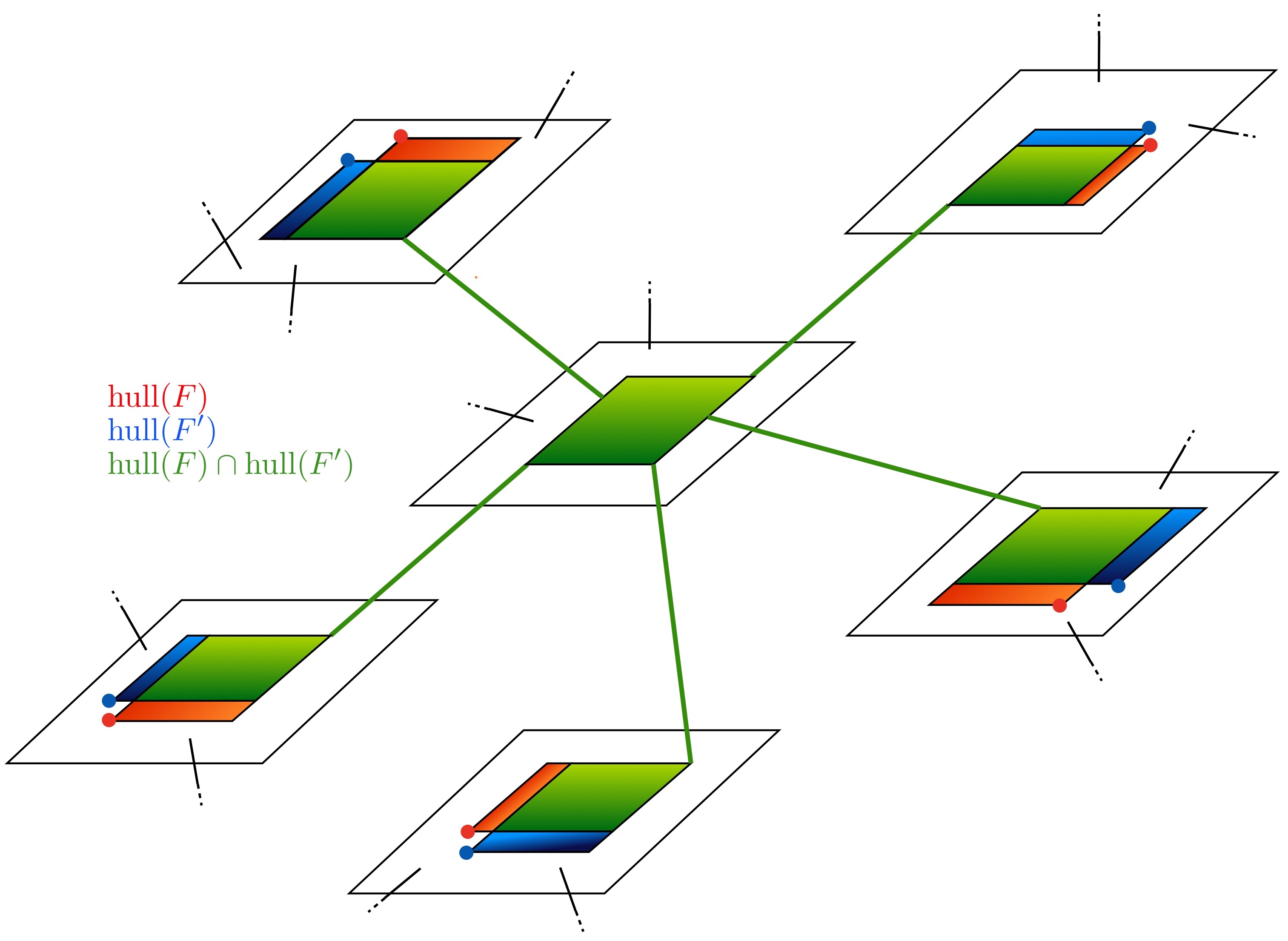}
\caption{Stability of hulls in the universal cover of $S^1 \vee T^2$: The subcomplex dual
  to all hyperplanes common to both $\mathrm{hull}(F)$ (red) and $\mathrm{hull}(F')$
  (blue)
is here realized as the intersection (green) of the $\ell^1$-hulls.}\label{fig:cubehull}
\end{figure}

\begin{proposition}{prop:cubehull}
Let $X$ be a CAT(0) cube complex endowed with the $\ell^1$ metric.  If $F,F' \subset X^0$
satisfy $\# F, \# F'\leq k$ and $d_{\mathrm{Haus}}(F,F') \leq 1$, then there are convex subcomplexes $X_F\subseteq\QQ_F$ and $X_{F'}\subseteq \QQ_{F'}$ both dual to the hyperplanes in $\HH_0 = \HH_F \cap
\HH_{F'}$ and so that $d_{\mathrm{Haus}}(X_F,X_{F'}) \leq 1$.  Moreover, both $\HH_F \ssm \HH_0$ and $\HH_{F'} \ssm \HH_0$ contain at most
$k$ hyperplanes.
\end{proposition}

Again, we will not use this proposition, so we omit its proof.

In the cubical structure on a simplicial tree, the hyperplanes correspond to midpoints of
edges.  Hence Proposition \ref{prop:cubehull} generalizes Proposition
\ref{prop:treehull}.  Note that now the diameters of $\QQ_F \ssm X_F$ and
$\QQ_{F'}\ssm X_{F'}$ can be arbitrarily large.
However, since the $\ell^1$ metric on $X$ is completely determined by its defining
hyperplanes, Proposition \ref{prop:cubehull} says that  $\QQ_F$ and $\QQ_{F'}$ are
metrically and combinatorially related,  depending only on $k$ and $X$---and not on $\diam(F)$.  In particular, one can delete boundedly many hyperplanes from the collections $\HH_F$ and $\HH_{F'}$ to generate a common model; see Subsubsection \ref{subsec:hyperplane deletion} for a discussion on hyperplane deletions.

\subsubsection*{Modeling hulls in hyperbolic spaces}

In coarse geometry, e.g., when $X$ is the Cayley graph of a finitely generated group, the notion of geodesic is often wobbly, and so our notion of hull needs to be more flexible.  Moreover, it will often be more fruitful to construct quasi-isometric \textbf{models} of hulls, which we should think of as nice combinatorial objects which coarsely encode the key geometric features of hulls into their combinatorial structure.  The main motivating examples here are hyperbolic spaces, where hulls are modeled by finite simplicial trees.

When $X$ is $\delta$-hyperbolic and $F \subset X$ with $\# F = k$, the right notion of $hull(F)$ is the \textbf{weak hull}, namely the set of all geodesics between points in $F$.  Notice then that the tripod-like $\delta$-slim-triangles condition generalizes to a tree-like slimness for $hull(F)$.  The following proposition is an easy consequence of Gromov's original arguments \cite{Gromov87} (see Figure \ref{fig:hyperhull}):

\begin{figure}
\includegraphics[width=0.75\textwidth]{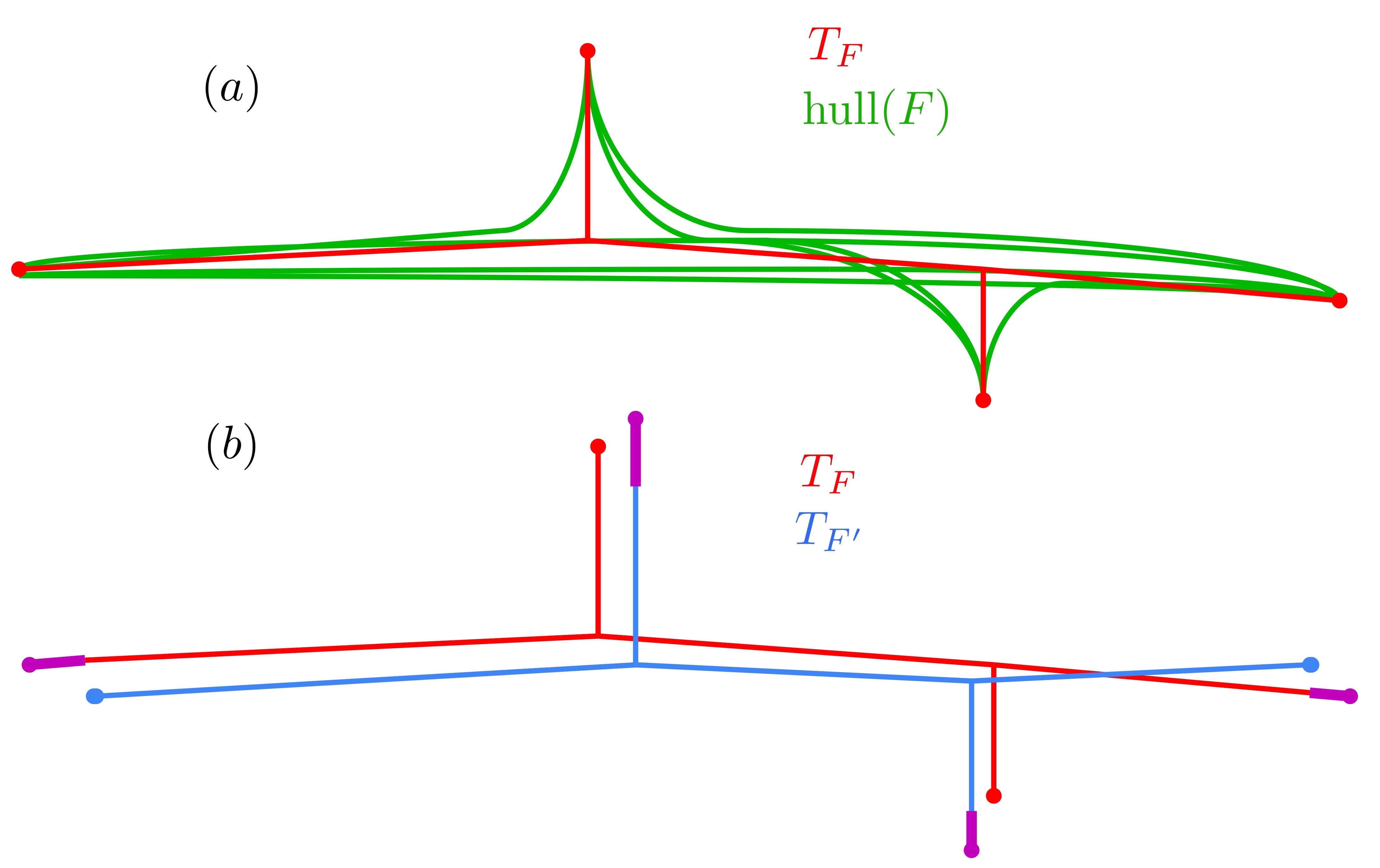}
\caption{Stability of hulls in a hyperbolic space: In (a), the modeling tree (red) for the coarse hull (green) of a finite set $F$.  In (b), the modeling trees $T_F$ and $T_{F'}$ for $\mathrm{hull}(F)$ and $\mathrm{hull}(F')$ are $(1,K)$-quasiisometric after deleting small subtrees (purple).}\label{fig:hyperhull}
\end{figure}

\begin{proposition}{prop:hyperhull}
For any $k \in \mathbb{N}$ and $\delta >0$, there exists $L = L(k, \delta)> 0$ such that the following holds:

Let $X$ be $\delta$-hyperbolic and $F \subset X$ with $\# F = k$.  Then there exists a simplicial tree $T_F$ and a $(1,L)$-quasi-isometric embedding $\phi_F: T_F \rightarrow X$ with $d_{Haus}(\phi_F(T_F), hull(F)) < L$.

Moreover, if $F' \subset X$ with $\# F' \leq k$ and $d_{Haus}(F,F') \leq 1$, then there exists a simplicial tree $T_0$ and a $(1,L)$-quasi-isometric embedding $\phi_0:T_0 \rightarrow X$ so that the diagram

\begin{equation}\label{hyp diagram}
  \begin{tikzcd}
    T_F \arrow[dr,"\phi_F"] \arrow[d,"h_F" left] &  \\
    T_0 \arrow[r,"\phi_0"] & X \\
    T_{F'}\arrow[ur,"\ \ \ \phi_{F'}" below] \arrow[u,"h_{F'}"] & \\
  \end{tikzcd}
  \end{equation}
commutes up to error at most $L$, where $h_F$ and $h_{F'}$ are quotient maps which collapse at most $L$ subtrees each of diameter at most $L$. 
\end{proposition}

Observe that Proposition \ref{prop:hyperhull} is a generalization of Proposition \ref{prop:treehull}, where $X$ is a tree and we can take the trees $T_F, T_{F'},$ and $T_0$ as before and the maps $\phi_F, \phi_{F'},$ and $\phi_0$ to be inclusions.  The main difference here is that a general hyperbolic space is stably \emph{locally} tree-like, and not a tree itself.  Hence the need for a model for the hulls.

Our main theorem is a common generalization of the stability properties in Proposition
\ref{prop:cubehull} and \ref{prop:hyperhull}.

\subsection{Stable cubical models for hulls in HHSes}

We will deal with \textbf{colorable} hierarchically hyperbolic spaces $(\XX,\mathfrak S)$, which means, for the reader familiar with HHSs, that there exists a decomposition of $\mathfrak S$ into finitely many families $\mathfrak S_i$, so that each $\mathfrak S_i$ is pairwise transverse. Colorable HHSs include mapping class groups and Teichm\"uller spaces of finite-type surfaces. 

In fact, colorability is a rather mild condition which is satisfied by all of the main motivating examples.  Its definition is inspired by Bestvina-Bromberg-Fujiwara's \cite{BBF} proof that curve graphs are finitely colorable; see Subsection \ref{subsec:HHSery} for a discussion.

Given a finite set of points $F \subset \XX$ in an HHS, the standard notions of a hull for
$F$ are very difficult to analyze.  For example, while little is known about geodesics in
the mapping class group, Rafi-Verberne \cite{RV18} proved that geodesics do not always
interact well with the curve graph machinery.  In Teichm\"uller space with the
Teichm\"uller metric, geodesics are unique, but it is an open question of Masur whether
the classical convex hull of a set of 3 points can be the whole space.  Moreover, it is a
result of Rafi that hulls of two points, i.e. geodesics, do not behave stably under
perturbation \cite[Theorem D]{Rafi:hypteich}.  These complications motivate a more
flexible definition of hull in this setting. 

The \textbf{hierarchical hull} of a finite set $F \subset \XX$, which we also denote $hull(F)$, was introduced in \cite{BKMM:qirigid} to study subspaces of the asymptotic cones of the mapping class group, on the way to proving that these groups are quasi-isometrically rigid.  In hyperbolic spaces and cube complexes, the hierarchical hull coincides with the notions of hull discussed above.  In the hierarchical setting, one instead has a notion of projecting $F$ to a family of hyperbolic spaces (e.g., curve graphs of subsurfaces).  In each of these hyperbolic spaces, one then takes the weak hull of the projection---which is coarsely a tree, as above---and the uses certain hierarchical consistency conditions \cite{BKMM:qirigid, HHS_II} to fashion these weak hulls in the various spaces into a hull in the ambient HHS which satisfies certain convexity properties \cite{BKMM:qirigid, HHS_II}.  In particular, the hierarchical hull of $F$ is \emph{hierarchically quasiconvex} \cite{HHS_II} and contains all of the hierarchy paths between points in $F$ \cite{BKMM:qirigid}.

In \cite{BHS:quasi}, Behrstock-Hagen-Sisto proved that the hierarchical hull of a finite set of points is quasi-isometric to a finite CAT(0) cube complex.  Their main observation was that the hierarchical consistency conditions are closely related to the consistency conditions on a wallspace from Sageev's construction of cubical complexes (Subsection \ref{subsec:walls}).  Their idea was to look at points on the modeling trees in the hyperbolic spaces which are unseen by the other projection data.  The preimages of these points under the projection maps turn out to behave like walls in the hull.  See Subsection \ref{subsec:proof discussion} for a sketch of these ideas, and Subsection \ref{subsec:constr} below for a full discussion.

Our main theorem stabilizes their construction, simultaneously generalizing the stability properties from Proposition \ref{prop:hyperhull} for any hyperbolic space and Proposition \ref{prop:cubehull} for cube complexes admitting an HHS structure.  The following is a more detailed version of Theorem \ref{thm:stable cubulation informal}:

\begin{theorem}{thm:stable cubulations mid-level}
Let $(\XX,\mathfrak S)$ be a colorable HHS. Then for each $k$ there exist $K, N$ with the following properties.  For any $F \subset \XX$ with $\#F = k$, there exists a finite CAT(0) cube complex $\QQ_F$ and a $(K,K)$-quasiisometric embedding $\Phi_F: \QQ_F \rightarrow \XX$ with $d_{\mathrm{Haus}}(\Phi_F(\QQ_F), hull(F)) \leq K$.

Moreover, if $F'\subseteq \XX$ is another subset with $\#F \leq k$ and $d_{Haus}(F,F') \le 1$,
there is a finite CAT(0) cube complex $\mathcal Y_0$ and a
$(K,K)$--quasi-isometric embedding $\Phi_0:\QQ_0 \rightarrow \XX$ such that the diagram
  \begin{equation}
  \begin{tikzcd}
    \QQ_F \arrow[dr,"\Phi_F"] \arrow[d,"\eta_F" left] &  \\
    \QQ_0 \arrow[r,"\Phi_0"] & \XX \\
    \QQ_{F'}\arrow[ur,"\ \ \ \Phi_{F'}" below] \arrow[u,"\eta_{F'}"] & \\
  \end{tikzcd}
  \end{equation}
commutes up to error at most $K$, where $\eta_F$ and $\eta_{F'}$ are hyperplane
deletion maps which delete at most $N$ hyperplanes. 
\end{theorem} 

See Theorem \ref{thm:stable_cubulations} below for the full version of the theorem, the details of which are necessary for our applications.

\subsection{Sketch of proofs}\label{subsec:proof discussion}

The proof of Theorem \ref{thm:stable_cubulations}, of which Theorem \ref{thm:stable cubulation informal} is an informal version, is contained in Section \ref{sec:stable cubulations} and depends crucially on our work in Section \ref{sec:stable trees}.  Theorems \ref{thm:bicombing main} and \ref{thm:barycenter main} are a consequence of Theorem \ref{thm:stable cubulation informal} and our work in Section \ref{sec:normal paths}.  We now explain the various parts and how they fit together.

In what follows, we will keep our discussion within the context of mapping class groups and its hierarchy of curve graphs \cite{MM_I, MM_II}, though we work in the more general context of HHSes.

Let $F \subset \MCG(\Sigma)$ be a finite subset and consider essential subsurfaces $V \subset \Sigma$ which are not 3-holed spheres. 
Roughly, the \emph{hierarchical hull} of $F$, $hull(F)$, is the set of points of $\MCG(\Sigma)$ whose subsurface projections in each curve graph $\CC(V)$ lie close to the weak hull of the subsurface projection $\pi_V(F)$ of $F$.

In the cubulation construction of \cite{BHS:quasi}, the authors build a wallspace for $hull(F)$.

To do this, they first consider the collection $\UU_F$ of \emph{relevant} subsurfaces $V \subset \Sigma$ for which $\diam_V \pi_V(F) > K$ for some fixed threshold $K>0$.  In each of these subsurfaces, they take a tree $T^V_F$ which coarsely models the hull of $\pi_V(F)$ in $\CC(V)$, as discussed in Subsection \ref{subsec:coarse hulls}.  For each such $V \in \UU_F$,  they then consider the collection of relative projections $\rho^W_V$ of $W \in \UU_F$ to $\CC(V)$, which correspond to the projection of $\partial W$ to $\CC(V)$ and thus are nonempty if $V$ is neither disjoint from nor contained in $W$.  The Bounded Geodesic Image Theorem \cite{MM_II} and certain consistency properties of projections \cite{Beh:thesis, BKMM:qirigid} imply that each $\rho^W_V$ for such $W$ lies uniformly close to the tree $T^V_F$.

They then consider, roughly, the complement $P^V_F$ in $T^V_F$ of a regular neighborhood of these projections, which consists of a number of subtrees of $T^V_F$ which are ``unseen'' by the other subsurfaces in $\UU_F$ which interact with $V$.  Any point in $T^V_F-P^V_F$ cuts $T^V_F$ into two subtrees. The partitions of $hull(F)$ that define the wallspace on $hull(F)$ come from these subdivision points in the $T^V_F$, namely one consider the subspaces of $hull(F)$ whose subsurface projections to $\CC(V)$ lie close to either of the subtrees.

While this construction is useful for studying top-dimensional quasiflats, it is unstable under perturbation of $F$, in that given some other $F'$ with $d_{Haus}(F,F') \leq 1$, then the cubical models $\QQ_F$ and $\QQ_{F'}$ might differ by a number of hyperplanes on the order of $\diam(F)$, which is not bounded. 

The proof of Theorem \ref{thm:stable cubulation informal} involves stabilizing this process in a number of places.  The first step is to robustly stabilize the collection of relevant subsurfaces $\UU_F$ (Proposition \ref{prop:bounded involved}), e.g., so that $|\UU_F \symdiff \UU_{F'}|$ is bounded in terms of the topology of the surface $S$.  We do this by applying work of Bestvina-Bromberg-Fujiwara-Sisto \cite{BBFS}, which allows us to stabilize subsurface projections (Theorem \ref{thm:BBFS}), and then use standard projection complex type arguments.

In Section \ref{sec:stable trees}, we stabilize the modeling trees $T^V_F$ for each $V \in \UU_F$.  Unlike before, it will not do to simply take any Gromov modeling tree, since unboundedly many pieces of it might change in the transition from $F$ to $F'$ when we cut it up using the relative projection data (the $P^V_F$ above).  Instead, we use the newly stabilized relative projection data to build a new stable tree.  We do this by taking a regular neighborhood of the relative projections, which then group into connected components we call \emph{clusters}.  As before, these clusters lie close to any Gromov modeling tree, but we cannot use these trees.  Instead, we define a separation graph for these clusters (Definition \ref{defn:separation graph}), and then prove that the combinatorics of this graph encode how these domain clusters are arranged on \emph{any} Gromov modeling tree.  We then construct our stable tree by connecting clusters both internally and externally via minimal spanning networks in $\CC(V)$.  The stability of the cluster data then is converted into stability of the tree construction in Theorem \ref{thm:stable tree}, which, in particular, says that the set of long edges of two related trees are in bijection and within bounded Hausdorff distance, with most long edges exactly the same.  See Figures \ref{stabletree} and \ref{fig:edgecomponent} below.

In Section \ref{sec:stable cubulations}, we then plug these stable trees into the cubulation machine from \cite{BHS:quasi}.  We must be mindful of how subdivision points change when transitioning from $F$ to $F'$.  In particular, we construct a common refinement of the sets of subdvision points for our two sets $F$ and $F'$ (Proposition \ref{prop:almost bijection of p}), with the delicate nature of this process necessitating the intricacies in the statement and proof of the stable tree theorem (Theorem \ref{thm:stable tree}).  With this in hand, we prove that this common refinement induces an isomorphism between the resulting cubical models for the hulls of both sets (Proposition \ref{prop:isomorphism of cc}); see Figure \ref{fig:trees to cubes}.  This isomorphism depends on a careful hierarchical analysis of when two halfspaces corresponding to two subdivision points intersect (Lemma \ref{lem:halfspaces_intersect}).  The full version of the stable cubulation theorem is achieved in Theorem \ref{thm:stable_cubulations}, which says that the two modeling cube complexes $\QQ_F$ and $\QQ_{F'}$ become isomorphic when we delete a bounded number of hyperplanes from each, with the bound depending only on $|F|$ and $|F'|$.

In Section \ref{sec:normal paths}, we adapt the normal path construction of Niblo-Reeves \cite{NibloReeves} and analyze how it changes under hyperplane deletion.  In particular, for any finite CAT(0) cube complex $\QQ$, we develop a sequence of contractions which take the extremal vertices of $\QQ$ (i.e., its corners) into a ``barycentric'' cube at the ``center'' of $\QQ$, and we prove that this contraction sequence is only boundedly perturbed by hyperplane deletions (Theorem \ref{thm:stable contraction}).

Stability of the cubical model and the contraction sequence easily give the barycenter theorem (Theorem \ref{thm:barycenter main}).  In the context of a bicombing (Theorem \ref{thm:bicombing main}) when $F = \{x,y\}$, we take the bicombing path from $x$ to $y$ to be the image in $\MCG(\Sigma)$ of the path obtained by following the contraction sequence of $x$ to the barycentric cube, and then traversing the contraction sequence from the barycentric cube to $y$ in reverse order.  Once again, stability of the contraction sequence and the cubical models implies that these are uniform quasi-geodesics which fellow-travel in a parametrized fashion; see Figure \ref{fig:bicombing example}.  Both Theorem \ref{thm:barycenter main} and Theorem \ref{thm:bicombing main} are proved in Section \ref{sec:main proofs}.

We note that the current paper would be simplified in a number of places if we were only
aiming at developing a bicombing, which would involve only analyzing the hull of two
points. 
In particular, all our stable trees from Section \ref{sec:stable trees} would be intervals, the lack of branching in the trees simplifies the discussion of partition points in Section \ref{sec:stable cubulations}, and the case of two points simplifies hyperplane separation considerations in Section \ref{sec:normal paths}.  The reader may keep in mind this simpler case first before considering the complications caused by trees.

\subsection{Outline}

In Section \ref{sec:background} we collect some background material.

Section \ref{sec:stable trees} takes place entirely in a fixed hyperbolic space, using methods from coarse hyperbolic geometry but with HHS ends in mind. The main result there is Theorem \ref{thm:stable tree}, and no other result from that section will be used elsewhere.

In Section \ref{sec:stable cubulations}, we prove the precise version of Theorem \ref{thm:stable cubulation informal}, which is Theorem \ref{thm:stable_cubulations}. Again, no other statement from this section will be used elsewhere. In this section, we use the combinatorial geometry of HHSes.

Section \ref{sec:normal paths} uses the tools of cubical geometry, and it is independent from the previous sections. Its main result is Theorem \ref{thm:stable contraction}, which once again is the only result from here needed in the rest of the paper.

Finally, in Section \ref{sec:main proofs} we put all the pieces together, and we prove Theorem \ref{thm:bicombing main} and Theorem \ref{thm:barycenter main}.

\subsection*{Acknowledgements}

Durham would like to thank Daniel Groves for many fruitful and insightful discussions
about bicombing the mapping class group.  Durham was partially supported by an AMS Simons
travel grant and NSF grant DMS-1906487. Minsky was partially supported by NSF grants
DMS-1610827 and DMS-2005328. Sisto  was partially supported by the Swiss
National  Science Foundation (grant \#182186).

\section{Background} \label{sec:background}

In this section, we will collect and record the various facts about cube complexes and hierarchically hyperbolic spaces that we need.

\subsection{CAT(0) cube complexes} \label{subsec:cube complexes}

We will briefly discuss some basic aspects of CAT(0) cubical geometry.  We direct the reader to Sageev's lecture notes \cite{Sageev:pcmi} for details.
\par\smallskip

A \textbf{ cube complex} is a simplicial complex $X$ obtained from a disjoint collection of Euclidean cubes which are glued along their faces by a collection of  Euclidean isometries.  A cube complex is \textbf{non-positively curved} (NPC)  if its vertex links are simplicial flag complexes.  An NPC cube complex is \textbf{CAT(0)} if it is a 1-connected NPC complex.
\par\smallskip

A \textbf{midcube} of an $n$-cube $C \subset X$ is an $(n-1)$-dimensional cube $H' \subset C$ running through the barycenter of $C$ and parallel to one of the faces of $C$.  A \textbf{hyperplane} $H \subset X$ is the maximal codimension-1 subspace obtained by extending $H'$ along the midcubes of any cube $C'$ which meets $C$ along a face intersecting $H'$, and then further extending along cubes adjacent to such $C'$ in the same fashion.  The \textbf{carrier} of $H$ is the union of all of the cubes in $X$ whose intersection with $H$ is a midcube, and it is naturally isomorphic to $H\times [0,1]$.

Equivalently, there is a natural equivalence relation on the set of edges in the 1-skeleton of $X$ generated by relating two edges if they are opposite edges of some square in $X$.  Any hyperplane can be obtained as the collection of midcubes which intersect the edges in a given equivalence class.
\par\smallskip

In this paper, we will be considering \textbf{finite} cube complexes, namely those with finitely many cubes.  
\subsubsection*{Metrics on cube complexes}

There are many interesting metrics one can put on a CAT(0) cube complex $X$. We will be interested in both:

\begin{itemize}

\item the $l^1$ or \textbf{combinatorial} metric, $d_1$,  which is generated by the $l^1$ norm in each cube of $X$, and can be equivalently defined on the 1-skeleton $X^{(1)}$ as the path metric thereon;
\item the \textbf{cubical sup metric}, $d_{\infty}$, which is the metric $d_{\infty}$ generated by the $l^{\infty}$ or sup-norm in each cube in $X$.
\end{itemize}

The following is an easy consequence of the observation that, given $n$, the $l^1$ and $l^{\infty}$ norm on an $n$-cube are bilipschitz equivalent.

\begin{lemma}{lem:different metrics}
For any $n> 0$, there exists $K=K(n)>0$ so that if $X$ is an $n$-dimensional cube complex, then the identity $id: (X, d_1) \rightarrow (X, d_{\infty})$ is a $(K,K)$-quasiisometry.
\end{lemma}

The differences between these metrics will come up in Sections \ref{sec:normal paths} and \ref{sec:main proofs}.

\subsubsection{Wallspaces and Sageev's construction} \label{subsec:walls}

In Section \ref{sec:stable cubulations}, we will adopt the perspective of obtaining cube complexes as duals to wallspaces.  Wallspaces were first defined by Haglund-Paulin \cite{HP98}; see Hruska-Wise \cite{HW14} for a broader discussion.

Let $Y$ be a nonempty set.  A \textbf{wall} in $Y$ is a pair of subsets $W = \{\OL W,\OR W\}$ where $Y = \OL W \sqcup \OR W$. In this case, $\OL W$ and $\OR W$ are called \textbf{halfspaces}.

\par\smallskip Two points $x,y \in Y$ are \textbf{separated} by a wall $W$ if $x$ is contained in a different halfspace from $y$.

\par\smallskip A \textbf{wallspace} is a set $Y$ with a collection of walls $\WW$ on $Y$ so that the number of walls separating any pair of points is finite.

An \textbf{orientation} on a wallspace $(Y, \WW)$ is an assignment $\sigma$ so that for each $W \in \WW$, we have  $\sigma(W) \in \{\OL W, \OR W\}$.  The orientation $\sigma$ is called \textbf{coherent} if for all $W, W' \in \WW$, we have $\sigma(W) \cap \sigma(W') \neq \emptyset$.  We call $\sigma$ \textbf{canonical} if there exists $x \in X$ so that $x \in \sigma(W)$ for all but finitely many $W \in \WW$.
\par\smallskip

Given a wallspace $(Y, \WW)$, we can consider the cube complex $X(Y, \WW)$ constructed as follows. The 0-cubes of $X(Y, \WW)$ are coherent, canonical orientations of $(Y, \WW)$. Two 0-simplices are connected by a 1-cube if, seen as orientations, they differ on only one wall. Finally, all subcomplexes of the 1-skeleton isomorphic to the 1-skeleton of an $n$-cube cube get filled by an $n$-cube.

Work of Chatterji-Niblo \cite{CN05}, Nica \cite{Nic04}, and Chepoi \cite{Che00}---building off of work of Sageev \cite{Sageev:cubulation}--- gives that $X(Y, \WW)$ is a CAT(0) cube complex.  We call $X(Y, \WW)$ the \textbf{dual cube complex} associated to the wallspace $(Y, \WW)$.

\subsubsection{Hyperplane deletions} \label{subsec:hyperplane deletion}

In Section \ref{sec:normal paths}, we will be interested in understanding how cube complexes change under deletions of hyperplanes, so we will use the alternative perspective of obtaining cube complexes from sets of hyperplanes.  We briefly explain how this works.
\par\smallskip

Let $X$ be a CAT(0) cube complex and $\HH_X$ its (finite) set of hyperplanes.  Then we can identify each hyperplane $H \in \HH_X$ with the two halfspaces into which it cuts $X$.  As such, we can and will think of $(X, \HH_X)$ as a wallspace, and one can show that $X$ is the dual cube complex associated to $(X, \HH_X)$.
\par\smallskip

Given any subset $\HH \subset \HH_X$ of hyperplanes in a cube complex $X$, there is a natural cube complex $X(\HH)$ defined as the dual cube complex associated to the wallspace defined by $\HH$ in $X$.  In particular, each point in $X(\HH)$ is a choice of coherent, canonical orientations of the half-spaces defined by $\HH$, and it follows that $X(\HH)$ embeds into $X(\HH_X) = X$.

\par\smallskip
With this notation, we can now define hyperplane deletions:

\begin{definition}{defn:hyperplane deletions}
Let $X$ be a CAT(0) cube complex obtained with hyperplanes $\HH_X$.  For a finite collection of hyperplanes $\GG \subset \HH_X$, the \textbf{\textbf{hyperplane deletion map}} for $H$ is the map 
$$\res_{\HH_X\ssm\GG}: X \rightarrow X(\HH_X \ssm \GG)$$
 obtained by restriction of orientations,
where $X(\HH_X \ssm \GG)$ is the dual cube complex associated to the wallspace $(X, \HH_X
 \ssm \GG)$.
\end{definition}

Equivalently, the map $\res_{\HH_X\ssm\GG}$ is the quotient map which collapses the $[0,1]$ factor of each of the carriers of the hyperplanes in $\GG$ (recall that the carrier of the hyperplane $H$ is naturally isomorphic to $H\times[0,1]$).
\par\smallskip

\par\smallskip

We also record the following fact, which indicates that
the isomorphism type of the cube complex coming from a wallspace is determined by the
intersection pattern of halfspaces. The proof is elementary.

\begin{lemma}{lem:halfspaces_bijection}
 Let $\mathcal W$, $\mathcal W'$ be wallspaces, and let $\iota:\mathcal H_{\mathcal
   W}\to\mathcal H_{\mathcal W'}$ be a bijection of their halfspaces, which preserves
 complements and disjointness. 

 Denote by $j$ the induced map on walls $\{H,H^c\}\mapsto \{\iota(H),\iota(H)^c\}$ and
by $h(x) = \iota \circ x\circ j^{-1}$ the induced map on orientations. Then $h$, viewed as a map
on $0$-cubes, induces an isomorphism 
$h:\mathcal Y_{\mathcal W}\to \mathcal Y_{\mathcal W'}$
between the corresponding CAT(0) cube complexes.

\end{lemma}

\subsection{HHS axioms} \label{subsec:HHSery}
We recall from~\cite{HHS_II} the definition of a hierarchically hyperbolic space.

\begin{definition}{defn:HHS} {\rm[HHS]}
The $q$--quasigeodesic space  $(\cuco X,\dist_{\cuco X})$ is a \textbf{\em hierarchically hyperbolic space} if there exists $\delta\geq0$, an index set $\mathfrak S$, and a set $\{\mathcal C W:W\in\mathfrak S\}$ of $\delta$--hyperbolic spaces $(\mathcal C U,\dist_U)$,  such that the following conditions are satisfied:\begin{enumerate}
\item\textbf{(Projections.)}\label{item:dfs_curve_complexes} There is
a set $\{\pi_W: \cuco X\rightarrow2^{\mathcal C W}\mid W\in\mathfrak S\}$
of \textbf{\em projections} sending points in $\cuco X$ to sets of diameter
bounded by some $\xi\geq0$ in the various $\mathcal C W\in\mathfrak S$.
Moreover, there exists $K$ so that for all $W\in\mathfrak S$, the coarse map $\pi_W$ is $(K,K)$--coarsely
Lipschitz and $\pi_W(\cuco X)$ is $K$--quasiconvex in $\mathcal C W$.

 \item \textbf{(Nesting.)} \label{item:dfs_nesting} $\mathfrak S$ is
 equipped with a partial order $\nest$, and either $\mathfrak
 S=\emptyset$ or $\mathfrak S$ contains a unique $\nest$--maximal
 element; when $V\nest W$, we say $V$ is \textbf{\em nested} in $W$.  (We
 emphasize that $W\nest W$ for all $W\in\mathfrak S$.)  For each
 $W\in\mathfrak S$, we denote by $\mathfrak S_W$ the set of
 $V\in\mathfrak S$ such that $V\nest W$.  Moreover, for all $V,W\in\mathfrak S$
 with $V\propnest W$ there is a specified subset
 $\rho^V_W\subset\mathcal C W$ with $\diam_{\mathcal C W}(\rho^V_W)\leq\xi$.
 There is also a \textbf{\em projection} $\rho^W_V: \mathcal C
 W\rightarrow 2^{\mathcal C V}$.  (The similarity in 
notation is justified by viewing $\rho^V_W$ as a coarsely constant map $\mathcal C
 V\rightarrow 2^{\mathcal C W}$.)
 
 \item \textbf{(Orthogonality.)} 
 \label{item:dfs_orthogonal} $\mathfrak S$ has a symmetric and
 anti-reflexive relation called \textbf{\em orthogonality}: we write $V\orth
 W$ when $V,W$ are orthogonal.  Also, whenever $V\nest W$ and $W\orth
 U$, we require that $V\orth U$.  We require that for each
 $T\in\mathfrak S$ and each $U\in\mathfrak S_T$ for which
 $\{V\in\mathfrak S_T\mid V\orth U\}\neq\emptyset$, there exists $W\in
 \mathfrak S_T-\{T\}$, so that whenever $V\orth U$ and $V\nest T$, we
 have $V\nest W$.  Finally, if $V\orth W$, then $V,W$ are not
 $\nest$--comparable.
 
 \item \textbf{(Transversality and consistency.)}
 \label{item:dfs_transversal} If $V,W\in\mathfrak S$ are not
 orthogonal and neither is nested in the other, then we say $V,W$ are
 \textbf{\em transverse}, denoted $V\transverse W$.  There exists
 $\kappa_0\geq 0$ such that if $V\transverse W$, then there are
  sets $\rho^V_W\subseteq\mathcal C W$ and
 $\rho^W_V\subseteq\mathcal C V$ each of diameter at most $\xi$ and 
 satisfying: 
 \begin{equation}\label{eq:transverse consistency}
   \min\left\{\dist_{
 W}(\pi_W(x),\rho^V_W),\dist_{
 V}(\pi_V(x),\rho^W_V)\right\}\leq\kappa_0
 \end{equation}
 for all $x\in\cuco X$.

 For $V,W\in\mathfrak S$ satisfying $V\nest W$ and for all
 $x\in\cuco X$, we have: 
\begin{equation}\label{eq:nested consistency}
\min\left\{\dist_{ W}(\pi_W(x),\rho^V_W),\diam_{\mathcal C
 V}(\pi_V(x)\cup\rho^W_V(\pi_W(x)))\right\}\leq\kappa_0. 
\end{equation}

 The preceding two inequalities are the \textbf{\em consistency inequalities} for points in $\cuco X$.
 
 Finally, if $U\nest V$, then $\dist_W(\rho^U_W,\rho^V_W)\leq\kappa_0$ whenever $W\in\mathfrak S$ satisfies either $V\propnest W$ or $V\transverse W$ and $W\not\orth U$.
 
 \item \textbf{(Finite complexity.)} \label{item:dfs_complexity} There exists $n\geq0$, the \textbf{\em complexity} of $\cuco X$ (with respect to $\mathfrak S$), so that any set of pairwise--$\nest$--comparable elements has cardinality at most $n$.
  
 \item \textbf{(Large links.)} \label{item:dfs_large_link_lemma} There
exist $\lambda\geq1$ and $E\geq\max\{\xi,\kappa_0\}$ such that the following holds.
Let $W\in\mathfrak S$ and let $x,x'\in\cuco X$.  Let
$N=\lambda\dist_{_W}(\pi_W(x),\pi_W(x'))+\lambda$.  Then there exists $\{T_i\}_{i=1,\dots,\lfloor
N\rfloor}\subseteq\mathfrak S_W-\{W\}$ such that for all $T\in\mathfrak
S_W-\{W\}$, either $T\in\mathfrak S_{T_i}$ for some $i$, or $\dist_{
T}(\pi_T(x),\pi_T(x'))<E$.  Also, $\dist_{
W}(\pi_W(x),\rho^{T_i}_W)\leq N$ for each $i$.

 \item \textbf{(Bounded geodesic image.)}
 \label{item:dfs:bounded_geodesic_image} There exists $\kappa_0>0$ such that 
 for all $W\in\mathfrak S$,
 all $V\in\mathfrak S_W-\{W\}$, and all geodesics $\gamma$ of
 $\mathcal C W$, either $\diam_{\mathcal C V}(\rho^W_V(\gamma))\leq \kappa_0$ or
 $\gamma\cap N_{\kappa_0}(\rho^V_W)\neq\emptyset$.
 
 \item \textbf{(Partial Realization.)} \label{item:dfs_partial_realization} There exists a constant $\alpha$ with the following property. Let $\{V_j\}$ be a family of pairwise orthogonal elements of $\mathfrak S$, and let $p_j\in \pi_{V_j}(\cuco X)\subseteq \mathcal C V_j$. Then there exists $x\in \cuco X$ so that:
 \begin{itemize}
 \item $\dist_{V_j}(x,p_j)\leq \alpha$ for all $j$,
 \item for each $j$ and 
 each $V\in\mathfrak S$ with $V_j\nest V$, we have 
 $\dist_{V}(x,\rho^{V_j}_V)\leq\alpha$, and
 \item if $W\transverse V_j$ for some $j$, then $\dist_W(x,\rho^{V_j}_W)\leq\alpha$.
 \end{itemize}

\item\textbf{(Uniqueness.)} For each $\kappa\geq 0$, there exists
$\theta_u=\theta_u(\kappa)$ such that if $x,y\in\cuco X$ and
$\dist_{\cuco X}(x,y)\geq\theta_u$, then there exists $V\in\mathfrak S$ such
that $\dist_V(x,y)\geq \kappa$.\label{item:dfs_uniqueness}
\end{enumerate}
 
We often refer to $\mathfrak S$, together with the nesting
and orthogonality relations, and the projections as a \textbf{\em hierarchically hyperbolic structure} for the space $\cuco
X$. 
\end{definition}

Where it will not cause confusion, given $U\in\mathfrak S$, we will often suppress the projection
map $\pi_U$ when writing distances in $\mathcal C U$, i.e., given $x,y\in\cuco X$ and
$p\in\mathcal C U$  we write
$\dist_U(x,y)$ for $\diam_{\mathcal C U}(\pi_U(x)\cup\pi_U(y))$ and $\dist_U(x,p)$ for
$\diam_{\mathcal C U}(\pi_U(x)\cup\{p\})$.
Given $A\subset \cuco X$ and $U\in\mathfrak S$ 
we let $\pi_{U}(A)$ denote $\displaystyle \bigcup_{a\in A}\pi_{U}(a)$.

There is a natural notion of automorphism of an HHS, which permutes the hyperbolic factor
spaces and preserves all the structure -- see \cite{HHS_II, DHS} for details and
analysis. We let $\mathrm{Aut}({\cuco X}, {\mathfrak S})$ denote the group of HHS automorphisms of $({\cuco X}, {\mathfrak S})$.

We say that a group $G$ is a \textbf{\em hierarchically hyperbolic group} if it acts properly and coboundedly by HHS automorphisms on some HHS $(\XX,\mathfrak S)$.

\subsubsection{Some useful facts}

We now recall results from \cite{HHS_II} that will be useful later on.

\begin{definition}{defn:consistent_tuple}
Let $\kappa\geq0$ and let $\tup b\in\prod_{U\in\mathfrak S}2^{\mathcal C U}$ be a tuple such that for each $U\in\mathfrak S$, 
the $U$--coordinate  $b_U$ has diameter $\leq\kappa$.  Then $\tup b$ is \textbf{\em $\kappa$--consistent} if for all 
$V,W\in\mathfrak S$, we have $$\min\{\dist_V(b_V,\rho^W_V),\dist_W(b_W,\rho^V_W)\}\leq\kappa$$ whenever $V\transverse W$ and 
$$\min\{\dist_W(x,\rho^V_W),\diam_V(b_V\cup\rho^W_V)\}\leq\kappa$$ whenever $V\propnest W$.
\end{definition}

The following is 
\cite[Theorem~4.5]{HHS_II}:

\begin{theorem}{thm:distance_formula}{\rm [Distance Formula]}
Let $(\cuco X,\mathfrak S)$ be a hierarchically hyperbolic space.  Then
there exists $s_0$ such that for all $s\geq s_0$, there exist $C,K$ so
that for all $x,y\in\cuco X$,
$$\dist(x,y)\asymp_{K,C}\sum_{U\in\mathfrak
S}\ignore{\dist_U(x,y)}{s}.$$
\end{theorem}

\noindent (The notation $\ignore{A}{B}$ denotes the quantity which is $A$ if $A\geq B$ and $0$ otherwise.)

We recall the notion of a {\em hierarchical hull}, which 
originates in \cite{BKMM:qirigid} for the setting of mapping class groups, and extends to the HHS
setting in \cite{HHS_II}.
Given a constant $\theta$, for any $F\subset \cuco X$ we define 
\begin{equation}\label{eq:hull-def}
H_\theta(F) = \{ x\in \cuco X  : \forall V\in\mathfrak S \ \pi_V(x) \in
\NN_\theta(hull(\pi_V(F)))  \},
\end{equation}
where $hull(A)$ denotes the union of all geodesics connecting points of $A$.
In words, $H_\theta$ is the set of points whose projections in every hyperbolic factor
space land in a specified neighborhood of the hull of the image of $F$. That this
is a sufficiently non-vacuous is indicated by the following result:

\begin{theorem}{thm:HHS-hull}{\rm [Hull]}
  Let $(\cuco X,\mathfrak S)$ be a hierarchically hyperbolic space.
Given $k$ there exists $\theta,\kappa$ such that, if $F\subset\cuco X$ is a set of
cardinality $k$ then for every $V\in\mathfrak S$ the image $\pi_V(H_\theta(F))$ is
$\kappa$-dense in the hull of $\pi_V(F)$. 
\end{theorem}

\subsection{Refined projections and stable subsurface collections}
\label{subsec:refined projections}

We will be working in a broad but restricted class of HHSes:

\begin{definition}{defn:colorable}
Let $(\XX,\mathfrak S)$ be an HHS and let $G < \mathrm{Aut}(\mathfrak S)$.  We say that
$(\XX, \mathfrak S)$ is $G$-\textbf{\em colorable} if there exists a decomposition of
$\mathfrak S$ into finitely many families $\mathfrak S_i$, so that each $\mathfrak S_i$ is
pairwise-$\pitchfork$ and $G$ acts on $\{\mathfrak S_i\}_i$ by permutations. We say that
$(\XX,\mathfrak S)$ is colorable if it is $\mathrm{Aut}(\mathfrak S)$-colorable. 
\end{definition}

The notion of colorability is inspired from \cite{BBF}, who proved that curve graph $\mathcal{C}(\Sigma)$ is finitely colorable for $\Sigma$ of finite type, thus making $\MCG(\Sigma)$ and $\Teich(\Sigma)$ finitely $\MCG(\Sigma)$-colorable HHSes; we sometimes refer to the $\mathfrak S_i$ as \textbf{BBF families}.

For $A,B\subseteq \mathcal C(Y)$, we denote $d_Y(A,B)=diam_{\mathcal C(Y)}(A\cup B)$.
  
\begin{theorem}{thm:BBFS}\cite{BBFS}
Let $(\XX, \mathfrak S)$ be a $G$-colorable HHS for $G < \mathrm{Aut}(\mathfrak S)$ with
standard projections $\hpi_-, \hrho^-_-$.  There exists $\theta>0$ and refined projections
$\pi_-, \rho^-_-$ with the same domains and ranges, respectively, and such that:  
\begin{enumerate}
 \item If $X,Y$ lie in different $\mathfrak S_j$, and $\hrho^X_Y$ is defined, then $\rho^X_Y=\hrho^X_Y$.
 \item If $X,Y\in\mathfrak S_j$ are distinct, then the Hausdorff distance between $\rho^X_Y$ and $\hrho^X_Y$ is at most $\theta$.
 \item If $x\in\XX$ and $Y\in\mathfrak S$, then the Hausdorff distance between $\pi_Y(x)$ and $\hpi_Y(x)$ is at most $\theta$.
 \item If $X,Y,Z \in \mathfrak S_j$ for some $j$ are pairwise distinct and $d_Y(\rho^X_Y, \rho^Z_Y) > \theta$, then $\rho^X_Z=\rho^Y_Z$.
 \item Let $x\in\XX$, and let $Y,Z\in\mathfrak S_j$ for some $j$ be pairwise distinct. If $d_Y(\pi_Y(x), \rho^Z_Y)>\theta$ then $\pi_Z(x)=\rho^Y_Z$.\label{item:strict_Behrstock_X}
\end{enumerate}

Moreover, $(\XX, \mathfrak S)$ equipped with $\pi_-, \rho^-_-$ is an HHS,  $G <
\mathrm{Aut}(\mathfrak S)$, and it is $G$-colorable.

\end{theorem}

\begin{proof}
The idea is to apply the construction from \cite{BBFS} to the standard projections  $\hpi_-, \hrho^-_-$  and distances $\hd_-$ for the sets $\mathfrak S_i \cup \XX$ for each $i$, where we think of $\XX$ as a collection of single point spaces $x=\{x\}$ for each $x \in \XX$.

Given a point $x \in \XX$, we define projections $\rho^-_x$ from domains in $\mathfrak S_i$ and $\XX$ to $x$ as the constant map $\rho^-_x \equiv x$.  It is easily checked that $\mathfrak S_i \cup \XX$, once equipped with the original projections $\hpi_-, \hrho^-_-$ and these additional projections, satisfies the projection axioms from \cite{BBF}.  The existence of projections and distances  $\pi_-, \rho^-_-,d_-$ and that all properties hold for them is then an immediate consequence of \cite[Theorem 4.1]{BBFS}.

Finally, the fact that $(\XX, \mathfrak S)$ equipped with these projections is an HHS follows from the fact that the new projections are bounded distance away from the old ones, by items (1)-(2)-(3).
\end{proof}

\begin{definition}{defn:stable HHS}
We say that a $G$-colorable HHS $(\XX, \mathfrak S)$ with $G < \mathrm{Aut}(\mathfrak S)$ has \textbf{\em stable projections} if it is equipped with the projections provided by Theorem \ref{thm:BBFS}. 
\end{definition}

For the rest of this section, fix a $G$-colorable HHS $(\XX, \mathfrak S)$ with $G <
\mathrm{Aut}(\mathfrak S)$ and with stable projections.  In particular, we assume that the
standard projections for $(\XX,\mathfrak S)$ satisfy the stability properties in Theorem
\ref{thm:BBFS}.

As usual, $d_Y(x_1,x_2)$ denotes $diam_{\mathcal C(Y)}(A_1\cup A_2)$, where
\begin{itemize}
 \item $A_i=\pi_Y(x_i)$ if $x_i\in \XX$,
 \item $A_i=\rho^{x_i}_Y$ if $x_i\in\mathfrak S$ and either $x_i\propnest Y$ or $x_i\transverse Y$.
\end{itemize}

For any pair of points $x,y \in \XX$ and constant $K>0$, we let
$\rel_K(x,y) \subset \mathfrak{S}$ denote the collection of $Y \in
\mathfrak S$ such that $d_Y(x,y)>K$; we also set $\rel^i_K(x,y) = \rel_K(x,y) \cap \mathfrak S_i$.

Let $\theta$ satisfy Theorem \ref{thm:BBFS}(5).  Following \cite{BKMM:qirigid,CLM12, HHS_II}, recall that, for any $K>10\theta$, $\rel^i_K(x,y)$ is a partially ordered set with order $\prec$ so that $X\prec Y$ whenever one of the following equivalent conditions hold:
\begin{multicols}{2}
\begin{itemize}
 \item $d_Y(x,\rho^X_Y)\leq \theta$,
 \item $d_X(\rho^Y_X,y)\leq \theta$,
 \item $d_Y(\rho^X_Y,y)\geq K-\theta$,
 \item $d_X(x,\rho^Y_X)\geq K-\theta$.
\end{itemize}
\end{multicols}

When restricted to $\rel^i_K(x,y)$, the relation $\prec$ becomes a total order.

For a finite set $F \subset
\XX$, we define

$$\rel_K(F) = \bigcup_{x, y\in F} \rel_{K}(x,y) \indent \text{ and } \indent \rel^i_K(F) = \rel_K(F) \cap \mathfrak S_i.$$

The following stability lemma follows directly from the construction
in \cite{BBFS}.

\begin{lemma}{lem:stable interval}
There exists $K\gg 2\theta$ such that whenever $x,y,y'\in \XX$ satisfy
$d{_\XX}(y,y')\le 1$ the following holds. 
For each $i$, we have 
$$\left|\rel^i_K(x,y) \bigtriangleup \rel^i_K(x,y')\right|\leq 2.$$
\end{lemma}

\begin{proof}
By contradiction, suppose we have distinct elements $Y_0,Y_1,Y_2\in \rel^i_K(x,y)- \rel^i_K(x,y')$, with $Y_0\prec Y_1\prec Y_2$.
If $K>10\theta$, applying the definition of $\prec$ and Theorem \ref{thm:BBFS}(\ref{item:strict_Behrstock_X}), we see that $\pi_{Y_1}(x)=\rho^{Y_0}_{Y_1}$ and $\pi_{Y_1}(y)=\rho^{Y_2}_{Y_1}$. Also, since $\pi_{Y_2}(y),\pi_{Y_2}(y')$ are uniformly close to each other, if $K$ is sufficiently large then we have $d_{Y_2}(Y_1,y)>\theta$ and hence, one again, $\pi_{Y_1}(y')=\rho^{Y_2}_{Y_1}$. But then $d_Y(x,y')=d_Y(x,y)\geq K$, a contradiction with $Y_1\notin \rel^i_K(x,y')$.
\end{proof}

\begin{proposition}{prop:stable family}
Let $K\gg 2\theta$ and $F \subset \XX$ be any finite set. 
There exists $M = M(K, \mathfrak S, |F|)>0$ such that, 
for any $F' \subset \XX$ with $d_{Haus}(F,F') \le 1$ and $|F'|\leq |F|$, we have
$$\left| \rel_K(F) \bigtriangleup  \rel_K(F')\right| < M.$$
\end{proposition}

\begin{proof}
Assume throughout the proof that $K$ is sufficiently large.

Since there are finitely many colors, it suffices to prove the analogous statement for $\rel^i_K(F) \bigtriangleup  \rel^i_K(F')$ for any given $i$.

Applying Lemma \ref{lem:stable interval} twice, we see that if $d_\XX(x,x'), d_\XX(y,y')\le 1 $, then $|\rel^i_K(x,y)- \rel^i_K(x',y')|\leq 4$.

For each of the $|F|^2$ pairs $x,y\in F$, we can choose any $x',y'$ with $d_\XX(x,x'), d_\XX(y,y')\le 1 $, so that there are at most $4|F|^2$ elements of $\rel^i_K(F)=\bigcup_{x,y\in F}\rel^i_K(x,y)$ that are not in $\rel^i_K(F')$, i.e. $\left| \rel^i_K(F) -  \rel^i_K(F')\right|\leq 4|F|^2$. Symmetrically, we have $\left| \rel^i_K(F') -  \rel^i_K(F)\right|\leq 4|F'|^2$, and since $|F'|\leq |F|$ by assumption, we finally get $\left| \rel^i_K(F) \bigtriangleup  \rel^i_K(F')\right| < 8|F|^2$, as required.
\end{proof}

\subsection{Bounding involved domains}\label{subsec:bounding involved domains}

Let $(\XX, \mathfrak S)$ be a $G$-colorable HHS with stable projections for $G <
\mathrm{Aut}(\mathfrak S)$, as provided by Theorem \ref{thm:BBFS}.

Let $F, F' \subset \XX$ with $|F|=
|F'| \leq k$ and $d_{Haus}(F,F') \leq 1$.  We will now prove some stronger stability
results about how the set of relevant domains (and their subdomains) changes between $F$
and $F'$. 

For any $K \gg 2\theta$ as above, let $\UU(F) = \rel_K(F)$ and $\UU(F') = \rel_K(F')$.  Given $V\in \mathfrak S$, let $\UU^V(F) = \{W\in\UU(F): W\nest V\}$ and define $\UU^V(F')$ similarly.

In many of our stability results, we will need to know how domains in $\UU(F)$ may differ from those in $\UU(F')$.  We call such domains \textbf{involved}, and they come in two flavors:

\begin{definition}{defn:involved}
We say that $V \in \UU(F) \cup \UU(F')$ is \textbf{\em involved} in the
transition between $F$ and $F'$ if one of the following holds:
\begin{enumerate}
  \item $\pi_V(F) \ne \pi_V(F')$, or
  \item $\UU^V(F) \ne \UU^V(F')$.
\end{enumerate}
 \end{definition}

\begin{proposition}{prop:bounded involved}
If $K$ is sufficiently large then the following holds.
Given $k>0$ there exists $N_1 = N_1(k, \mathfrak S)>0$ such that, if $|F|,|F'| \le k$ and
$d_{Haus}(F,F') \le 1$, then there are at most $N_1$ domains
$V\in\UU(F) \cup \UU(F')$  involved in the transition between $F$ and $F'$.
\end{proposition}

\begin{proof}
By Proposition \ref{prop:stable family}, it suffices to bound the number of involved domains in $\UU(F) \cap \UU(F')$. However, we will still have to bound the number of involved domains of type (1) in $\UU(F) \cup \UU(F')$.
We note that, since $F,F'$ lie within Hausdorff distance $1$, up to increasing $K$ we can assume that for each $V\in\UU(F)$ we have $\diam_{\mathcal C V}(\pi_V(F'))\geq K/2$, and similarly for $V\in\UU(F')$.

\vspace{.1in}
\underline{\textbf{Involved of type (1)}:} Let $x \in F$.  We say that $V \in \UU(F) \cup \UU(F')$ is \textbf{exposed to $x$} if $\pi_V(x)$ is not contained in $\pi_V(F')$.
We define exposure for $x\in F'$ similarly (with an abuse, here we are considering $F$ and $F'$ as disjoint, so we should actually define exposure for $x\in F\sqcup F'$).

Observe that $V \in \UU(F) \cup \UU(F')$ satisfies $\pi_V(F) \neq \pi_V(F')$ if and only if $V$ is exposed to some $x$ in either $F$ or $F'$. Hence it suffices to bound the number of exposed domains.

Since $|F|, |F'| \leq k$, we may fix a point $x \in F$ and consider domains $V$ which are exposed to $x$.  The case of domains exposed to points in $F'$ follows from a symmetric argument.

Given $x,V$ as above, there is $y\in F$ so that $d_V(x,y)\geq K/4$ (this is because $\diam_{\mathcal C(V)}(\pi_V(F))\geq K/2$). Since $F$ has at most $k$ elements, we can further assume fix $y$ with said property.

Suppose for a contradiction that there exist domains $V_1, V_2, V_3 \in \left(\UU(F) \cup \UU(F')\right) \cap {\mathfrak S}_i$ which are exposed to $x$, where $\mathfrak S_i$ is the $i^{th}$ BBF family, making the $V_i$ necessarily pairwise transverse (this suffices since there are finitely many BBF families). Up to reordering, we have $V_1 \prec V_2 \prec V_3$ in $\rel_{K/2}(x,y)$.

Since $F$ and $F'$ lie at Hausdorff distance at most $1$, there is a pair $x_2, y_2 \in F'$ so that $d(x,x_2), d(y,y_2) \leq 1$, and necessarily we have $\pi_{V_2}(x) \neq \pi_{V_2}(x_2)$ (as we cannot have the containment ``$\subseteq$'').

Since $d(x,x_2), d(y,y_2) \leq 1$, by taking $K \gg 2\theta$ sufficiently large, we can ensure that $d_{V_i}(x_2, y_2) > 2\theta$ for $i=1, 2, 3$.  Since $V_1 \prec V_2 \prec V_3 \in \rel_{K/2}(x,y)$, we also must have the same order $V_1 \prec V_2 \prec V_3$ in $\rel_{2\theta}(x_2, y_2)$.

Thus by Theorem \ref{thm:BBFS}, it follows that $\pi_{V_2}(x_2) = \rho^{V_1}_{V_2}$.  However, Theorem \ref{thm:BBFS} also implies that $\pi_{V_2}(x) = \rho^{V_1}_{V_2}$.  This contradicts the fact that $\pi_{V_2}(x) \neq \pi_{V_2}(x_2)$, and completes the proof that there is a bound of domains of type (1).

\underline{\textbf{Involved of type (2)}:} Notice that if $W \in \UU(F) \cap \UU(F')$ of type (2), then there necessarily exists an exposed domain $V \in \UU(F) \cup \UU(F')$ of type (1) with $V \nest W$ (by a $\nest$-minimality argument).  We therefore bound the number of such containers $W$ for a fixed exposed domain $V$, of which there is a bounded number by the first part of the proof. 

Suppose for a contradiction that $W_1, W_2, W_3 \in \UU(F) \cap \UU(F')$ with $V \nest W$.  Since $|F|, |F'| \leq k$, we may further assume that there exist $x,y \in F$ for which $W_1, W_2, W_3 \in \rel_K(x,y)$.  Moreover, we may assume that each $W_i \in \mathfrak S_j$, for a fixed BBF family $\mathfrak S_j$.  Finally, up to reordering, we may assume that $W_1 \prec W_2 \prec W_3$ in  $\rel_K(x,y)$.

Theorem \ref{thm:BBFS} then provides that $\pi_{W_2}(x) = \rho^{W_1}_{W_2}$ and $\pi_{W_2}(y) = \rho^{W_3}_{W_2}$.  However, since $V \nest W_i$ for each $i=1,2,3$, we have $d_{W_2}(\rho^{W_1}_{W_2}, \rho^V_{W_2}) < \theta$ and $d_{W_2}(\rho^{W_3}_{W_2}, \rho^V_{W_2}) < \theta$, and so $d_{W_2}(\rho^{W_1}_{W_2}, \rho^{W_3}_{W_2}) < 2\theta$ by the triangle inequality.  But since $d_{W_2}(\rho^{W_1}_{W_2}, \rho^{W_3}_{W_2}) = d_{W_2}(x,y) > K > 2\theta$ by assumption, this is a contradiction.  This completes the proof.
\end{proof}

\begin{rem}\label{rem:rho_can_be_points}
 We conclude this section with a remark on HHS structures that, while not strictly necessary, will allow us to simplify the setup that we have to deal with in Section \ref{sec:stable trees}. The remark is that, given an HHS, we can $\mathrm{Aut}(\cuco{X},\mathfrak S)$--equivariantly change the structure in a way that all $\pi_V(x)$ and $\rho^U_V$ for $U\propnest V$ are points, rather than bounded sets, and that moreover the new structure has stable projections if the old one did. This can be achieved, for example, by replacing each $\CC(V)$ by the nerve of the covering given by subsets of sufficiently large diameter (which is quasi-isometric to $\CC(V)$). In particular, the vertices of the new $\CC(V)$ are labelled by bounded sets, and we can redefine $\pi_V(x)$ to be the vertex labelled by $\pi_V(x)$, and similarly for $\rho^U_V$; all properties required are straightforward to check.
 
 In Section \ref{sec:stable trees}, we will deal with finite subsets of a hyperbolic space, while, if (in Section \ref{sec:stable cubulations}) we did not modify the HHS structure as outlined above, we would have to deal with finite collections of bounded subsets. This is possible, but would make the arguments more opaque.
\end{rem}

\newcommand\ltree{\lambda}
\newcommand\lforest{\lambda'}

\section{Stable trees}\label{sec:stable trees}
In this section we will consider the geometry of trees in a $\delta$-hyperbolic space, in
preparation for arguments that will take place in the individual hyperbolic spaces of our
hierarchical structure. Our main result will be Theorem
\ref{thm:stable tree}, stated below after some preliminary
definitions. This is the only result from this section that will get used later (namely, in Section \ref{sec:stable cubulations}).

Fix a geodesic $\delta$-hyperbolic space $\ZZ$.
For a finite subset $F\subset \ZZ$ let $hull(F)\subset\ZZ$
be the set of geodesics connecting points of $F$. Hyperbolicity tells us that $hull(F)$
can be approximated by a finite tree with accuracy depending only on $\delta$ and the
cardinality $\# F$. To systematize this for the purposes of this section, we make the
following definitions.

Let us fix, a function $\ltree$ which assigns, to any finite subset $F$ of
$\ZZ$, a minimal network spanning $F$. That is, $\ltree$ is a 1-complex with the property
that $\ltree\union F$ is connected, and has minimal length among all such
1-complexes. Minimality implies $\ltree(F)$ is a tree. Let us similarly define 
$\lforest$ which assigns, to any finite collection $A_1,\ldots,A_k$ of subsets of $\ZZ$
a minimal network that spans them. 
That is, $\lforest(A_1,\ldots,A_k)$ is a 1-complex in $\ZZ$ of minimal length with the
property that the quotient of $\lforest\union A_1\union\cdots A_k$ obtained by collapsing
each $A_i$ to a point is connected. Minimality again implies that this collapsed graph is a
tree.  For convenience we assume that $\ltree(\{x_1,\ldots,x_k\}) =
\lforest(\{x_1\},\ldots,\{x_k\})$.

The following lemma illustrates a basic property of hyperbolic spaces, and we omit its
proof. 
\begin{lemma}{basic tree lemma} 
Let $\ZZ$ be a geodesic $\delta$-hyperbolic space and $\ltree$ a minimal network function
as above. Then there exists $\epsilon_0=\epsilon_0(k,\delta)$ so that for all $\epsilon\geq \epsilon_0$ there exists $\ep'>\ep$ such that, if $F\subset \ZZ$ has
cardinality $k$ then
\begin{itemize}
  \item There is 
a $(1,\ep/2)$-quasi-isometry $\ltree(F)\to hull(F)$ which is $\ep/2$-far from the
identity.
\item
For any two points $x,y\in
N_\ep(\ltree(F))$, any geodesic joining them is in $N_{\ep'}(\ltree(F))$.
\end{itemize}

\end{lemma}

In the rest of this section we consider the following situation. 
Let a (large but) finite set  $\YY\subset N_{\ep/2}(hull(F))$ be given (see Section \ref{sec:stable cubulations} for what $\YY$ will
be in our setting). 
It is possible to divide $\ltree(F)$ up into a union of
subtrees some of which are close approximations to ``clusters'' in $\YY$ and the rest interconnect
the clusters, but such a construction is not unique, depending on many choices (including
the choice of $\ltree(F)$ itself). Our goal in this
section is to describe a version of this which is stable, in the sense that small changes
in the sets $F$ and $\YY$ only alter the tree and its subtrees in a controlled way -- independently of the
diameter of $F$ or the cardinality of $\YY$. 

{\em Remark:} For convenience in our discussion we allow ourselves to assume that the points of $F$ are
all leaves of $\ltree(F)$. This can be arranged by a slight perturbation, or by
considering each point of $F$ as the endpoint of an additional edge of length 0.

Given $E\gg \ep$, let $\CC_E(\YY\union F)$ be the graph whose edges
connect points of $\YY\union F$ that are at most $E$ apart. Vertex sets of connected components of
$\CC_E$ are called \textbf{clusters}.  We will choose $E$ to be a suitably large multiple of
$\ep'$.  We note that the relation of $E$ and $\ep',\ep$ is the one sensitive part of the argument, and
elsewhere we can be content with order-of-magnitude arguments.

For a simplicial tree $T$, let $d(v)$ be the valence of each vertex $v$ and let $k(T)$ be
the number of leaves, i.e. vertices of valence 1. We have
$\sum_{d(v)>2}( d(v)-2) = k-2$, for example by an Euler characteristic argument. We call this quantity the \textbf{total branching} of $T$.

The following theorem is the main result of this section:

\begin{theorem}{thm:stable tree}
  Given $k,N, \delta$, and $\ep\geq \ep_0(k,\delta)$ as in Lemma \ref{basic tree lemma} there exists $K=K(k, N, \delta)>0$ such that the following holds.
Let $\ZZ$ be a geodesic $\delta$-hyperbolic space and let
$F,\YY\subset \ZZ$ be finite subsets, where $|F|\le k$ and $\YY\subset N_{\ep/2}(hull(F))$.

There exists a metric tree
$T=T(F,\YY)$ with a
decomposition into two forests $T=T_c \union T_e$ intersecting along a
finite set of points, and a map $\Xi=\Xi_{F,\YY}:T(F,\YY)\subset\ZZ$ such that
\begin{enumerate}[$(a)$]
\item\label{item:branching} The total branching of $T$ is bounded by $2k-4$.
  \item\label{item:qi} $\Xi$ is a $(K,K)$--quasi-isometric embedding with image $K$--Hausdorff close to $hull(F)$.
  \item\label{item:qi2} For each component $\tau$ of $T_e$ we have that $\Xi|_\tau$ is a $(1,K)$--quasi-isometric embedding, and an isometry onto $\Xi(\tau)$ endowed with its path metric.
    \item\label{item:bijection} There is a bijection $b$ between components of $T_c$ and
      clusters in $\CC_E(\YY\union F)$, so that each component $\tau$ of $T_c$ is so that $\Xi(\tau)$ is $K$--Hausdorff close to $b(\tau)$. 
\end{enumerate}

Furthermore, if $F',\YY'\subset \ZZ$ and $g\in Isom(\ZZ)$ are such that $|F'|\le n+1$, $\YY'$ is finite, $d_{Haus}(gF,F') \le 1$, and $|g\YY \symdiff \YY'| < N$, then there exists a constant $L = L(N, k, \delta)>0$ and subsets $T_s \subset T_e(F,\YY)$ and $T'_s \subset T_e(F', \YY')$ such that, identifying components of $T_e(F,\YY), T_e(F', \YY')$ with their images in $\ZZ$, we have:
\begin{enumerate}
\item The components of $T_s$ and $T'_s$ are contained in the edges of $T_e(F,\YY)$ and $T_e(F', \YY')$, respectively.
\item The complements $T_e(F,\YY) \ssm T_s$ and $T_e(F', \YY') \ssm T'_s$ have at most $L$ components each of diameter at most $L$.
\item There is a bijective correspondence between the sets of the components of $gT_s, T'_s$.
\item Under this correspondence, all but $L$ components are exactly the same, and the identical components of $T_s$ and $T'_s$ come from the identical components of $T_e(F, \YY)$ and $T_e(F', \YY')$. 
\item The remaining $L$ components of $gT_s$ are each at Hausdorff distance $L$ of the corresponding component in $T'_s$.
\end{enumerate}

\end{theorem}

\realfig{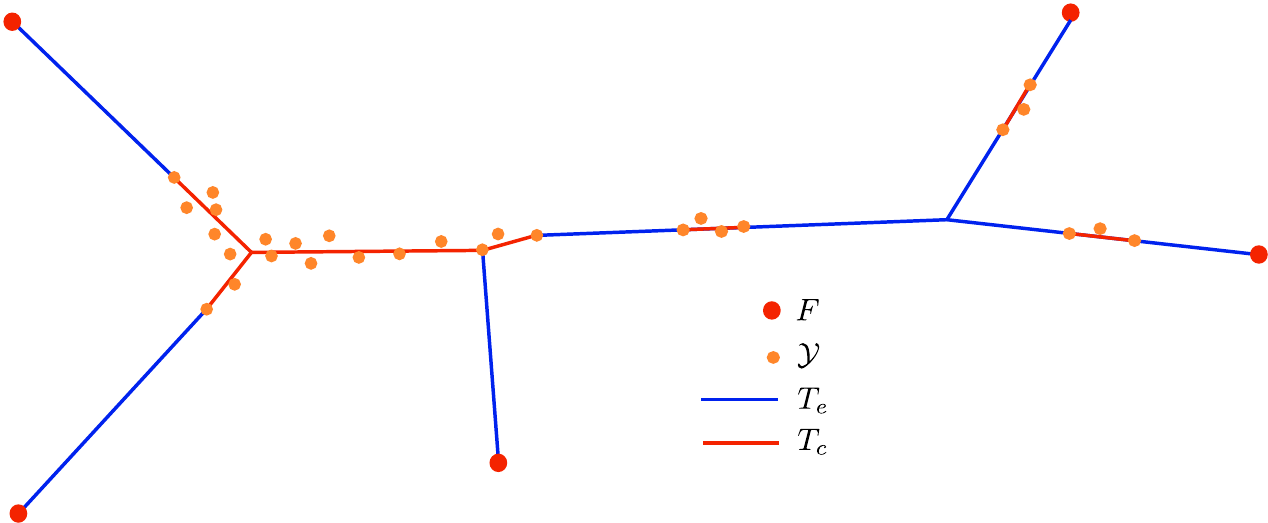}{An example of the stable tree $T =  T_c\union T_e$ provided by
  Theorem \ref{thm:stable tree}.}
  
We call the trees $T(F, \YY)$ \textbf{stable trees}.

\begin{rem}[Coarse equivariance and its proof]\label{rem:equivariance}
 The ``furthermore'' part of Theorem \ref{thm:stable tree} can be interpreted as simultaneously stating two facts. For $g$ the identity, it says that the trees are stable under perturbations of $F$ and $\YY$. Instead, for $F'=gF$ and $\YY'=g\YY$, it says that the construction is coarsely equivariant. In either case, what we have to prove is essentially the following. The construction relies on certain choices, namely the choices of functions $\ltree,\lforest$ as above, and we have to show that these only cause the kinds of perturbations described in the statement of the theorem. From this perspective, it is clear that the proof for a general $g$ is the same as that for $g=1$, as $g T(F, \YY)$ coincides with the tree $T(gF, g\YY)$ constructed based on different choices. To save notation and make the proof more readable, we will prove only the case where $g$ is the identity.
\end{rem}

\subsection{Cluster separation graph}\label{separation graph}
Let $\CC_E = \CC_E(F\union\YY)$ be as above and let $C_1,C_2,C_3 \subset \CC^0_{E}$ be clusters
(i.e. vertex sets of connected components). We say that $C_2$ \textbf{separates $C_1$ 
from $C_3$ in $\ZZ$} if there exists a minimal $\ZZ$-geodesic segment $\sigma$ with
endpoints on 
$C_1$ and $C_3$ which meets the $2\ep'$-neighborhood of $C_2$.

\begin{definition}{defn:separation graph}
Let $\GG_E=\GG_E(F\union\YY)$ be a graph whose vertex set $\GG^0_E$ is the set of clusters of
$\CC_E$, and where
$[C_1,C_2]$ is an edge whenever there is no cluster separating $C_1$ from
$C_2$ in $\ZZ$.  We call $\GG_E$ the \emph{separation graph} for $\CC_E$.
\end{definition}

\begin{lemma}{G connected}
If $E > 4\ep'$, then $\GG_E$ is connected.
\end{lemma}

\begin{proof}
Let $C,C' \in \GG^0_E$. If $C \neq C'$, then $d_\ZZ(C,C') > E$.  If
they are not adjacent in $\GG$, then there is a third cluster $B$
separating them in $\ZZ$. Let $\sigma$ be a minimal geodesic connecting $C$ to
$C'$ with $p\in \sigma$ within $2\ep'$ of $B$. Then $p$ is distance at
least $E - 2\ep'$ from each end of $\sigma$ since $B$ is at least $E$
from both $C$ and $C'$. It follows that $d_{\ZZ}(B,C) \le d_{\ZZ}(C,C') - E +
4\ep'$, and similarly for $d_{\ZZ}(B,C')$.

If $d_\ZZ(C,C') \le 2E-4\ep'$ this gives $d_{\ZZ}(B,C) \le E$ which is a contradiction so
$C$ and $C'$ must be connected by an edge in $\GG_E$. For $d_\ZZ(C,C')> 2E-4\ep'$, we have that $d_\ZZ(B,C)$ and $d_\ZZ(B,C')$ are smaller than
$d_\ZZ(C,C')$ by at least $E-4\ep'$, so we can proceed inductively.
\end{proof}

For ease of notation, set $\CC=\CC_E$ and $\GG = \GG_E$.

\begin{definition}{defn:shadow}
For any subset $A$ of $\ZZ$, let its {\em shadow}
$s(A)$ be the subtree of $\ltree(F)$ obtained by taking the convex hull (in
$\ltree(F)$) of all the points in $\lambda(F)$ within distance $\ep$ from points of $A$. For a singleton $\{x\}$ we also write $s(x) :=s(\{x\})$. 
\end{definition}
Note that, since $\YY\union F$ is in $N_\ep(\ltree(F))$ by hypothesis, 
$s(C)\ne \emptyset$ for any non-empty subset $C\subset\YY\union F$. 

The rest of this subsection is devoted to establishing several properties of shadows which
will connect the separation properties of clusters in $\GG$ to separation properties of
their shadows in $\ltree(F)$, thereby allowing us to work with $\GG$ and independently of
$\ltree(F)$.

The next lemma controls how and when shadows of clusters can intersect.

\begin{lemma}{shadow intersection}
Let $E>7\ep$ and let $C, C' \in \GG^0$ be distinct clusters.  Then:
 \begin{enumerate}
 
  \item $s(C) \intersect s(C')$ can contain no leaf of $s(C)$ or $s(C')$;
    \item The diameter of $s(C) \intersect s(C')$ is bounded by a constant depending on
      $\#F$, $E$, and $\epsilon$. 
  \item If at least one of $s(C)$ and $s(C')$ is an interval along an edge of $\ltree(F)$,
    then $s(C) \intersect s(C') = \emptyset$.
  \end{enumerate} 
\end{lemma}

\realfig{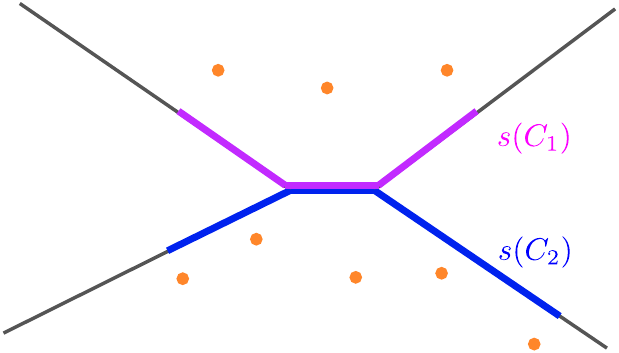}{When there is branching, shadows can overlap in their interiors,
  but never at their leaves.}

\begin{proof}
Note first that for any $x\in
  C$, $s(x)$ is a subtree of diameter at most $3\ep$. This is because any two extreme
  points of $s(x)$ are within $\ep$ of $x$, and $\ltree(F)$ is $(1,\ep)$-quasi-isometrically
  embedded. Similarly, for any $x,y\in C$, $\diam_{\ltree(F)}(s(\{x,y\})) \le d(x,y)+3\ep$.

\begin{claim}\label{claim:shadow_close_to_cluster}
 For every $p\in s(C)$, there exists $q\in s(C)$ at distance (in
  $\ltree(F)$) at most $(E+3\ep)/2$, such that $d(q,C)\le \ep$.
\end{claim}

  \begin{proof}
Either $p\in s(x)$ for some $s(x)$ containing an extreme point of $s(C)$, or $p$
  separates some $s(x)$ from $s(y)$, for $x,y\in C$. In the first case $p$ is within
  $3\ep/2$ of a point $q$ for which $d(q,x)\le \ep$ and we are done. In the second case, 
a path in $\CC_E$ from $x$ to $y$ then yields a sequence of points $x_i\in C$ such that
$d(x_i,x_{i+1}) \le E$ and $p$ is contained in one of the shadows $s(\{x_i,x_{i+1}\})$. Since
$\diam_{\ltree(F)}(s(\{x_i,x_{i+1}\}) \le E+3\ep$, we find that $p$ is within $(E+3\ep)/2$ of an
extreme point $q$ of $s(\{x_i,x_{i+1}\})$, so that $d(q,x_i)\le \ep$ or $d(q,x_{i+1})\le
\ep$. The claim follows.
\end{proof}

For (1), suppose that a leaf $p$ of $s(C)$ is in $s(C')$. 
Note that the leaves of $s(C)$ and $s(C')$ are within $\ep$ of $C$ and $C'$, respectively. 
By the previous paragraph, there is a point $q$ of $s(C')$ within $(E+3\ep)/2$ of $p$ which is
$\ep$ close to $C'$
Thus we obtain $d(C,C')\le 2\ep  + (E+3\ep)/2 < E$, so $C=C'$.

For (2), suppose that $s(C)\intersect s(C')$ contains an edge $e$ of length greater than
$2(E+3\ep)$. Claim (\ref{claim:shadow_close_to_cluster}) implies that there is a set $R$
in $s(C)$ consisting of points at distance $\ep$ from $C$ and whose
$(E+3\ep)/2$-neighborhood covers $s(C)$; there is also a similar set $R'$ in $s(C')$. Since $e$ is
in both shadows, it must be that $e\intersect R$ and $e\intersect R'$ both cut $e$ into
intervals of length at most $E+3\ep$. Thus it must be that there is a point $r\in
R\intersect e$ and $r'\in R'\intersect e$ that are distance $(E+3\ep)/2$ apart. Then just
as before we obtain $d(C,C') < E$ so $C=C'$.  Now the number of edges in $s(C)\intersect
s(C')$ is bounded by the total branching of the tree, which depends on $\#F$. This gives
(2).

Finally, for (3), if one of $s(C)$ and $s(C')$ is an interval contained in 
an edge of $\ltree(F)$ then it is easy to see that, if they overlap, then
one must contain a leaf of the other, thereby violating (1).
\end{proof}

\begin{definition}{defn:E}
From now on we set $E=8\ep'$ so that the conclusions of both  Lemmas \ref{G connected} and
\ref{shadow intersection} hold.
\end{definition}

The following lemma connects the separation properties in $\GG$ of a cluster to the separation properties in $\ltree(F)$ of its shadow:

\begin{lemma}{general separation}
Let $C$ be a cluster and $S_1,\ldots,S_k$ be the components of $\ltree(F) \ssm int(s(C))$ which
meet $s(C)$ at a leaf of $s(C)$. Let $G_i$ be the set of clusters $B\in \GG^0\ssm\{C\}$ such
that $s(B)\intersect S_i \ne \emptyset$. Then each $G_i$ is in a distinct component of
$\GG\ssm C$, and moreover the valence of $C$ in $\GG$ at least
$k$. 
\end{lemma}

\realfig{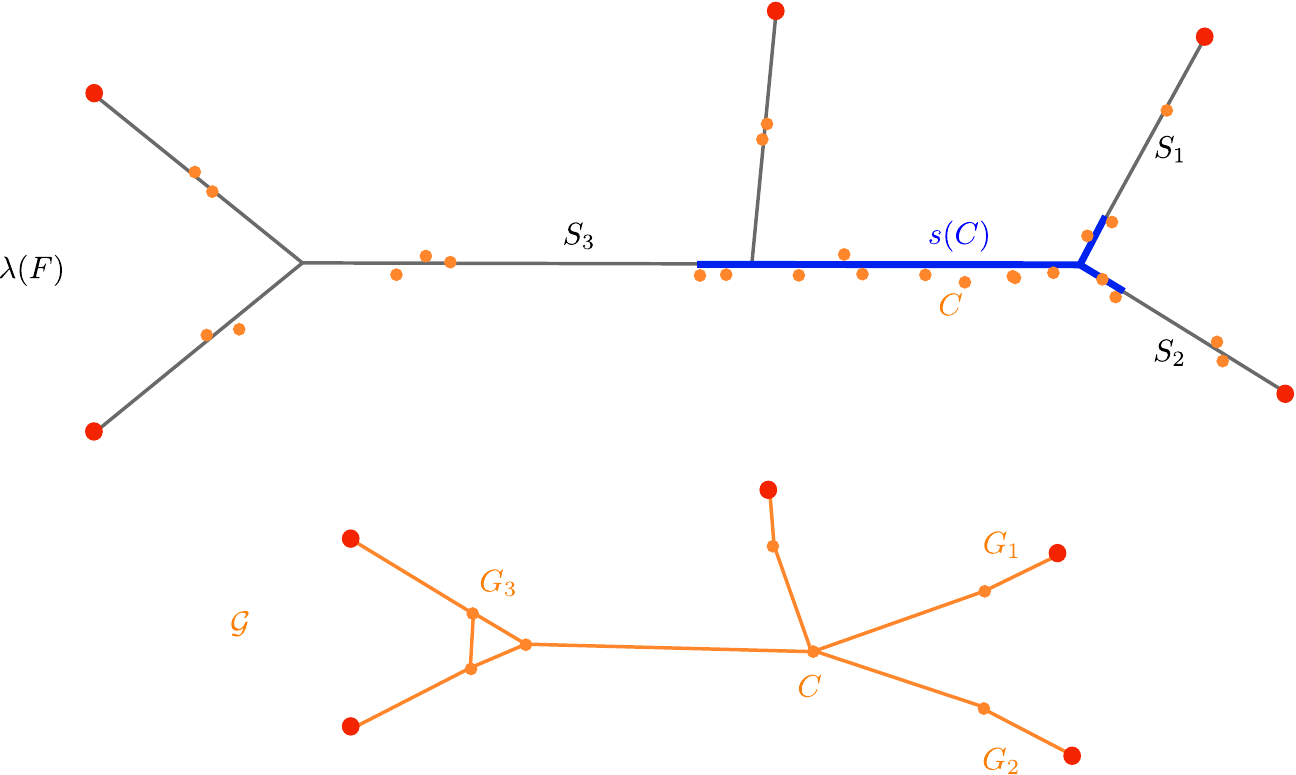}{The cluster $C$ is guaranteed valence at least 3 (via $S_1,S_2$
  and $S_3$) by Lemma \ref{general separation}. In this case it has valence 4. }

\begin{proof}
Note that if $B$ is a cluster in $G_i$ then $s(B)$ is actually
disjoint from $s(C)$, since the leaves of $s(C)$ cannot meet $s(B)$
by Lemma \ref{shadow intersection}.  Moreover, there may be
clusters $C'\ne C$ that are not in any $G_i$; their shadows meet
components of $\ltree(F)\ssm int(s(C))$ that do not meet leaves of $s(C)$. 

Let $A\in G_i$ and $B\notin G_i\union\{C\}$.
A minimal geodesic $\sigma $ in $\ZZ$ connecting $A$ to $B$ must be $\ep'$-close to the path in
$\ltree(F)$ connecting $s(A)$ to $s(B)$, and this path passes through a leaf $p$ of $s(C)$ (namely $s(C)\intersect S_i$). Thus there is a point of $C$ within
$\ep+\ep' < 2\ep'$ of 
$\sigma$, so $C$ separates $A$ from $B$ in $\ZZ$. In particular $A$ and $B$ cannot be
adjacent in $\GG$. 

Thus $G_i$ cannot be connected to any vertex in $\GG^0 \ssm (G_i\union\{C\})$, which implies
distinct $G_i$ are in distinct components of $\GG\ssm C$. 

To see that the valence is at least $k$, we must check that each $G_i$
is nonempty. But each $S_i$ must contain a leaf of $\ltree(F)$, which
is a point of $F$, so there must be a cluster whose shadow is in $S_i$. 
\end{proof}

\begin{lemma}{edge bivalent}
  If $e$ is an edge of $\ltree(F)$, the clusters $C$ whose shadows
  $s(C)$ are subintervals of $e$ form a path in $\GG$ whose interior
  vertices are bivalent. The ordering of this path matches the
  ordering of the shadows in $e$.
\end{lemma}

\realfig{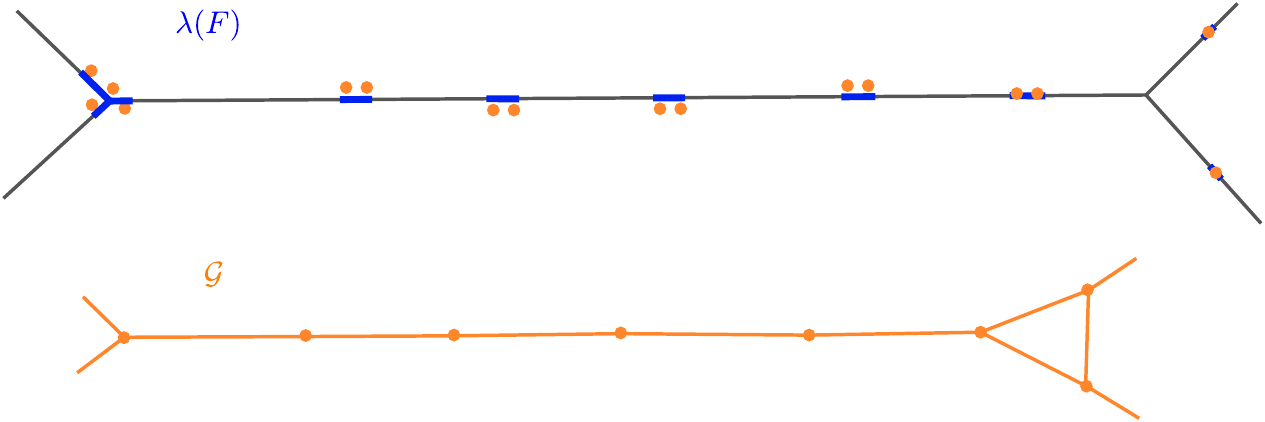}{Clusters with shadows on an edge in $\ltree(F)$ give rise to a
  path of bivalent vertices in $\GG$.}

\begin{proof}
Let $\{C_1, \dots, C_l\}$ be the set of clusters whose shadows are subintervals of $e$.
By Lemma \ref{shadow intersection}, $s(C_i) \cap s(C_j) = \emptyset$ for all $i,j$.  We
may therefore assume that their indices correspond to the order they appear along $e$ in
$\ltree(F)$.

The complement $\ltree(F) \setminus int(s(C_i))$ has two components for each $i$, labeled
$A_i^-$ and $A_i^+$ so that $A_i^-$ contains $s(C_{i-1})$ when $i>0$ and 
$A_i^+$ contains $s(C_{i+1})$ when $i<l$. By our ordering no shadows lie between $s(C_i)$
and $s(C_{i+1})$.  Lemma \ref{general separation} implies that
$C_i$ separates (in $\GG$) the clusters whose shadows lie in $A_i^-$ from those in
$A_i^+$. In particular no $B$ can separate $C_i$ from $C_{i+1}$ in $\GG$, so they are
adjacent and we obtain a path $C_1,\ldots,C_l$ in $\GG$. Moreover for $1<i<l$ we can see that $C_i$ is bivalent as
follows: if $D \in \CC \setminus \{C_{i-1}, C_i, C_{i+1}\}$, then one of $C_{i-1}$ or
$C_{i+1}$ separates $D$ from $C_i$ in $\GG$, again by Lemma \ref{general separation}, and
so there can be no edge $[C_i,D]$ and the valence of $C_i$ is exactly 2. 

\end{proof}

\begin{lemma}{detect good cluster}
If $C$ has valence 2 in $\GG$ but $s(C)$ is not an interval inside an
edge of $\ltree(F)$, then $C$ contains a point of $F$. 
\end{lemma}

\begin{proof}
If $s(C)$ is not an interval in an edge of $\ltree(F)$, it has a branch point and hence at least 3 leaves. At most
two of these can be interior to $\ltree(F)$, because otherwise $C$
would have valence at least 3 in $\GG$ by Lemma \ref{general separation}.

Thus $s(C)$ contains a leaf $q$ of $\ltree(F)$, which is a point of
$F$. This means $d(q,C)\le\ep < E$ (notice that, since $q$ is a leaf, it lies in the convex hull of a subset of $\ltree(F)$ only if it lies in the subset).  Hence, we have $q\in C$. 
\end{proof}

\subsubsection*{Structure of bivalent clusters}

Let $\EE^0$ denote the set of clusters $C\in\GG^0$ which have valence 2 in $\GG$ and do not
contain a point of $F$. Lemma \ref{detect good cluster} implies that each $C\in\EE^0$ has
shadow inside an edge of $\ltree(F)$. 

The next lemma gives that almost all clusters are bivalent:

\begin{lemma}{clusterfacts}
 $\#(\GG^0\ssm \EE^0) \le 2k-2$.
\end{lemma}

\begin{proof}
For a cluster $C\in\GG^0\ssm\EE^0$, either $C$ contains a point of $F$, or $s(C)$ contains
a branch point of $\ltree(F)$. There are at most $k$ clusters of the former type. The number of clusters of the latter type is bounded by the total branching $b(\ltree(F))$, but to show this we must contend
with the fact that shadows can overlap.

Let $W\subset \ltree(F)$ be a connected union of shadows $s(C_1),\ldots,s(C_m)$, each of which
contains a branch point. By Lemma \ref{shadow intersection}, no leaf of $s(C_i)$ can be in $s(C_j)$ for
$i\ne j$. Hence all leaves of $s(C_i)$ must be leaves of $W$ and disjoint from each
other. Since each $s(C_i)$ has at least two leaves, we have
$$
m \le \lfloor k(W)/2 \rfloor = \lfloor (b(W)+2)/2 \rfloor
$$
where $k(W)$ is the number of leaves and $b(W)$ is the total branching of $W$. Since
$b(W)\ge 1$, this implies $m \le b(W)$. Summing over all such $W$ we find that the number of
clusters with branch points in their shadows is bounded by $b(\ltree(F))$, or $k-2$. The
desired inequality follows. 
\end{proof}

Let $\EE$ be the subgraph of $\GG$ induced on the vertices $\EE^0$. 

\begin{lemma}{components of E}
Let $\EE_1,\ldots\EE_m$ be the components of $\EE$. 
For each $\EE_i$ there is an edge $e_i$ of $\ltree(F)$ such that $\EE_i$ is a path $C_1, \dots, C_{r_i}$ in $\GG$ consisting of all
elements of $\EE$ whose shadows lie in the interior of $e_i$; the edges $e_i$ are distinct.
\end{lemma}

\begin{proof}
Since each cluster $D \in \EE_i$ is a bivalent vertex of $\GG$ with shadow in an edge of
$\ltree(F)$ by Lemma \ref{detect good cluster}, and Lemma \ref{edge bivalent} implies that all
such clusters with shadows on a given edge $e \in \ltree(F)$ form a path in $\GG$, it suffices
to prove that no two such edge paths of bivalent clusters in $\GG$ are directly connected
by an edge.

Suppose $C, D \in \EE_i$ are connected by an edge in $\GG$ but $s(C)$ and $s(D)$ are not
contained in a single edge of $\ltree(F)$.  Since $s(C) \cap s(D) = \emptyset$, we may label the components of $\ltree(F)\ssm s(C)$ and $\ltree(F)\ssm s(D)$ by $\gamma_{\pm}$ and $\delta_\pm$, respectively, so that $s(C)\subset \delta_-$ and $s(D) \subset
\gamma_-$. Then the intersection $\gamma_-\intersect \delta_-$
contains a vertex $v$ of $\lambda(F)$ of valence at least $3$.

By Lemma \ref{general separation}, $\GG\ssm C$ is divided into subgraphs $\GG(\gamma_\pm)$ 
spanned by clusters whose shadows are in
$\gamma_\pm$ respectively and are separated by $C$, and similarly 
$\GG(\delta_\pm)$ are separated by $D$ respectively. In particular note $C\in
\GG(\delta_-)$ and $D\in \GG(\gamma_-)$. 

Since $v$ has valence at least 3, there is a component of $\ltree(F)\ssm
\{v\}$ that meets neither $s(C)$ or $s(D)$. A leaf of this
component is in the shadow of a cluster $B$ which is therefore in
$\GG(\gamma_-)\intersect \GG(\delta_-)$.

Since $\GG$ is connected, $B$ is
connected to $C$ within $\GG(\gamma_-)$ and to $D$ within $\GG(\delta_-)$. Since $C$ and
$D$ are bivalent and by hypothesis adjacent in $\GG$, the edge between them is the only edge connecting
$C$ to $\GG(\gamma_-)$, and the only edge connecting $D$ to $\GG(\delta_-)$. Hence any path from $B$ to $C$ must pass through
this edge and must therefore meet $D$ first. Reversing the roles of $C$ and $D$ we obtain a contradiction.
\end{proof}

\subsection{Constructing the stable tree}
In this subsection, we construct our stable tree $T(F,\YY)$  from the structure of
$\GG$ without referring to $\ltree(F)$ directly. 
In Proposition \ref{the forest for the trees} below, we prove it is quasi-isometric to $\ltree(F)$.

\begin{figure}
\includegraphics[width=0.75\textwidth]{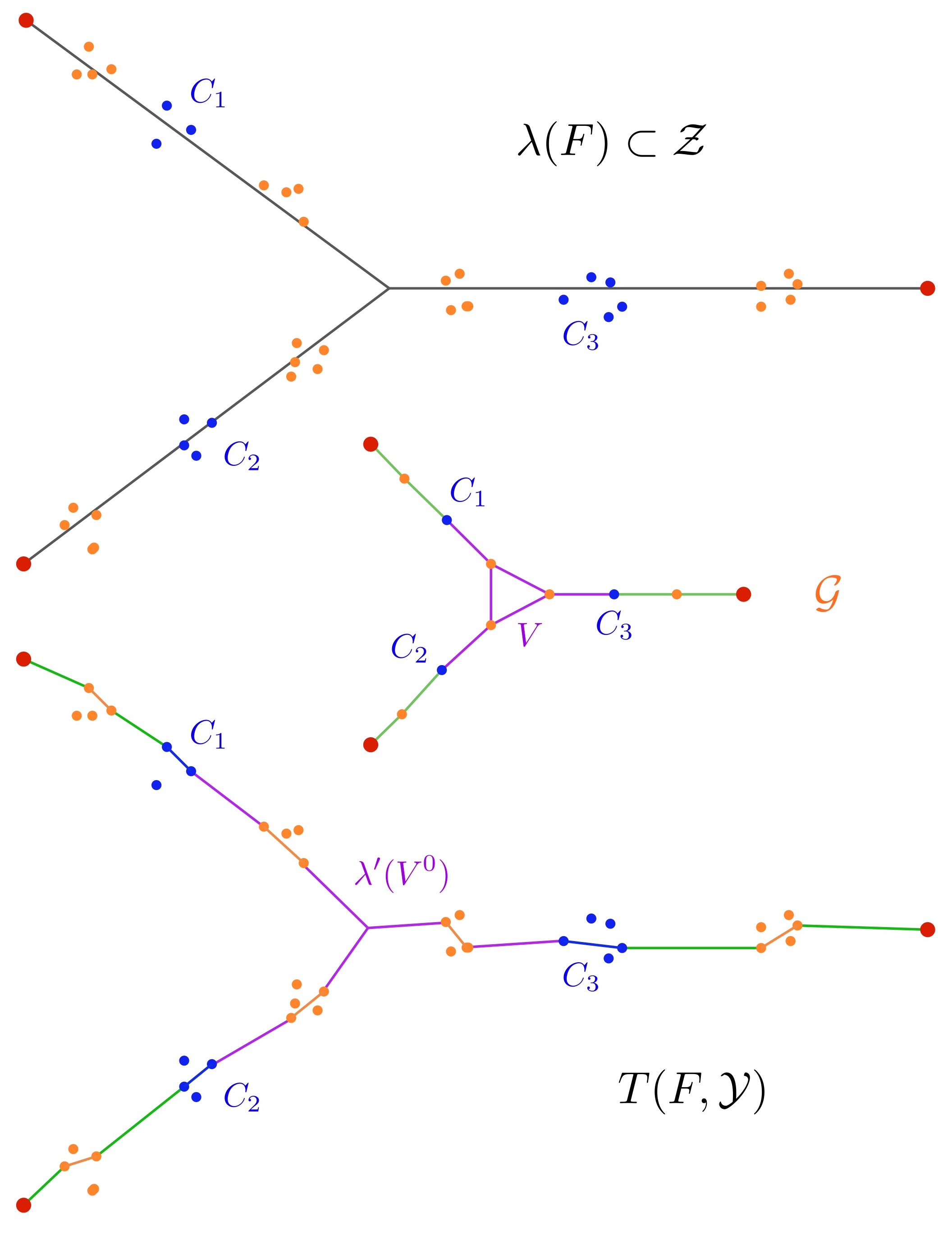}
\caption{The construction of a stable tree.  Note that each complementary component of $\GG \ssm \EE^0$ determines multiple components of $T_e$, e.g., the lavender forest $\lambda'(V^0)$ determined by the component $V$ whose boundary is the bivalent clusters $C_1,C_2,C_3$.  Each cluster $C$ then determines a single component $\mu(C)$ of $T_c$ by connecting the points $r(C) = C \cap (T_e \cup F)$.}\label{fig:edgecomponent}
\end{figure}

\subsubsection*{The two forests}

Now let us proceed to define the forests $T_c(F,\YY)$ and $T_e(F,\YY)$. We let
$\GG=\GG_E(F\cup\YY)$ be as above. 

Let $\VV$ denote the set of closures of connected components of $\GG\ssm \EE^0$. Thus each element of
$\VV$ is a subgraph connected to the rest of $\GG$ along vertices in $\EE^0$. 
For each $V\in \VV$ let $V^0$ denote its vertex set, which is a collection of clusters.
Some elements of $\VV$ are single edges $[C,D]$ where
$C,D\in\EE^0$, and others are subgraphs containing vertices in $\GG^0\ssm\EE^0$, and we
note they are not necessarily trees, though Lemma \ref{clusterfacts} bounds their size.

For each $V\in \VV$, let $\lforest(V^0)$ be the minimal network defined
in the beginning of this section, where the
elements of $V^0$ are interpreted as sets of clusters in $\ZZ$.

Now define
$$T_e = T_e(F,\YY) = \bigsqcup_{V \in \VV} \lforest(V^0).$$

\begin{rem}
 The forest $T_e$ is a disjoint union of copies of forests each contained in $\ZZ$. It is important to note, however, that these trees might in fact intersect in $\ZZ$. With a slight abuse, we will conflate the abstract copies of the $\lforest(V^0)$ that constitute $T_e$ and their ``concrete'' counterparts in $\ZZ$. Similar comments apply to $T_c$ below. Since the map $\Xi:T\to \ZZ$ is just going to be the identity on all the components of $T_e$ and $T_c$, we will allow ourselves to regard $T$ as a subset of $\ZZ$ for purposes that do not require understanding the metric of $T$, e.g. when measuring the Hausdorff distance between (the image in $\ZZ$ of) a subset of $T$ and a subset of $\ZZ$.
\end{rem}

Note that $T_e$ is a forest whose leaves are points of clusters.

Collapsing clusters to points, $T_e$ becomes a connected network $N$,
by the definition of $\VV$. This connected network is a union of trees joined at points that correspond to vertices of $\EE^0$. Since any vertex in $\EE^0$ disconnects $\GG$, each of these join points disconnects $N$, so that we see that $N$ is a tree.

Now for each cluster $C \in \GG^0$, we consider the set of points $r(C) = C\intersect (T_e
\union F)$.  We let $\mu(C)$ denote the tree $\ltree(r(C))$, and define
$$T_c =T_c(F,\YY) = \bigsqcup_{C \in \GG^0} \mu(C).$$

\subsubsection*{The tree}
We now define  $T(F,\YY) = T_c(F,\YY) \union T_e(F,\YY)$, or $T=T_c\union
T_e$ for short. Note that $T$ is a tree because as above collapsing the subtrees of $T_c$
to points yields a tree; see Figure \ref{fig:edgecomponent}.

\begin{lemma}{minimal networks don't overlap}
Let $T=T(F,\YY) = T_c\union T_e$.   
\begin{enumerate}
\item The total branching $b=b(T)$ is bounded by $2k-4$, and the leaves of $T$ are contained in $F\cup \YY$.
\item $\mu(C)\subseteq N_{O(\ep)}(C)$, so that $T_c\subset \mathcal{N}_{O(\ep)}(\YY \cup F)$.
\item For all $p \in T_e$, we have $d_{\ZZ}(p, \YY \cup F) \geq \frac{1}{b}d_T(p,\partial T_e) -
  O(\ep)$.
\end{enumerate}
\end{lemma}

\begin{proof}
To bound $b(T)$, we bound the number of leaves $T$ can have. Leaves of $T$ are leaves of the various components of $T_e$ and $T_c$, and thus can arise in two ways: (a) If a cluster $C$ contains points of $F$, the points in $C \cap F$ can be leaves of
$\mu(C)$ which are also leaves of $T$. There are at most $k$ such points.
(b) If a cluster $C$ contains no points of $F$ and a single point $q$ of $\boundary
T_e$, which is connected to only one subtree of $T_e$, 
then $q$ is a leaf of $T$ (Figure \ref{fig:newleaf}).  All other vertices of $\boundary T_e \union \boundary T_c$ 
have valence at least 2. Notice that we already showed that all leaves of $T$ are contained in $F\cup \YY$.

Clusters of type (b) must be in $\GG^0\ssm\EE^0$ since every cluster in $\EE^0$ 
belongs to two subgraphs in $\VV$, and hence either has two points in $\boundary T_e$ or
two subtrees of $T_e$ meeting at a single point.  The number of clusters in
$\GG^0\ssm\EE^0$ that don't contain points of $F$ was bounded in Lemma \ref{clusterfacts} by $k-2$. 

This gives us a bound of $2k-2$ on the total number of leaves in $T$, which bounds the
total branching by $2k-4$. This proves part (1).

\begin{figure}
\includegraphics[width=1\textwidth]{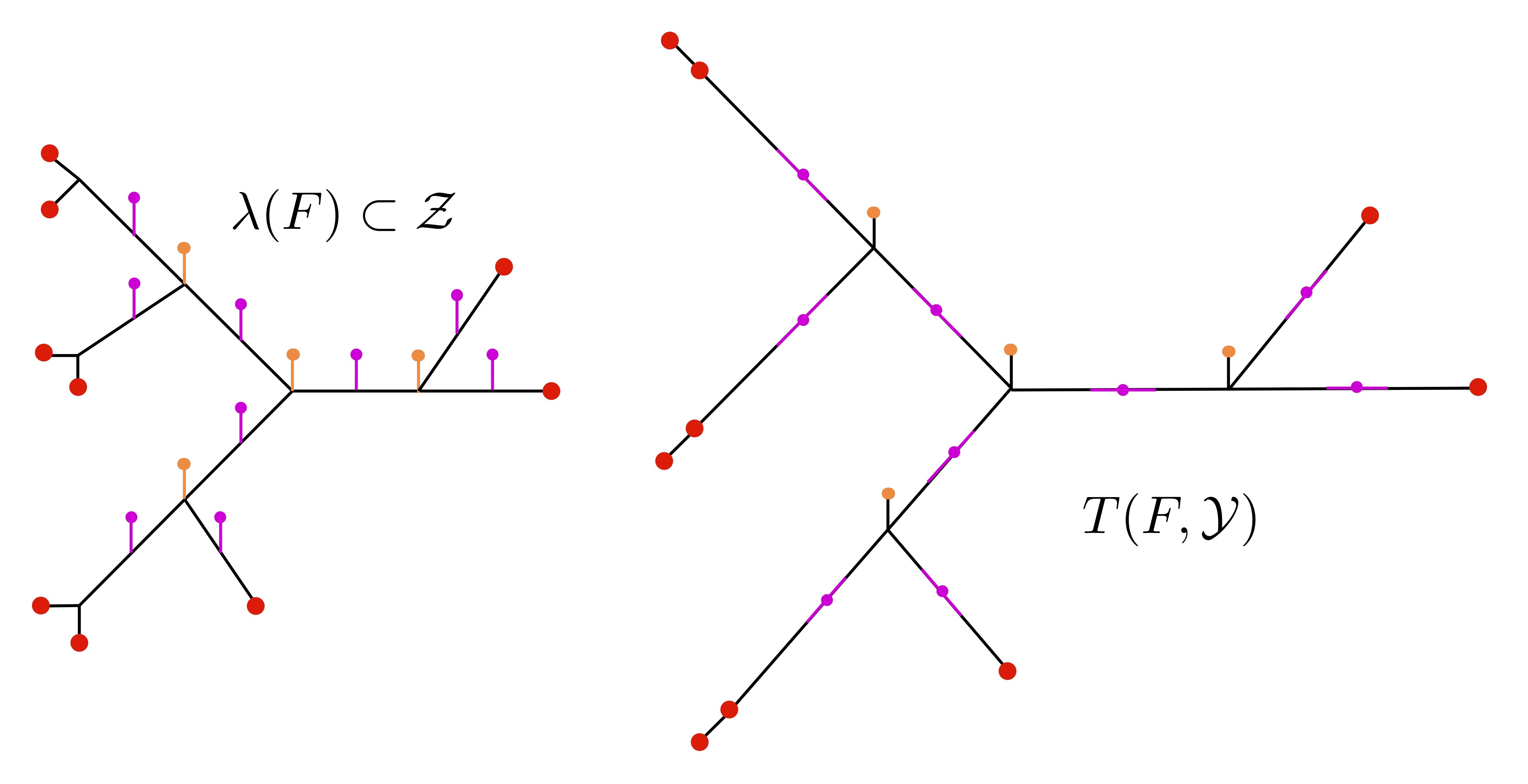}
\caption{The stable tree $T = T(F,\YY)$ may have leaves which are not points in $F$, and some points of $F$ may not be leaves of $T$.  In this example, the ambient space $\ZZ$ is the whole graph on the left, and $T$ is realized geometrically on the right.  The orange cluster points atop spikes at the branch points of the underlying tree create leaves in $T$.  New leaves in $T$ always arise from clusters near branch points of $\lambda(F)$.  The pink bivalent cluster points determine bivalent vertices in $T$, with small neighborhoods thereof folding into the spikes upon inclusion of $T\to \ZZ$.  By contrast, $\lambda(F)$ contains none of the spikes.  Finally, the pairs of nearby points of $F$ in $\lambda(F)$ on the left side of $\ZZ$ form clusters.  The components of $T_e$ connect one point from each pair to a pink cluster, while a component of $T_c$ connects the pair.  As a result, some points of $F$ are not leaves of $T$.}
\label{fig:newleaf}
\end{figure}

\medskip

Now for part (2), consider the minimal network $\mu(C)$ for the cluster $C$.
By Lemma \ref{basic tree lemma} and the definition of
shadows, $\mu(C)$ is within $O(\ep)$ of the shadow of $C$ in $\ltree(F)$, and it follows from Claim \ref{claim:shadow_close_to_cluster} (in the proof of Lemma \ref{shadow intersection}) that every
point of $s(C)$ is within $O(E)$ of $C$. This proves part (2). 

\medskip

For part (3), let $p \in \lforest(V^0) \subset T_e$, where $V\in\VV$, and let the distance
$d_\ZZ(p,\YY\cup F)$ be realized on a point in a cluster $C_1$. Write $d_\ZZ(p,C_1)=t$.

Suppose first that $C_1 \in V^0$.
The quotient of $\lforest(V^0)$ obtained by collapsing the clusters
of $V^0$ to points is a tree by minimality of the network, so 
there is some sequence of components of $\lforest(V^0)$ which connects $p$ to $C_1$, possibly through clusters $C_2, \dots, C_l \in V^0$.

Consider the unique path $\alpha$ in $\lforest(V^0)$ from
$p$ to $C_2$.  The path $\alpha$ branches at no more than $b = b(T)$ points, so let
$\alpha' \subset \alpha$ be the longest unbranched subsegment of $\alpha$.  We thus have
$|\alpha'| \ge \frac{1}{b}d_T(p, \partial \lforest(V^0))$. If $|\alpha'| > t$, we may remove
$\alpha'$ from $\lforest(V^0)$, attach a minimal length path in $\ZZ$ from $p$ to $C_1$ (of
length $t$), and obtain a network with smaller total length than $\lforest(V^0)$ that still 
connects the clusters in $V$.  This would violate the minimality of
$\lforest(V^0)$, so we must have $t\geq |\alpha'|$, and therefore
$t \geq \frac{1}{b}d_T(p, \partial \lforest(V^0))$, as required. 

Now consider the possibility that $C_1$ is a cluster outside of $V$.  Let $s(V)$ denote
the shadow of the union of clusters $s(\bigcup_{A\in V^0} A)$, which is the same as the hull in
$\ltree(F)$ of the shadows $\{s(A) : A\in V^0\}$. We claim that $s(V) \intersect s(C)=\emptyset$ for every
$C\in \GG^0\ssm V^0$.

Recall from Lemmas \ref{shadow intersection} and \ref{general separation} that the shadow
$s(C)$ for each $C\in \EE^0$ is disjoint from all other cluster shadows,
and that the separation of shadows by $s(C)$ in $\ltree(F)$ is the same as the separation of
the corresponding vertices in $\GG$ by $C\in \GG^0$.
In particular, if $V\in\VV$ and $C\in\EE^0$ then all vertices $D\in V^0$ (other than $C$
itself if $C$ happens to lie in $V^0$) have shadows
$s(D)$ on one side of $s(C)$. Any $V_1\ne V_2$ in $\VV$ are separated in $\GG$ by some
$C\in\EE^0$, including the case when $C$ is the common vertex of $V_1$ and $V_2$. Thus the
shadows $s(V_1)$ and $s(V_2)$ are either disjoint or overlap exactly on $s(C)$ for this
common vertex $C$. The claim follows.

Now applying Lemma \ref{basic tree lemma} again we find that $\lforest(V^0)$
is in an $O(\ep)$ neighborhood
of $s(V)$. Thus
if $C_1$ is a cluster in $\GG^0\ssm V^0$ then  any $\ZZ$-geodesic from $p$ to $C_1$ has an $O(\ep)$
fellow traveling path in $\ltree(F)$ which must exit $s(V)$ before it arrives at $s(C_1)$. It
follows that $d(p,C_1) \ge d(p,C_2) - O(\ep)$, for some $C_2\in V^0$. This reduces to the
previous case.
\end{proof}

We are now ready to prove that our stable tree $T(F,\YY)$ coarsely behaves like $\ltree(F)$.  Unlike Gromov's trees, stable trees quasi-isometrically embed with multiplicative constants possibly larger than $1$; see Figure \ref{fig:pathologies}.  This is an inconvenient fact for what follows and later in Section \ref{sec:stable cubulations}.

\begin{figure}
\includegraphics[width=1\textwidth]{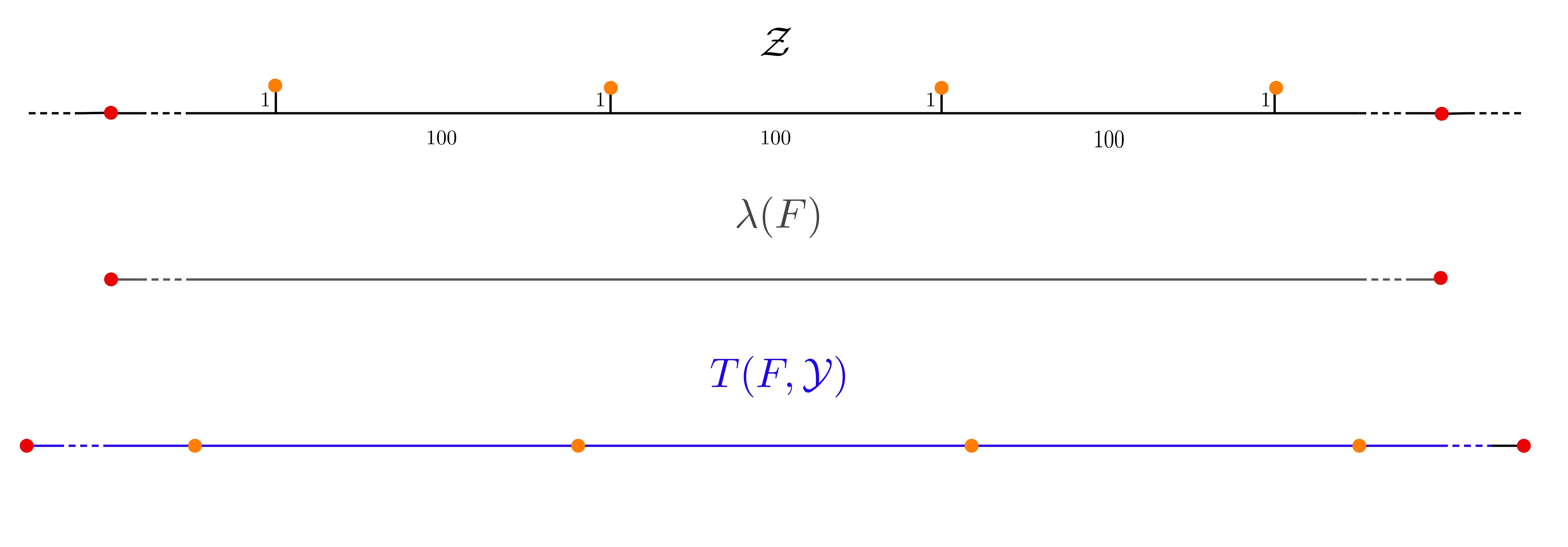}
\caption{A basic, complicating example.  The ambient space $\ZZ$ is a biinfinite line, which we think of as $10$-hyperbolic.  The cluster points $\YY$ (in orange) sit on spikes of height $1$ at distance $100$ from each other.  The two points in $F$ are very far apart.  The spanning tree $\lambda(F)$ for $F$ is a long line.  The stable tree $T(F,\YY)$ is also abstractly a long line composed of segments of length $102$.  The natural map $T(F,\ZZ) \to \ZZ$ folds the ends of these segments onto the spikes and is therefore not an embedding.  It is a quasi-isometric embedding, but the multiplicative constant is at least $\frac{102}{100}$.}\label{fig:pathologies}
\end{figure}

\begin{proposition}{the forest for the trees}
The natural map $T(F,\YY)\to \ZZ$ is a $(K_1,K_1)$--quasi-isometric embedding, and $T(F,\YY)$ lies within Hausdorff distance $K_1$ of $\ltree(F)$, where $K_1 = K_1(k,\delta, \epsilon)$.
\end{proposition}

\begin{proof}
It follows from Lemma \ref{basic tree lemma} that each component of
$T_c$ and $T_e$ is $(1,O(\epsilon))$-quasi-isometric to its shadow in $\ltree(F)$, and
moreover is within Hausdorff distance $O(\epsilon)$ of its shadow
in $\ltree(F)$. 

Consider two distinct clusters $C_1,C_2 \in \GG^0$, and their shadows.
By Lemma \ref{shadow intersection}-(2), the shadow intersection $s(C_1)\intersect s(C_2)$ has uniformly bounded
diameter. If $C_1$ and $C_2$ belong to different pieces $V_1,V_2 \in \VV$, then there is a
cluster $D\in \EE^0$ separating $V_1$ from $V_2$ in $\mathcal G$. If $D$ is not equal to either $C_i$ then
its shadow separates $s(C_1)$ from $s(C_2)$ by Lemma \ref{general separation}, and hence
the shadows are disjoint. If $C_1=D$, say, then again the shadows are disjoint, by Lemma
\ref{shadow intersection}-(3).

In particular, the clusters in $\EE^0$ have pairwise disjoint shadows, and moreover by
Lemma \ref{general separation} their separation properties in the graph $\GG$ are
preserved in $\lambda(F)$ (that is: if $C_2$ separates $C_1$ from $C_3$ in $\GG$ then
$s(C_2)$ separates $s(C_1)$ from $s(C_3)$ in $\lambda(F)$). This means that any $V\in \VV$ is associated to a complementary component $c(V)$ 
of the shadows of $\EE^0$ in $\lambda(F)$, in the following way. For every $C\in \EE_0$, any cluster in $V^0-\{C\}$ has shadow contained in one of the two components of $\lambda(F)-s(C)$, by Lemma \ref{general separation}. We let $c(V)$ be the intersection of all these components. Notice that if $V\neq V'$ then $c(V)\neq c(V')$, since in that case some $C\in\EE_0$ will separate $V$ from $V'$ in $\mathcal G$.

We now study overlaps of the shadows of the various relevant subtrees of $T$, showing that said overlaps are bounded.

Let $V\in\VV$ be the closure of a component of $\GG\ssm\EE^0$. If $\lambda_1,\lambda_2$
are distinct components of $\lambda'(V^0)$, we claim that their shadows in $\lambda(F)$ have
an intersection of bounded diameter. 

Indeed, if the shadows of $\lambda_1$ and $\lambda_2$ had overlap of size $\gg \epsilon$, then
$\lambda_1$ would contain points within $O(\epsilon)$ of $\lambda_2$, at distance $\gg
\epsilon$ from each other, and with no branch point of either $\lambda_1$ or $\lambda_2$
within $O(\epsilon)$ of the geodesic in $\lambda_1$ connecting the two points (this uses
the bound on the branching of $T$). A simple surgery would then reduce the total length of
$\lambda'(V^0)$, contradicting its minimality.

Now consider a component $\lambda_1$ of $\lambda'(V^0)$ and one of the clusters $C$ in $V$. We claim
their shadows also have bounded-diameter intersection. 
Lemma \ref{minimal networks don't overlap} tells us that any point of $\lambda_1$ within $d$
of $\mu(C)$ is within $O(d)$ of the boundary of $\lambda_1$. This proves the claim. 

Note that the number of clusters in $V$, and
therefore the number of components of $\lambda'(V^0)$, 
are bounded via Lemma \ref{clusterfacts}. Thus the subtree $T^V$ comprised of
$\lambda'(V^0)$ together with all the components of $T_c$ associated to clusters in $V$ has
a decomposition into a bounded number of subtrees, and a map to $\lambda(F)$ (using the
shadows) which is a $(1,O(\ep))$-quasi-isometric embedding on each subtree and such that
the images of distinct subtrees have bounded overlap. Under these circumstances it follows
that the map $T^V \to \lambda(F)$ is a $(k,k)$--quasi-isometric embedding, where $k$
depends on these bounds. Moreover, the image of this map must, up to bounded error, lie in
the component of $\lambda(F)$ minus the shadows of those clusters in
$\EE^0$ that separate $V$ from the rest of $\GG$ (by the preservation of separation
properties noted above).

It follows that these maps piece together to give a $(K_1,K_1)$--quasi-isometry for $K_1=K_1(k, \delta, \epsilon)$. This
completes the proof.
\end{proof}

\subsection{Proof of Theorem \ref{thm:stable tree}} \label{subsec:stable tree 1}

Property \ref{item:branching}, the branching bound on $T(F,\YY)$, was proved in Lemma \ref{minimal networks
don't overlap}.

Property \ref{item:qi}, the quasi-isometry, is given by Proposition \ref{the
  forest for the trees}.
  
Property \ref{item:qi2} follows from the construction of $T_e$ and Lemma \ref{basic tree lemma}.

Regarding property \ref{item:bijection}, we have a natural bijection between components of $T_c$ and clusters by construction, and each component $\mu(C)$ is contained in a controlled neighborhood of the corresponding cluster $C$ by Lemma \ref{minimal networks don't overlap}, where by controlled we mean that the corresponding constant depends on $\delta,\epsilon,$ and $E$. We are left to argue that $C$ lies in a controlled neighborhood of $\mu(C)$. This is equivalent to showing that $s(C)$ lies in a controlled neighborhood of $\mu(C)$. If this was not true then, in view of the bound on the total branching of $\lambda(F)$, we would have that $s(C)$ contains an interval $I$ in an edge of $\lambda(F)$ of length $\gg E$ not contained in a controlled neighborhood of $\mu(C)$. From Proposition \ref{the forest for the trees} we know that $T$ lies within controlled Hausdorff distance of $\lambda(F)$, so that $I$ is contained in a union of controlled neighborhoods of the $\mu(C')$ for $C'\neq C$, and controlled neighborhoods of the components of $T_e$. Neighborhoods of the latter type cannot contain points in $I$ far from its endpoints by Lemma \ref{minimal networks don't overlap}-(3), and the same holds true for neighborhoods of the former type in view of Lemma \ref{shadow intersection}-(1), a contradiction. Therefore, $s(C)$ and $C$ are contained in a controlled neighborhood of $\mu(C)$, as required.

We now prove the ``furthermore'' part of the statement. Recall that, for the reasons explained in Remark \ref{rem:equivariance}, we only treat the case that $g$ is the identity.
Let 
$(F',\YY')$ be a second configuration differing from $(F,\YY)$ as in the
statement. We name the constructions arising from $(F',\YY')$ by
$\CC'$, $\GG'$, $\EE'$, $\VV'$, etc. Also, we denote $T=T(F,\YY)$ and $T'=T(F',\YY')$, and similarly for $T_c,T_e$, $T'_c,T'_e$. Set $N = \#\left(\YY \symdiff \YY'\right)$. 

\

\textbf{Claim 1}: The cardinality $\left|\GG^0 \bigtriangleup (\GG')^0\right|$ is bounded in terms of $k$, $\delta$, and $N$.

\begin{proof}
A cluster $C$ is in the symmetric difference $\GG^0 \bigtriangleup (\GG')^0$
only if it is within $E$ of a point of $F\union F'\union (\YY
\bigtriangleup  \YY')$, of which there are at most $2k+2 + N$. 
Now each point of a cluster in $\CC$ is within $\ep$ of some
point in $\ltree(F)$, and there is a number $R$ depending only on the
total branching of $\ltree(F)$ such that among any $R$ points in $\ltree(F)$ within a ball of radius
$(E+\ep)$ (in $\lambda(F)$), there must be two which are less than $\ep$
apart (and the same is true for $\CC'$ and $\ltree(F')$). Thus if there are more than $(2k+2 + N)R$ elements
in $\GG^0 \bigtriangleup (\GG')^0$  then two are closer than $E$ apart,
which is a contradiction, proving Claim 1.
\end{proof}

\textbf{Claim 2}: The symmetric difference of the edge sets of
$\GG$ and $\GG'$ has cardinality bounded in terms of $k$, $\delta$, and $N$.

\begin{proof}
By Lemma \ref{clusterfacts}, the maximal valence of any vertex of
$\GG$ is bounded, and so the number of edges incident to elements
of $\GG^0 \bigtriangleup (\GG')^0$ is bounded. Therefore it suffices to
consider the case where
$C,D \in \GG^0 \cap (\GG')^0$ with $[C,D]$ an edge in $\GG$ but not in
$\GG'$.  This implies there is a $B' \in (\GG')^0 \ssm \GG^0$ separating $C$
from $D$ when no such cluster in $\GG^0$ did so before.

Since $B'$ separates $C$ from $D$, 
there is a point $q\in B' \cap \left(\left(\YY'\symdiff \YY\right) \union \left(F' \symdiff F\right)\right)$
which lies at distance at most $\ep$ from a $\ZZ$-geodesic $\gamma$ joining
$C$ and $D$. The shadow $s(q)$ on $\ltree(F)$ must therefore be in a
$10\ep$-neighborhood of the interval in $\ltree(F)$ between $s(C)$ and $s(D)$.

Each such $q$ can only affect a bounded number of such edges $(C,D)$ in this
way, because the shadows of edges in each component of $\EE$ are arranged sequentially
and disjointly along edges of $\ltree(F)$ by Lemma \ref{components of E}, and $\GG^0\ssm \EE^0$ is bounded (again by Lemma \ref{clusterfacts}).
Since there are only boundedly many such $q$, this bounds the number of edges in the symmetric difference, completing the proof of Claim 2.
\end{proof}

\medskip

Let $\pi_0(\TT)$ denote the set of components of a forest $\TT$. Since the components of $T_c$
and $T'_c$ correspond to the elements of $\GG^0$ and $(\GG')^0$ respectively, Claim 1 gives us a
bound on $|\pi_0(T_c) \symdiff \pi_0(T'_c)|$. 

Similarly, by Claim 2, there is a bound on the number of collections of clusters in $\VV \symdiff
\VV'$, and this gives us a bound, say $K$,
on $|\pi_0(T_e)\symdiff \pi_0(T'_e)|$.

We now increase $K$ in a controlled way a few times, with the result of each step depending only on $k, \delta, \epsilon$. 

By item \ref{item:qi}, we can increase $K$ to ensure that $d_{Haus}(T,T') < K$. By Lemma \ref{lem:branch_close} (below) we can further assume that $d_{Haus}(\mathcal B\cup T_c\cup F,\mathcal B'\cup T'_c\cup F')\leq K$, where $\BB$ and $\BB'$ are the sets of branch points of $T$ and $T'$. We have to be careful in using \ref{lem:branch_close} because the sets of leaves of $T$ and $T'$ need not be within bounded Hausdorff distance of each other, since they might contain more than $F$ and $F'$ (see Figure \ref{fig:newleaf}). However, we can apply the lemma after slightly modifying $T$ and $T'$ by adding spikes of length, say, $1$ to $T$ and $T'$ to ensure that the sets of leaves of the new trees that we obtain do lie within controlled Hausdorff distance. Such spikes only need to be added close to $T_c$ and $T'_c$ by part (1) of Lemma \ref{minimal networks don't overlap}, yielding the required Hausdorff distance estimate.  We also note that the number of spikes added is controlled by Claim 1. 

We can then increase $K$ once more to ensure that $T_c\cup (\YY'-\YY)$ and $T'_c\cup(\YY-\YY')$ also lie at Hausdorff distance bounded by $K$; this can be done since $T_c, \YY$ and $T'_c, \YY'$ are at bounded Hausdorff distance by Lemma \ref{minimal networks don't overlap}(2).  Finally, we also require that $K> K_1$ as in Proposition \ref{the forest for the trees}.

Now let $\sigma = \sigma(K, \delta)$ be the fellow-traveling constant for $(1,K)$-quasigeodesics with endpoints at distance at most $K$ in a $\delta$-hyperbolic space.  This constant will be relevant later because geodesics in our trees $T,T'$ are $(1,K)$-quasigeodesics in $\ZZ$ by item \ref{item:qi} and our choice of $K$.

For ease of notation, we will refer to components in $\mathcal U = \pi_0(T_{e}) \cap \pi_0(T'_{e})$ as ``unchanged'' components, and the remaining components as ``changed''.  We note that there are at most $K$ changed components in each of $T_{e}, T'_{e}$.

For each component $E \in \pi_0(T_e) \ssm \pi_0(T'_e)$, let $E_{\YY'} = hull_E(E \cap N_K(\YY'))$, where the hull of this intersection is taken in the tree $E$, while the neighborhood is taken in $\ZZ$.  Now define  $\displaystyle \RR = \bigcup_{E \in \pi_0(T_e) \ssm \pi_0(T'_e)} E_{\YY'}$, and define $\RR'$ similarly.
We now collect the ``unstable parts'' of the trees along with the unchanged parts; set $$U = T_c \cup \RR \cup \UU \cup \mathcal B \cup F\indent \text{ and } \indent U' = T'_c \cup \RR' \cup \UU\cup \mathcal B'\cup F'.$$ 
Note that we included $F$ and $F'$ in these sets to ensure that they lie within bounded Hausdorff distance of each other, but this is inconsequential for the purposes of considering the complementary forests, which is what we want to do next.

Let $K'=K'(K,N,k)\geq 10K$ so that the following hold:
\begin{itemize}
\item[(i)] $d_{Haus}(T, T')<K$ and $d_{Haus}(U, U') < 10K$,
\item[(ii)] $\left| \# \pi_0(T\ssm U)\right|< K'$ and $\left| \# \pi_0(T'\ssm U')\right| < K'$,
\item[(iii)] $(T \ssm U) \subset T_e$ and $(T' \ssm U') \subset T'_e$, and
\item[(iv)] $\left| \# \pi_0(T_e \cap U)\ssm\UU\right|< K'$ and $\left| \# \pi_0(T'_e \cap U'))\ssm \UU'\right|< K'$, and
\item[(v)] for any $C \in (\pi_0(T_e \cap U)\cup \pi_0(T'_e \cap U'))\ssm \UU$, we have $\diam C < K'.$
\end{itemize}

Regarding property (i), the ``$10$'' in ``$10K$'' is there to keep into account the fact that we took hulls.

Property (ii) can be shown observing that the intersection of $U$ with each changed component can only have a bounded number of components because of the bound $N$ on $|\YY\symdiff \YY'|$ and the bound on $|\BB|$ given by Lemma \ref{minimal networks don't overlap}. This same observation shows item (iv).  Property (iii) holds by construction.

Property (v) is non-trivial for $\RR$ and $\RR'$, in which case it holds since $C$ is a union of boundedly many components, each of bounded diameter. In particular, any set $E_{\mathcal Y'}$ is a union of hulls in $E$ of intersections with balls centered at an element of $\YY'$, which have bounded diameter, and the union consists of at most $N$ elements.

Now let $L_1 = L_1(K', \delta), L_2 = L_2(K', \delta)>0$ be the constants given by Lemma \ref{lem:finding close components} (below) with $T,U$ and $T',U'$ satisfying the conditions of that lemma via (i) and (ii).

Let $\LL \subset \pi_0(T \ssm U)$ be the set of all components of $T \ssm U$ of diameter greater than $L_1$.

Define $\LL' \subset \pi_0(T' \ssm U')$ to be the set of components of $T'\ssm U'$ that lie within Hausdorff distance $L_2$ of an element of $\LL$.  Since there are at most $K'$ components of $T\ssm U$, this bounds the cardinality of $\pi_0(T \ssm U) \ssm \LL$, and any component in this set has diameter at most $L_1$.  Similarly, the numbers and diameters of the components of $(T' \ssm U')$ not in $\LL'$ are also bounded, this time the bound on the diameter being $L_3$ by the moreover part of Lemma \ref{lem:finding close components}.

Lemma \ref{lem:finding close components} provides a bijection $\rho: \LL \to \LL'$ which sends any component in $\LL$ to the unique component in $\LL'$ which is within Hausdorff distance $L_2$.
That is, for any $C \in \pi_0(\LL)$, we have $d_{Haus}(C, \rho(C)) < L_2$.

Now set $T_{s,0}$ to be the union of all elements of $\LL \cup \UU$, and define $T'_{s,0}$ similarly. We observe that we have by construction and (iii) above that $T_{s,0} \subset T_e$ and $T'_{s,0} \subset T'_e$.  Moreover, we have that the number of components of $T_e \ssm T_{s,0}$ and their diameters are bounded by $2K'(L_1 + 10L_2+K'+1)$, and similarly for $T'_e \ssm T'_{s,0}$. In fact, the number of such components is bounded by $2K'$, since each is a union of 

\begin{itemize}
\item changed components of $T_e\cap U$, and there are at most $K'$ of those by (iv), and
\item components of $T\ssm U$ of diameter at most $L_1$, and again there are at most $K'$ of those.
\end{itemize}

The bound on the diameter also follows from this description.

To obtain the sets $T_s$ and $T'_s$ required by the theorem, it suffices now to remove the branch points from the unchanged components contained in $T_{s,0}$ and $T'_{s,0}$ to ensure (1) (which at this point is not satisfied only because of the unchanged components, since we included the branch points in $U$ and $U'$), while all other properties have been checked above.\qed

\subsubsection*{Two supporting lemmas}
The following two lemmas were used in the proof of Theorem \ref{thm:stable tree} above.  To simplify notation, we will not distinguish between a tree quasi-isometrically embedded in a metric space, and the image of said tree.

\begin{lemma}{lem:branch_close}
  For each $K, \delta$ there exists $L_0$ such that the following holds. 
   Let $T,T'$ be trees $(K,K)$-quasi-isometrically embedded in the $\delta$-hyperbolic metric space $\ZZ$, with $d_{Haus}(F_0,F'_0)\leq K$, where $F_0$ and $F'_0$ are the sets of leaves of $T$ and $T'$ respectively.
   Then the sets of branch points $\BB$ and $\BB'$ of $T$ and $T'$ satisfy $d_{Haus}(\BB\cup F_0,\BB'\cup F'_0)\leq L_0$
  \end{lemma}
  
  \begin{proof}
   The set $\BB\cup F_0$ can be coarsely characterized as the set of points $x$ of $T$ so that there are $f_1,f_2,f_3$ in $F_0$ (not necessarily distinct) with the property that the Gromov product at $x$ between any $f_i$ and $f_j$ is small, and similarly for $T'$.  We leave the details to the reader.
  \end{proof}

 \begin{lemma}{lem:finding close components}
  For each $K$ there exist $L_1,L_2, L_3$ so that the following holds. 
   Let $T,T'$ be trees $(K,K)$-quasi-isometrically embedded in the metric space $\ZZ$, with $d_{Haus}(T,T')\leq K$. Also, let $U\subseteq T$, $U'\subseteq T'$ be sub-forests so that $d_{Haus}(U,U')\leq K$, and so that all branch points of $T$ (resp. $T'$) are contained in $U$ (resp. $U'$).
   
   Then for each component $C$ of $T\ssm U$ of diameter at least $L_1$ there exists a unique component $C'$ of $T'\ssm U'$ within Hausdorff distance $L_2$ of $C$. Moreover, every component $C'$ of $T'\ssm U'$ of diameter at least $L_3$ arises in this way. 
  \end{lemma}
  
  \begin{proof}
    We will conflate components of $T\ssm U$ with their closures, so we can talk about their leaves, and similarly for $T'\ssm U'$.
   
   The main observation is that there exists $K_2=K_2(K)$ such that the following holds. Let $C$ be a component of $T\ssm U$ and let $x,y$ be (not necessarily distinct) points in $C$ that, in the metric of $C$, are at least $K_2$ from all the leaves of $C$. Then there exists a unique component $C'$ of $T'\ssm U'$ which is within $K$ of both $x$ and $y$.
   
   To prove this, suppose by contradiction that $K_2\gg K$ and that there are distinct components of $T'\ssm U'$ that contain points $x'$ and $y'$ that are within $K$ of $x$ and $y$, respectively. Then there exists some $p'\in U'$ on the geodesic $[x',y']$ in $T'$ from $x'$ to $y'$. Let $p\in U$ be such that $d(p, p')<K$. Then $p$ lies within $\sigma=\sigma(K)>0$ from the geodesic from $x$ to $y$, since considering points in $T$ that are within $K$ of those along $[x',y']$ yields a quasi-geodesic in $T$. Since $x$ and $y$ are at least $K_2$ from the leaves of $C$, they cannot lie close to any point of $U$, in particular $p$. We can then deduce that either $p$ lies along the geodesic $[x,y]$ in $T$ from $x$ to $y$, or there is a branch point of $T$ along $[x,y]$. In either case, $x$ and $y$ do not lie in the same component of $T\ssm U$, a contradiction. (Recall that $U$ contains all branch points of $T$ by hypothesis.)

   Consider now a component $C$ of $T\ssm U$ of diameter sufficiently large that it contains a point which is at least $K_2$ from all the leaves of $C$. By the observation above, all such points are close to a unique component $C'$ of $T'$, and since the set of all such points has bounded Hausdorff distance from $C$, we have that $C$ is contained in a uniform neighborhood of $C'$. Moreover, if $C$ has sufficiently large diameter, then we can apply the same reasoning to $C'$ and deduce that $C'$ contains points that are within $K$ of a unique component $C''$ of $T-U$, and that $C'$ is contained in a uniform neighborhood of $C''$. But the above observation implies that $C''=C$, and it follows that $C$ and $C'$ lie within uniformly bounded Hausdorff distance.
   
   Finally, the moreover part follows from a similar back-and-forth using the previous part of the statement. Namely $C'$ has sufficiently large diameter and is within bounded Hausdorff distance of a component $C$ of $T\ssm U$, then $C$ also has large diameter and is thus in turn within bounded Hausdorff distance of some component of $T'\ssm U'$, which needs to be $C'$.
  \end{proof}

\section{Stable cubulations}\label{sec:stable cubulations}

Fix a $G$-colorable HHS $(\XX, \mathfrak S)$ for $G < \mathrm{Aut}(\mathfrak S)$, and let
$F \subset \XX$ be a finite set.

In this section, we use the stable trees constructed in Section \ref{sec:stable trees} for the projections of $F$ to the relevant domains to define a wallspace on the hull $H_\theta(F)$. This wallspace can then be plugged into Sageev's machine to produce a cube complex which, by an argument from \cite{BHS:quasi}, coarsely models the hull of $F$ in $\XX$.  Stability of the tree construction then induces stability in the cubulations under perturbations of $F$.  We refer the reader to Subsection \ref{subsec:cube complexes} for some background and references on cube complexes, wallspaces, and hyperplane deletions, as well as to Subsection \ref{subsec:HHSery} for background on HHSs.

\par\smallskip
The main result of this section, and the only statement from this section that we will use in the rest of the paper, is the following precise version of Theorem \ref{thm:stable cubulation informal}:

\begin{theorem}{thm:stable_cubulations}
  Let $(\XX,\mathfrak S)$ be a $G$-colorable HHS for $G < \mathrm{Aut}(\mathfrak S)$.
 Then for each $k$ there exist $K, N$ with the following properties.  To each subset $F\subseteq \XX$ of cardinality at most $k$ one can assign a triple $(\mathcal Y_F,\Phi_F,\psi_F)$ satisfying:
 \begin{enumerate}
  \item $\mathcal Y_F$ is a CAT(0) cube complex of dimension at most the maximal number of pairwise orthogonal domains of $(\XX,\mathfrak S)$,
  \item $\Phi_F:\mathcal Y_F \to H_\theta(F)$ is a $K$--median
    $(K,K)$--quasi-isometry,

  \item $\psi_F:F\to\mathcal (\QQ_F)^{(0)}$ satisfies $d_{\cuco X}(\Phi_F \circ \psi_F(f), f) \leq  K$ for each $f\in F$.
  \end{enumerate}
Moreover, suppose that $F'\subseteq \XX$ is another subset of cardinality at most $k$, $g\in G$,
and $d_{Haus}(gF,F')\leq 1$. Choose any map $\iota_{F}:F\sqcup F'\to F$ so that $\iota_{F}(f)=f$ if $f\in F$ and $d_\XX(g(\iota_{F}(f)),f)\leq 1$ if $f\in F'$. Also, choose a map $\iota_{F'}:F\sqcup F'\to F'$ such that $\iota_{F'}(f)=f$ if $f\in F'$ and $d_\XX(g(f),\iota_{F'}(f))\leq 1$ if $f\in F$. Then the following holds:

There is a third CAT(0) cube complex $\mathcal Y_0$ and $K$--median
$(K,K)$--quasi-isometric embedding $\Phi_0$ such that the diagram
  \begin{equation}\label{eq:Phi diagram}
  \begin{tikzcd}
    & F\arrow[r,"\psi_{F}"]&\QQ_{F} \arrow[dr,"g\circ\Phi_{F}"] \arrow[d,"\eta" left] &  \\
    F\sqcup F' \arrow[ur,"\iota_{F}"]\arrow[dr,"\iota_{F'}\ \ \ " below] & &\QQ_0 \arrow[r,"\Phi_0"] & \XX \\
    & F'\arrow[r,"\psi_{F'}"]&\QQ_{F'}\arrow[ur,"\ \ \ \Phi_{F'}" below] \arrow[u,"\eta'"] & \\
  \end{tikzcd}
  \end{equation}
commutes up to error at most $K$, where $\eta$ and $\eta'$ are hyperplane
deletion maps that delete at most $N$ hyperplanes. The left side commutes exactly, that is, we have $\eta\circ\psi_F\circ\iota_F=\eta'\circ\psi_{F'}\circ\iota_{F'}$.
\end{theorem}

The notion of $K$--median refers to the coarse median structure on $\cuco{X}$ in the sense
of \cite{Bow:coarse_median}, and we only define it later where it is needed, since we
obtain it directly from \cite{BHS:quasi}.

 The remainder of this section is devoted to the proof of this theorem.

\subsubsection*{Standing assumptions} Throughout this section, we fix an HHS $(\cuco{X},\mathfrak S)$ as in Theorem \ref{thm:stable_cubulations}, which in light of Theorem \ref{thm:BBFS} we can assume to have stable projections, and furthermore we can assume that it has the property that all $\pi_V(x)$ and all $\rho^U_V$ for $U\propnest V$ are single points, see Remark \ref{rem:rho_can_be_points}.

\subsection{Subdivision sets for stable trees} \label{subsec:subdivision}

In this subsection, we establish a formalism for subdividing trees.  In Subsection \ref{subsec:constr}
 these subdivision points will give the walls in our cubulation.  This mostly follows the strategy of \cite{BHS:quasi}, except that we need to take greater care in making choices for the subdivision.

\begin{definition}{defn:subdivision}
Let $M'>M>0$.  An $(M, M')$-\textbf{\emph{subdivision}} of a tree $T$ is a collection of points $\pt(T) \subset T$ satisfying:

\begin{enumerate}
\item The points $\pt(T)$ are contained in the interiors of edges of $T$.  We set $\pt(e) = \pt(T) \cap int(e)$ for each edge $e$ of $T$.
\item The $\frac{M'}{2}$-neighborhood of $\pt(e) \cup \partial e$ covers $e$.
\item All points of $\pt(e) \cup \partial e$ are at least $M$ apart in $e$.
\end{enumerate}

In other words: The spacing between points of $\pt(e)\union \partial e$ along $e$ is at least $M$ and strictly less than $M'$.

We additionally say that $\pt(T)$ is \textbf {\emph{ $(M,M')$-evenly-spaced}} if $M' \ge 8M$ and the spacing between successive points of $\pt(e)$ is exactly $M$ for each edge $e$.

\end{definition}

We will specify $M, M'$ later, though for now it suffices to assume they are large relative to the various HHS constants.

We fix, once and for all, a \textbf{subdivision operator} $\wp_{M,M'}$ which to any tree $T$
associates a fixed $(M,M')$-evenly spaced subdivision $\wp_{M,M'}(T)$ of (the edges of)
$T$.  Often the constants $M, M'$ are fixed, and we simply write $\wp(T)$. Similarly we
can define $\wp_{M,M'}$ on a forest as the union of subdivisions $\wp_{M,M'}$ on its components.

We now explain how to associate a collection of subdivisions on stable trees to a set of points $F \subset \cuco X$.

Fix some large $K$ (depending on $M$, to be specified later) and write $\UU(F) = \rel_K(F)$ for any finite set $F \subset \XX$.

For any $V\in\mathfrak S$, define
$$
\YY^V = \{\rho^W_V: W\in\UU(F), W\propnest V \}, 
$$

and
$$
F^V = \pi_V(F).
$$
Note that whenever $K$ is larger than the bounded Geodesic image constant $\kappa_0$, we have that $\YY^V$ is contained in the $\kappa_0$--neighborhood of the hull of $F^V$ in $\mathcal C(V)$. This ensures that $\YY^V$ satisfies the requirements of Theorem \ref{thm:stable tree} for any $\ep\geq 2 \kappa_0$. We fix such $\epsilon$ as in Theorem \ref{thm:stable tree}, and we will always apply that theorem with this $\ep$.

Note that if $V\notin \UU(F)$ then $\YY^V\union F^V$ has uniformly bounded diameter.

Let $V \in \UU(F)$.  From $V$, we get corresponding sets of projections $F^V, \YY^V$ in $\CC(V)$, which we may consider independently of the set $F$.  Doing so, we obtain from Theorem \ref{thm:stable tree} a fixed
stable tree
$$T^V_F := T(F^V,\YY^V)$$
and we denote its decomposition $T_c^V \union T_e^V.$

\begin{figure}
\includegraphics[width=0.75\textwidth]{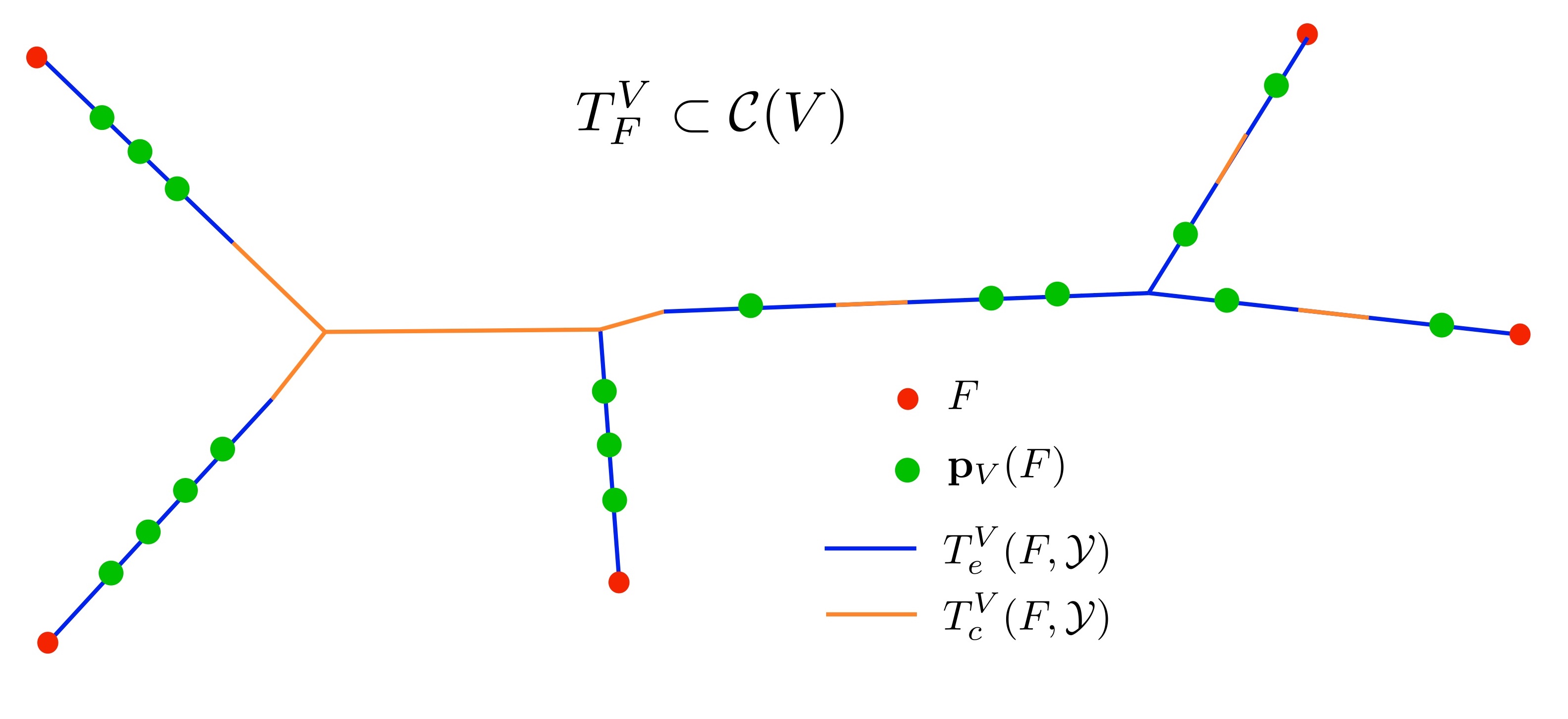}
\caption{An evenly-spaced subdivision of the stable tree $T^V_F$.  Note that subdivision points are far away from leaves, components of $T_c$, and branch points.}\label{fig:subdivision}
\end{figure}

Let $M'>8M>0$ be subdivision constants (to be specified later).
Applying $\wp_{M,M'}$ to each forest $T^V_e$ in
$T^V_F$, we call the resulting subdivision $\wp^V(F)$.

\begin{rem}
 We emphasize that the distance between the subdivision points is measured in the various trees themselves, not in the corresponding $\CC(V)$ where the trees quasi-isometrically embed. This is something that will require care throughout this section, but the fact that the trees are quasi-isometrically embedded in the corresponding $\CC(V)$ will ensure that no real issues arise.
\end{rem}

 The disjoint union over all $V\in\UU(F)$
gives us
$$\wp \equiv \wp(F) = \bigsqcup_{V \in \UU(F)} \wp^V(F).$$

More generally, we will consider $(M,M')$ subdivisions of the forests $T^V_e$ which are
not necessarily obtained from our subdivision operator $\wp_{M,M'}$, and in particular may
not be evenly spaced (they will arise by taking subsets). In this case if we name the full
configuration $\pt$ we will again denote the restriction to $T^V_e$ as $\pt_V(F)$.

\begin{rem}\label{rem:sub points and edges}
Note that our choice of constants $M' > 8M$ give that any point $p \in \wp^V(F)$ is at least distance $4M$ from any  point $\rho^W_V\in\YY^V$ for $W \nest V$ or leaf $f \in F^V$.
\end{rem}

 \subsection{The cube complex}\label{subsec:constr}

Following the scheme of \cite{BHS:quasi}, given a  union $\pt$ of $(M',M)$ subdivisions
as above we now describe a cube complex $\QQ_{F,\pt}$ and a map
$$
\Phi_{F,\pt}: \QQ_{F,\pt} \to \XX
$$
which turns out to be a quasi-isometric embedding whose image is within bounded Hausdorff distance of $H_{\theta}(F)$, when the constants $M,M',K$ are chosen suitably (see below).

To build the cube complex $\QQ_{F,\pt}$, we build a wallspace
structure on the hull $H_\theta(F)$, in which each $p\in \pt$
corresponds to a wall, i.e. a partition of $H_\theta(F)$ into two
sets.

 For each $V\in\mathfrak S$, we let $\beta^V_F:\CC(V)\to T^V_F$ be, roughly, a closest-point projection map to $T^V_F$, and more precisely any fixed map so that $d_V(x,\Xi(\beta^V_F(x)))$ is minimal for all $x\in \CC(V)$, where we recall that $\Xi$ is the quasi-isometric embedding of $T^V_F$ in $\CC(V)$.  With slight abuse of notation, for $x\in \XX$ we write $\beta^V_F(x)$ for $\beta^V_F(\pi_V(x))$.
 
 \begin{rem}
  To save notation, we will often omit the map $\Xi$, thereby identifying points and subsets of some $T^V_F$ with their image in $\CC(V)$ (as one often does with quasi-geodesics). For example, for $p\in \beta^V(F)$ and $x\in \CC(V)$ we will write $d_V(\beta^V_F(p),x)$ rather than $d_V(\Xi(\beta^V_F(p)),x)$. 
 \end{rem}
 
 Given $p\in\pt_V(F)$,  let $T^V_{p,+}$ denote one of the components of $T^V_F-\{p\}$, and let $T^V_{p,-}$ be the union of the other component and $\{p\}$ (we
arbitrarily choose the first component).  Let
$$ W^V_{p,\pm}=(\beta^V_F)^{-1}(T^V_{p,\pm})\cap H_{\theta}(F).$$
Note that $W^V_{p,+}$ and $W^V_{p,-}$ form a partition of $H_\theta(F)$. 
Let $\mathcal L^V_p=\{W^V_{p,+},W^V_{p,-}\}$ be the wall associated to
$p$. We call $T^V_{p,\pm}$ the \textbf{half-trees} associated to the wall $\mathcal L^V_p$.

Let $\QQ=\QQ_{F,\pt}$ be the CAT(0) cube complex dual to the wallspace $\{\mathcal L^V_p\}$.

\bigskip

To define $\Phi_F = \Phi_{F,\pt}:\cuco Y\to \cuco X$, note that it suffices to define
$\Phi_F$ on the $0$--skeleton of $\cuco Y$.  Let $x\in\cuco Y^{(0)}$;
we view $x$ as a coherent orientation of the walls $\mathcal
L^V_p$ (see Subsection \ref{subsec:walls}). That is, for each $p\in\pt$, we have
$x(p)$ equal to one of $ W^V_{p,+}$ or $W^V_{p,-}$.

Coarsely, we would like to define $\Phi(x)$ by
$$
\Phi(x) \approx \bigcap_{p\in\pt} x(p).
$$
This is done in \cite{BHS:quasi} by considering the projections of
$x(p)$ to the factor trees $T^Y_F$. That is, we set
$$
S_{p,Y}(x) =  {\rm hull}\left(\beta^Y_F(x(p))\right)
$$
where $Y\in \mathfrak S$ and hull denotes convex hull in the tree $T^Y_F$.  
Note that $p\in\pt_V(F)$ for some $V$, which is typically different from $Y$. 

We now define the intersection 
\begin{equation}\label{eq:bY def}
  b_Y=b_Y(x)=\bigcap_{p\in\pt} S_{p,Y}(x).
\end{equation}
These  $b_Y$ will serve as (coarse) coordinates for the map $\Phi$, as given by the theorem that we state below, after we explain how it can be extracted from \cite{BHS:quasi}.

\subsubsection*{Remarks on the construction in \cite{BHS:quasi}}
 We now summarize various results proven in \cite{BHS:quasi} regarding cubulations of hulls in HHSs. We note that the construction in \cite{BHS:quasi} of the CAT(0) cube complexes approximating hulls is the same as the one we just described, with one difference. The difference is that, rather than using $T^V_F$ to approximate the hull of $\pi_V(F)$ in $\CC(V)$, in \cite{BHS:quasi} the authors use any choice of tree contained in $\CC(V)$ that uniformly approximates the hull and is quasi-isometrically embedded with multiplicative constant 1. This choice is not due to the fact that the multiplicative constant being 1 is needed for the arguments, but more simply due to the fact that one such tree exists, and so it is more convenient to use it. For our purposes, we have to be more careful in the choice of the tree, and as a result we cannot guarantee that the multiplicative constant is 1 with our construction (recall Figure \ref{fig:pathologies}). However, this does not affect the arguments of \cite{BHS:quasi}, except that the subdivision constant $M$ has to be chosen large compared to the quasi-isometric embedding constants, so that the subdivision points are sufficiently far apart in the various hyperbolic spaces $\CC(V)$.

Another remark about the statement below is that the constant dependencies that we give below are not explicitly stated in \cite{BHS:quasi}, but can be recovered as follows. In \cite{BHS:quasi}, the constant $M$ is chosen large compared to various HHS constants and $k=|F|$ throughout Section 2, so that the construction has all the stated properties for any sufficiently large $M$. Regarding $M'$, in \cite{BHS:quasi} it is taken to be $10Mk$, as can be seen from point (4) of the construction of the walls in Subsection 2.1. The reason for the constant $10Mk$ is that one can choose subdivision points that make the diameter of the complementary components at most that quantity, but with any other bound one would obtain the same properties (e.g., that the CAT(0) cube complex quasi-isometrically embeds in the HHS), with different constants. Regarding $K$, in \cite{BHS:quasi} it is chosen to be $100Mk$, see again Subsection 2.1, and similar remarks to those regarding $M'$ apply.

\subsubsection*{Properties of the cubulation}

With this in mind, we now state various results about the construction we explained above, and point out where the arguments for those can be found in \cite{BHS:quasi}.

\begin{theorem}{thm:qi_to_cube_cplx}
Given an HHS $(\cuco{X},\mathfrak S)$ and an integer $k$, there exist $M_0\geq 1$ and functions $M'_0:\mathbb R\to \mathbb R$ and $K_0:\mathbb R\to \mathbb R$ with the following property. Whenever $M\geq M_0$, $M'\geq M'_0(M)$, and $K\geq K_0(M)$, there exists $\xi$ so that
for every  $F\subseteq \cuco X$ with $|F|\leq k$, 
the following hold: 
\begin{enumerate}
 \item\cite[Lemma 2.6, paragraph ``Definition of $\mathfrak p_A$'' in proof of Theorem 2.1
 ]{BHS:quasi} For every $x\in \QQ_{F, \pt}^{(0)}$, $b_Y(x)$ is non-empty and
   $$\diam_{\mathcal C(Y)}(b_Y(x))\leq \xi.$$
 \item \cite[Lemma 2.7, paragraph ``Definition of $\mathfrak p_A$'' in proof of Theorem 2.1]{BHS:quasi}\label{item:realization} For every $x\in \QQ_{F, \pt}^{(0)}$, there exists a point in $H_\theta(F)$, denoted $\Phi(x)$, whose
projections to all $\mathcal C(Y)$ are within distance $\xi$ of
$b_Y(x)$.
\item \cite[Theorem 2.1]{BHS:quasi}$\Phi$ is a $\xi$-median $(\xi,\xi)$-quasi-isometry to $H_\theta(F)$. \label{item:median map}
\item \cite[Theorem 2.1]{BHS:quasi} The dimension of $\QQ_{F, \pt}$ is bounded by the maximal number of pairwise orthogonal domains in $\mathfrak S$.
\end{enumerate} 
\end{theorem}

We will also need some more technical properties of the trees $T^V_F$ and projections $\rho^U_V$ related to the HHS consistency axioms.
\begin{proposition}{prop:rhos_close_to_leaves}
 Given an HHS $(\cuco{X},\mathfrak S)$ and an integer $k$ there exists $\kappa$ so that
 given $K$  sufficiently large (depending only on $(\XX,\mathfrak S)$) we have: 
 \begin{enumerate}
  \item \cite[Lemma 2.3]{BHS:quasi} if $U,V\in\UU(F)$ and $U\transverse V$, then $\rho^U_V$ lies $\kappa$-close in $\CC(V)$ to a point of $\pi_V(F)$.
  \item  \cite[Lemma 2.5]{BHS:quasi} if $U,V\in\UU(F)$, $V\propnest U$, and $q\in \pt(U)$, then $\rho^U_V(q)$ lies $\kappa$-close in $\CC(V)$ to a point of $\pi_V(F)$.
 \end{enumerate}

\end{proposition}

\begin{rem}
 In the rest of this section, whenever we use constants $M,M',K$ we will assume that they are chosen as as in Theorem \ref{thm:qi_to_cube_cplx}, and so that all supporting lemmas in \cite{BHS:quasi} apply.  Moreover, we will impose further requirements as needed. We note that the role of $K$ is often hidden in the statements, since it affects the set $\UU(F)$ which in the various statements often only plays a role implicitly.
\end{rem}

\subsection{Deleting subdivision points}

Now we consider how the construction of $\QQ_{F,\pt}$ is affected by the deletion of
points in $\pt$.
If $\pt_0\subset \pt$ there is a hyperplane-deletion map
$h:\QQ_{F,\pt}\to\QQ_{F,\pt_0}$. That is, 
for $x\in \QQ_{F,\pt}$, the image 
$h(x)\in \QQ_{F,\pt_0}$ is just the orientation on the remaining
walls:  $h(x)(p) = x(p)$ for $p\in \pt_0$.   We note that the subdivisions in the following proposition need not be evenly spaced.

\begin{proposition}{point deletion}
 For every $k$, $n$, and $M,M',K$ as in Theorem \ref{thm:qi_to_cube_cplx}, there exist $K'$ such that, if $F$ has
 cardinality at most $k$ and $\pt$ is an $(M,M')$-subdivision,
 and $\pt_0\subset \pt$ satisfies $|\pt\ssm\pt_0| \le n$, then $\pt_0$ is an $(M,M'')$-subdivision satisfying the conclusions of Theorem \ref{thm:qi_to_cube_cplx} and the diagram
  \begin{equation}
    \begin{tikzcd}
    \QQ_{F,\pt} \arrow[dr,"\Phi_{F,\pt}"] \arrow[d,"h" left] &  \\
    \QQ_{F,\pt_0} \arrow[r,"\Phi_{F,\pt_0}"] & \XX
  \end{tikzcd}
  \end{equation}
  commutes up to error $K'$.
\end{proposition}

\begin{proof}
As above, the map $\Phi_{F,\pt}(x)$ is determined by the coordinates
  $$
  b_U = \bigcap_{p\in \pt} S_{p,U}(x)
  $$
  whereas $\Phi_{F,\pt_0}(h(x))$ is determined by 
  $$
  b_{0,U} = \bigcap_{p\in \pt_0} S_{p,U}(x).
  $$
  Note that $b_U \subset b_{0,U}$, and in the other direction the diameter of $b_{0,U}$ is
  bounded by Theorem \ref{thm:qi_to_cube_cplx}(1), because the new set $\pt_0$ is
  an $(M,M''')$-subdivision for some $M''\geq M'$ (and in particular $M''\geq M'_0(M)$, so that the conclusion of Theorem \ref{thm:qi_to_cube_cplx} holds for $\pt_0$) depending only on $M,M'$ and the number $n$ of deletions.

  The bound on $d(\Phi_{F,\pt},\Phi_{F,\pt_0}\circ h)$ then follows
  from the distance formula (Theorem \ref{thm:distance_formula}), since the coordinates of $\Phi_{F,\pt_0}(h(x))$ and $\Phi_{F,\pt}(x)$ coarsely coincide with $b_U$ and $b_{0,U}$.
\end{proof}

\subsection{Intersection conditions}
Recall that Lemma \ref{lem:halfspaces_bijection} explains how a bijection between halfspaces that preserves intersection properties induces an isomorphism of the corresponding cube complexes. In view of this, we are interested in knowing when two of our halfspaces intersect.

We fix the setup of Subsection \ref{subsec:constr}. The next lemma is the main technical support for Proposition \ref{prop:isomorphism of cc} below:

\begin{lemma}{lem:halfspaces_intersect}
There exists $M_1$, depending on $(\XX,\mathfrak S)$ and $|F|$, so that the following holds.  Let $\pt$ be an $(M,M')$-subdivision, with $M\geq M_1$.
 Consider two halfspaces $W^V_{p,\sigma},W^Z_{q,\tau}$, with associated half-trees $T^V_{p,\sigma}, T^Z_{q,\tau}$, where $p \in \pt_V(F), q\in \pt_Z(F)$ and $\tau, \sigma \in \{\pm\}$.
 
Then $W^V_{p,\sigma},W^Z_{q,\tau}$ intersect if and only if one of the following holds, up to switching the roles of the half-spaces:
 \begin{enumerate}
 
   \item $V\orth Z$,
  \item $V=Z$, and $T^V_{p,\sigma}\cap T^Z_{q,\tau}\neq \emptyset$,
   \item $V\transverse Z$, and $T^Z_{q,\tau}$ contains $\beta^Z_F(\rho^V_Z)$,
  \item $V\propnest Z$, and $T^Z_{q,\tau}$ contains $\beta^Z_F(\rho^V_Z)$, 
  \item $V\propnest Z$, and $T^V_{p,\sigma}$ contains $\beta^V_F(\rho^Z_V(q))$.

 \end{enumerate}
\end{lemma}

We note that the last 3 cases of the lemma boil down to the consistency inequalities for HHSs (and could even be seen as interpretations thereof).

\begin{proof}
In an effort to enhance readability of the proof, we will make coarse, comparative arguments which keep track of dependencies of constants and their relative size, instead of precise quantities.  

It follows from Theorem \ref{thm:HHS-hull} and the fact the the various $T^V_F$ are quasi-isometrically embedded (Theorem \ref{thm:stable tree}) that the image $\beta^V_F(H_\theta(F))$ is
$D$-dense in the tree $T^V_F$ (with respect to its path metric), where $D$ only depends on the HHS structure and $|F|$.
We may assume that $M$ has been chosen greater than
$10D$. Moreover, throughout the proof we will further specify conditions on $M$, requiring it to be suitably larger than other constants appearing in the argument.  It is important to notice that, in each case, these constants depend only on $(\XX,\mathfrak S)$ and $|F|$.  We also remark that we need to be careful when comparing distances in the trees $T^V_F$ and the ambient spaces $\CC(V)$.

{\bf Case $V\orth Z$.} The fact that in this case all pairs of halfspaces intersect is \cite[Lemma 2.13]{BHS:quasi}.
\par\smallskip

{\bf Case $V=Z$.} If the halfspaces intersect, then any point $x$ in their intersection
has $\beta^V_F(x)\in T^V_{p,\sigma}\cap T^V_{q,\tau}$. Conversely, suppose $T^V_{p,\sigma}\cap
T^V_{q,\tau}\neq \emptyset$. If the intersection contains both $p$ and $q$, then
$$\diam_{T^V_F}(T^V_{p, \sigma} \cap T^V_{q, \tau}) \geq d_{T^V_F}(p,q) \geq M.$$
If not, then the intersection contains all of $T^V_{p,\sigma}$ or all of $T^V_{q,\tau}$,
and again
$$\diam_{T^V_F}(T^V_{p, \sigma} \cap T^V_{q, \tau})  \geq M.$$

\begin{figure}
\includegraphics[width=0.65\textwidth]{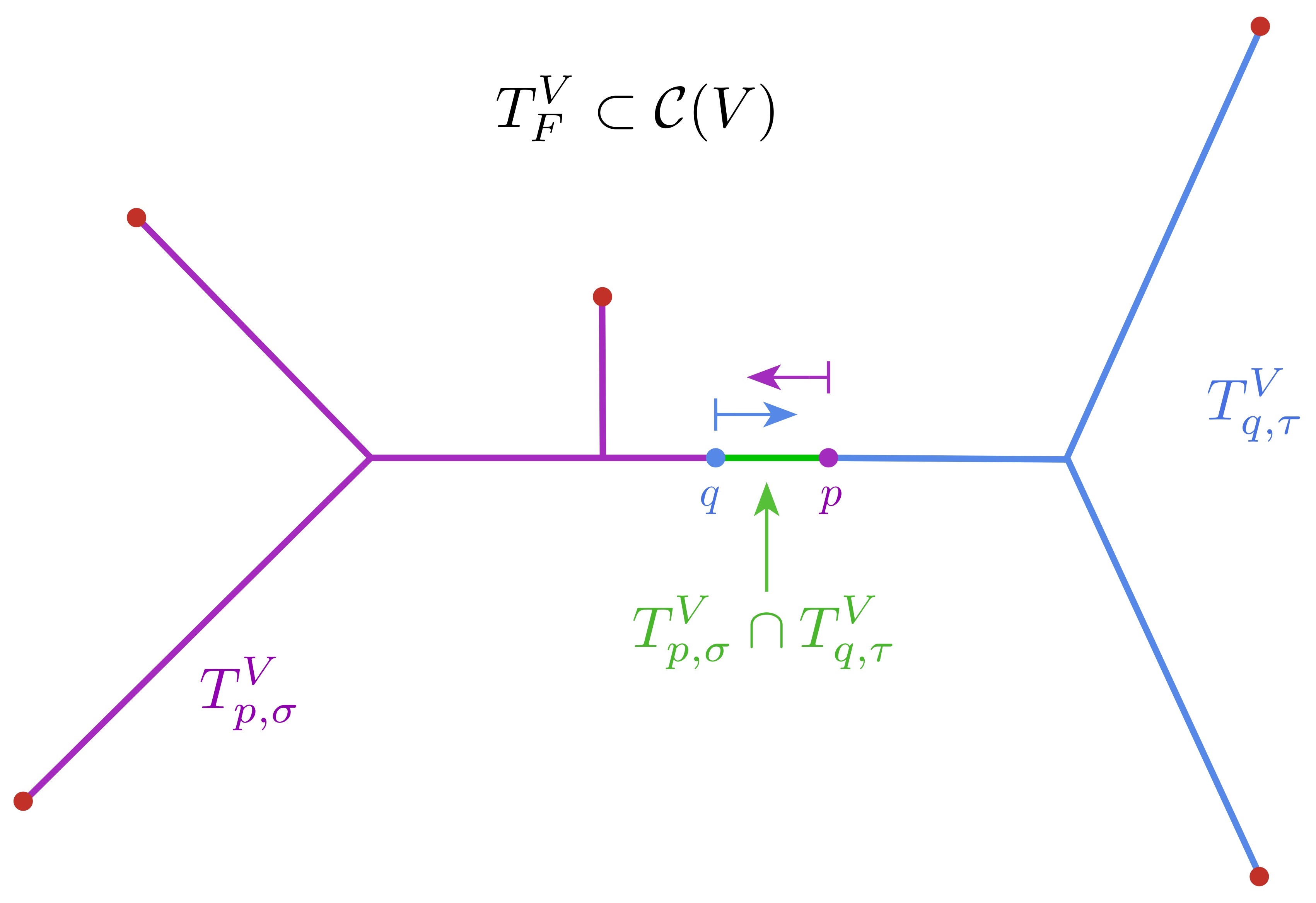}
\caption{The first possibility in case (2), when $V=Z$, where $p$ and $q$ choose different ends of $T^V_F$, and necessarily $\displaystyle \diam(T^V_{p, \sigma} \cap T^V_{q, \tau}) \geq d_{T^V_F}(p,q) \geq M.$}\label{fig:V=Z1}
\end{figure}

\begin{figure}
\includegraphics[width=0.65\textwidth]{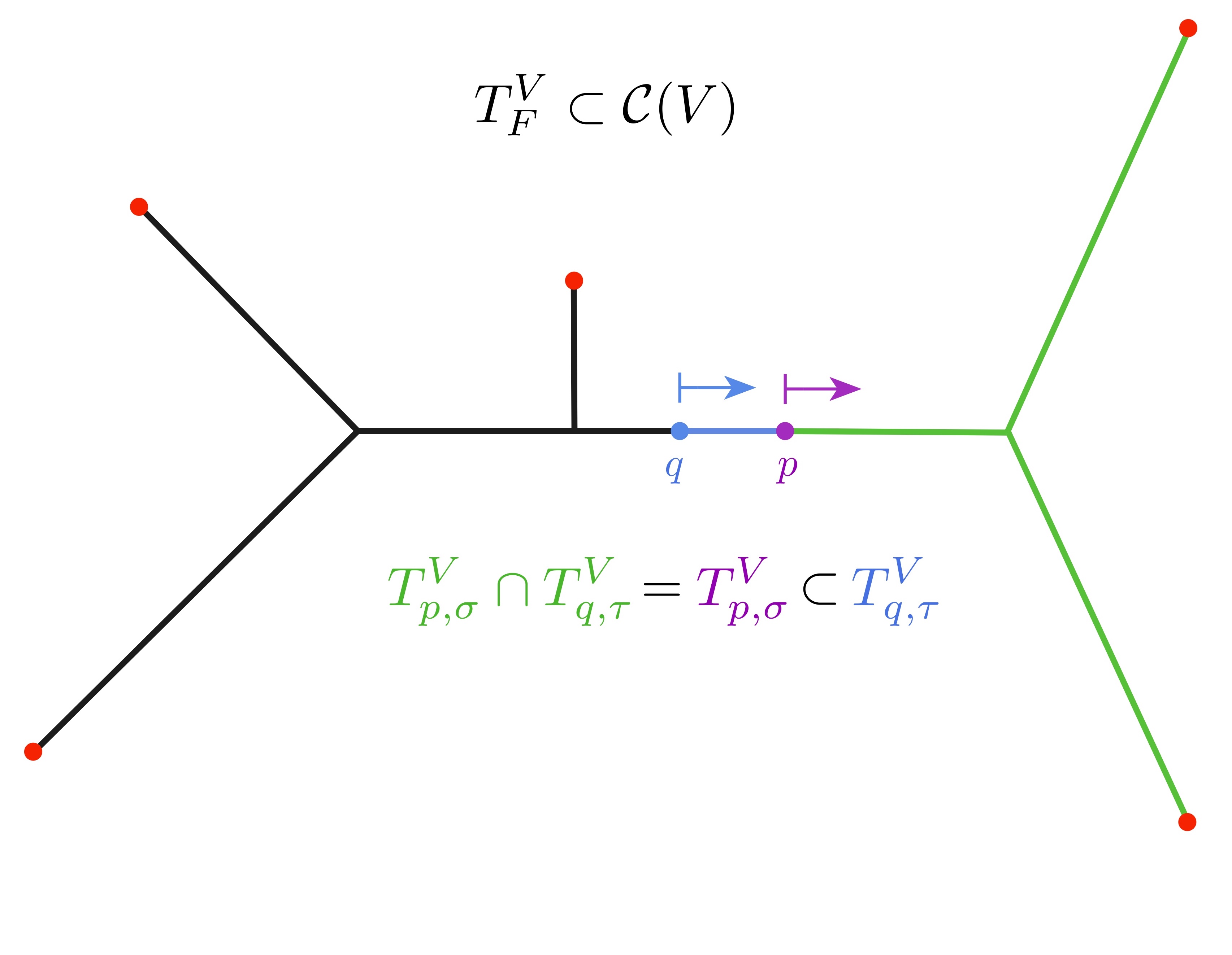}
\caption{The other possibility in case (2), when $V=Z$, where $p$ and $q$ now choose the same ends of $T^V_F$.}\label{fig:V=Z2}
\end{figure}

Since $\beta^V_F(H_{\theta}(F))$ is $D$-dense in $T^V_F$ and $M>D$, 
we can then find $x\in H_{\theta}(F)$ with $\beta^V_F(x) \in T^V_{p,\sigma}\cap T^Z_{q,\tau}$, so that $x \in W^V_p \cap W^Z_q$, as required.

\par\smallskip
{\bf Case $V\transverse Z$.}
If the halfspaces $W^V_{p,\sigma}$ and $W^Z_{q,\tau}$ intersect, then any point $x$ in
their intersection has $\beta^V_F(x)\in T^V_{p,\sigma}$ and $\beta^Z_F(x)\in
T^Z_{q,\tau}$. We claim that
$$ \beta^Z_F(\rho^V_Z)\notin T^Z_{q,\tau} \implies d_Z(\rho^V_Z, \pi_Z(x)) > \kappa_0,$$
where $\kappa_0$ is the constant in the transverse consistency inequality (\ref{eq:transverse consistency}),
and similarly with the roles of $Z$ and $V$ interchanged. Indeed, 
recall that Proposition \ref{prop:rhos_close_to_leaves}(1) says that (assuming $M$ is
sufficiently large) $\rho^V_Z$ lies within $\kappa$ of a leaf of $T^Z_F$, where $\kappa$
only depends on the HHS structure and $|F|$. Since $q$ is at least $M$ away from the
leaves, and we can assume assume that $M$ is sufficiently large compared to $\kappa$ and the quasi-isometric embedding constants of $T^Z_F$, we see that if $\beta^Z_F(\rho^V_Z)$ is not in $T^Z_{q,\tau}$
then it must be at least $M/2$ from it (as measured in the metric of $T^Z_F$). Thus $\beta^Z_F(\rho^V_Z)$ is at least $M/2$ from $\beta^Z_F(x)$ in $T^Z_F$, and
since we can assume that $M$ is sufficiently large compared to $\kappa$, the quasi-isometric embedding constants of $T^Z_F$, and the distance between $\pi_Z(x)$ and (the image in $\CC(Z)$ of) $T^Z_F$, we have the desired inequality. The same holds with $V$ and $Z$
interchanged.

However, the transverse consistency inequality (\ref{eq:transverse consistency}) says that
we cannot have both $d_Z(\rho^V_Z, \pi_Z(x)) > \kappa_0$ and $d_V(\rho^Z_V, \pi_V(x)) >
\kappa_0$. Hence one of
$\beta^Z_F(\rho^V_Z)\in T^Z_{q,\tau}$
or 
$\beta^V_F(\rho^Z_V)\in T^V_{p,\sigma}$
must hold, which is what we wanted to prove.

Conversely, suppose that $T^Z_{q,\tau}$ contains $\beta^Z_F(\rho^V_Z)$. Because
$\beta^V_F(H_\theta(F))$ is $D$-dense in $T^V_F$, $p$ is far from any leaf of $T^Z_F$, and $\beta^Z_F(\rho^V_Z)$ is close to a leaf (Proposition \ref{prop:rhos_close_to_leaves}(1)),
we may choose $x\in H_\theta(F)$ with $\beta^V_F(x)\in T^V_{p,\sigma}$ and $d_{T^V_F}(\beta^V_F(x), \beta^V_F(\rho^Z_V))
> M/2$ (again, with distance in $T^V_F$).

Then $x \in W^V_{p,\sigma}$ by construction, and we claim also $x\in W^Z_{q,\tau}$.
Indeed, for $M$ sufficiently large we have $d_V(x, \rho^Z_V)>\kappa_0$, so that transverse consistency (\ref{eq:transverse consistency}) implies that
$d_Z(\pi_Z(x), \rho^V_Z) < \kappa_0$. Again for $M$ sufficiently large, this gives $d_{T^Z_F}(\beta^Z_F(x),\beta^Z_F(\rho^V_Z))\leq M/2$, and since  the $M/2$-neighborhood in $T^Z_F$ of
$\beta^Z_F(\rho^V_Z)$ is contained in $T^Z_{q,\tau}$, we conclude 
$\beta^Z_F(x)\in T^Z_{q,\tau}$.

\begin{figure}
\includegraphics[width=0.75\textwidth]{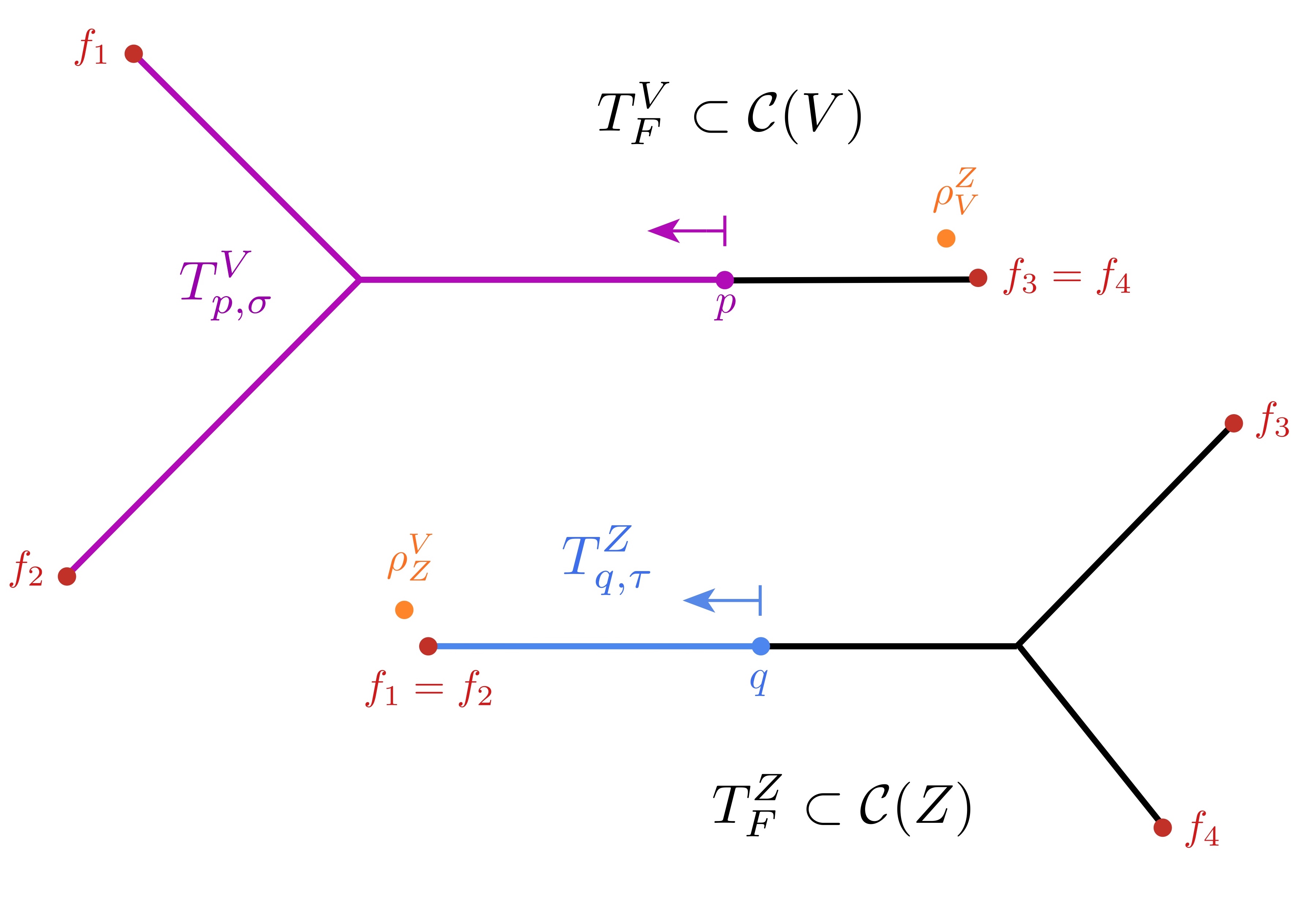}
\caption{Case (3), when $V\pitchfork Z$.  }\label{fig:VtransZ}
\end{figure}

\par\smallskip

{\bf Case $V\propnest Z$.} Cases (4) and (5) are similar to the previous case, instead using the nested consistency inequality \eqref{eq:nested consistency}.

Let $x \in W^V_{p,\sigma} \cap W^Z_{q,\tau}$.
Because the partition points in $T^Z_F$ are $M$-evenly spaced by assumption, we have again that either 
\begin{itemize}
\item $T^Z_{q,\tau}$ contains the $M/2$-neighborhood of $\beta^Z_F(\rho^V_Z)$, or 
\item $d_{T^Z_F}(T^Z_{q,\tau}, \beta^Z_F(\rho^V_Z))>M/2$.
\end{itemize}
In the first case, we are done (obtaining case (4)). 
In the second case, for $M$ sufficiently large, we have $d_V(\pi_V(x),\rho^Z_V(\pi_Z(x)))<\kappa_0$ by the nested
consistency inequality (\ref{eq:nested consistency}). 
Moreover, for $M$ sufficiently large, there is a geodesic (in $\CC(Z)$) from $q$ to
$\pi_Z(x)$ which is farther than $\kappa_0$ from $\rho^V_Z$, where $\kappa_0$ is the constant in the bounded geodesic image property (Axiom
(\ref{item:dfs:bounded_geodesic_image}) in Definition \ref{defn:HHS}). In fact, the distance from $\pi_Z(x)$ to $T^Z_{q,\tau}$ is bounded in terms of $(\XX,\mathfrak S)$ and $|F|$, as is the quasiconvexity constant (of the image in $\CC(Z)$) of $T^Z_{q,\tau}$, and thus so is the distance of any point along this geodesic from $T^Z_{q,\tau}$.

On the other hand, we can estimate $d_{Z}(T^Z_{q,\tau}, \rho^V_Z)$ in terms of $M$ and the quasi-isometric embedding constants of $T^Z_F$ (which is independent of $M$).
Hence by choosing $M$ large enough, we may use the bounded geodesic image property to conclude that
$d_V(\rho^Z_V(\pi_Z(x)),\rho^Z_V(q))<\kappa_0$. Since $\pi_V(x)$ is within distance of $T^V_{p,\sigma}$ which is bounded in terms of $(\XX,\mathfrak S)$ and $|F|$, we see that
$d_V(\beta^V_F(\rho^Z_V(q)),T^V_{p,\sigma})<M/2$ for $M$ sufficiently large. However, this implies that
$\beta^V_F(\rho^Z_V(q))$ is in fact contained in $T^V_{p,\sigma}$, since
$\rho^Z_V(q)$ is within $\kappa$ of a leaf of $T^V_F$ by Proposition
\ref{prop:rhos_close_to_leaves}(2).  Hence we are done in either case. 

\begin{figure}
\includegraphics[width=0.75\textwidth]{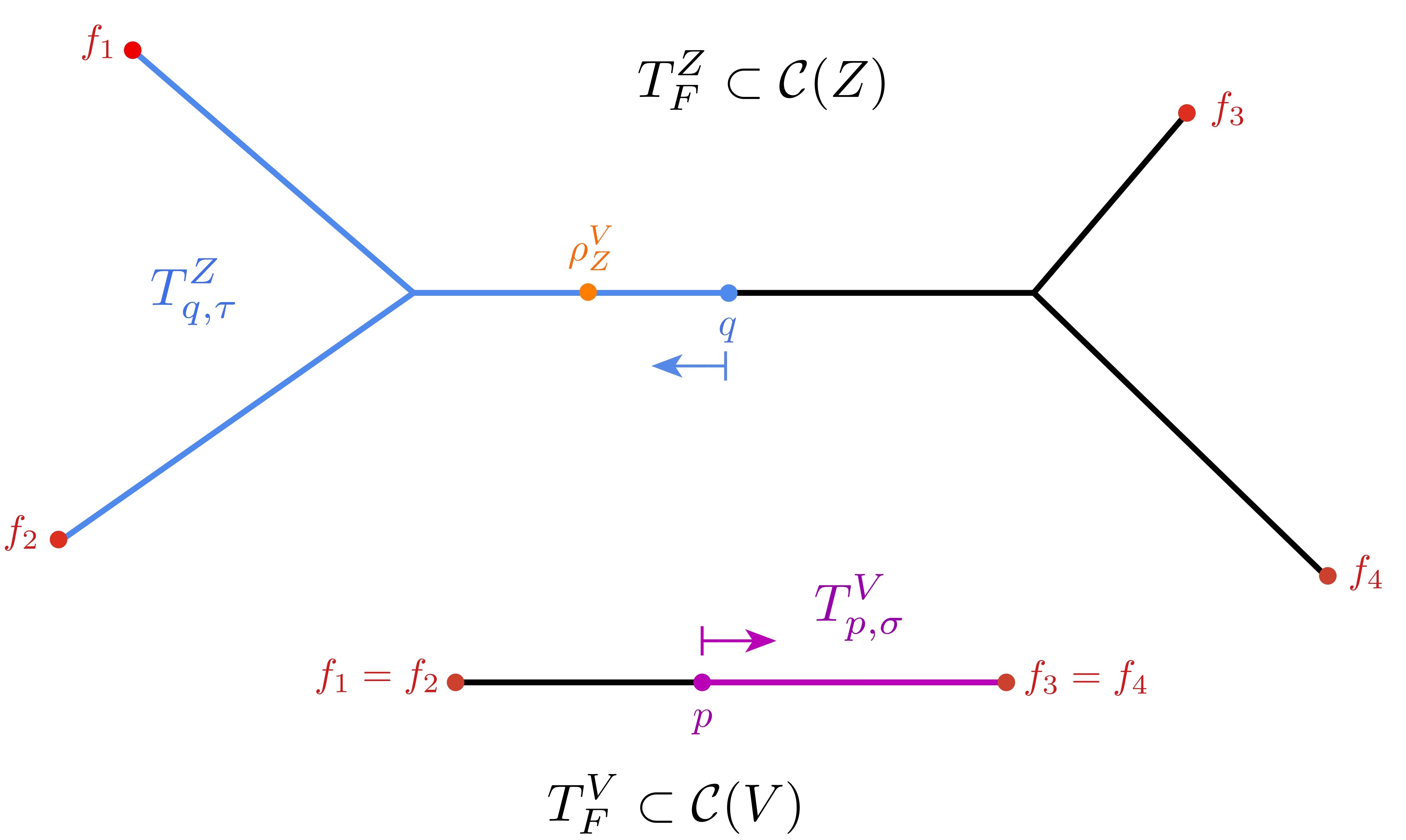}
\caption{Case (4), when $V \propnest Z$ and $T^Z_{q, \tau}$ contains $\beta^Z_F(\rho^V_Z)$.}\label{fig:VnestZ1}
\end{figure}
Now in the converse direction, suppose first that (4) holds, namely that $T^Z_{q,\tau}$
contains $\beta^Z_F(\rho^V_Z)$. Again by the $M$-even spacing of the partition we know that
 $T^Z_{q,\tau}$ contains the $M/2$-neighborhood (in $T^V_F$) of $\beta^Z_F(\rho^V_Z)$.
Using $D$-density of $\beta^V_F(H_\theta(F))$ in $T^V_F$ and the fact that $p$ is far from any leaf of $T^V_F$,
we can choose $x\in H_\theta(F)$ with $\beta^V_F(x)\in T^V_{p,\sigma}$ and $\beta^V_F(x)$
at least $M/2$ from a leaf of $T^V_{p, \sigma}$.

We claim that this implies that 
$\beta^Z_F(x)$ has to lie within $M/2$ of $\beta^Z_F(\rho^V_Z)$, as usual for $M$ large. If not, another
application of the bounded geodesic image axiom would imply that $\rho^Z_V(\pi_Z(x))$
is within $\kappa_0$ of a leaf, and the nested consistency inequality again gives that 
$\rho^Z_V(\pi_Z(x))$ and $\pi_V(x)$ are within $\kappa_0$ of each other, showing that $\pi_V(x)$ lies within $2\kappa_0$ of a leaf.  Since the constants involved in the preceding argument depend only on $(\XX,\mathfrak S)$ and $|F|$, it follows that the distance (in $T^V_F$) between $\beta^V_F(x)$ and a leaf is bounded independently of $M$, contradicting the choice of $x$.

As a result, we
find that  $\beta^Z_F(x)$ is in $ T^Z_{q,\tau}$ and hence $x \in W^V_{p,\sigma} \cap W^Z_{q,\tau}$, as required.

Suppose now as in (5) that $T^V_{p,\sigma}$ contains $\beta^V_F(\rho^Z_V(q))$, and assume
also that (4) does not hold so that $T^Z_{q,\sigma}$ avoids the $M/2$ neighborhood of
$\beta^V_F(\rho^Z_V)$ in $T^V_F$. In this
case, we can take $x\in H_\theta(F)$ with $\beta^Z_F(x)\in T^Z_{q,\tau}$ and
$d_{T^Z_F}(\beta^Z_F(x),q) \leq D$.   Then since $\beta^Z_F(\rho^Z_V)$ is at least $M/2$ from the
geodesic in $T^Z_F$ between $\beta^Z_F(x)$ and $q$, and since $T^Z_F$ is quasi-isometrically embedded in $\CC(Z)$, we may apply the bounded geodesic image axiom and the nested consistency inequality to obtain that
$d_V(\beta^V_F(x), \beta^V_F(\rho^Z_V(q)))<M/2$, where as before we choose $M$ as large as necessary.
Since $\rho^Z_V(q)$ lies within
$\kappa $ of a leaf of $T^V_F$ by Proposition \ref{prop:rhos_close_to_leaves}(2) with $\kappa$ depending only on $\mathfrak S$ and $|F|$,
the $M/2$-neighborhood of $\beta^V_F(\rho^Z_V(q))$ in $T^V_F$ must be contained in $T^V_{p,\sigma}$, and hence $\beta^V_F(x)$ is
contained in $T^V_{p,\sigma}$, showing $x \in W^V_{p,\sigma} \cap W^Z_{q,\tau}$ as required. 
\begin{figure}
\includegraphics[width=0.75\textwidth]{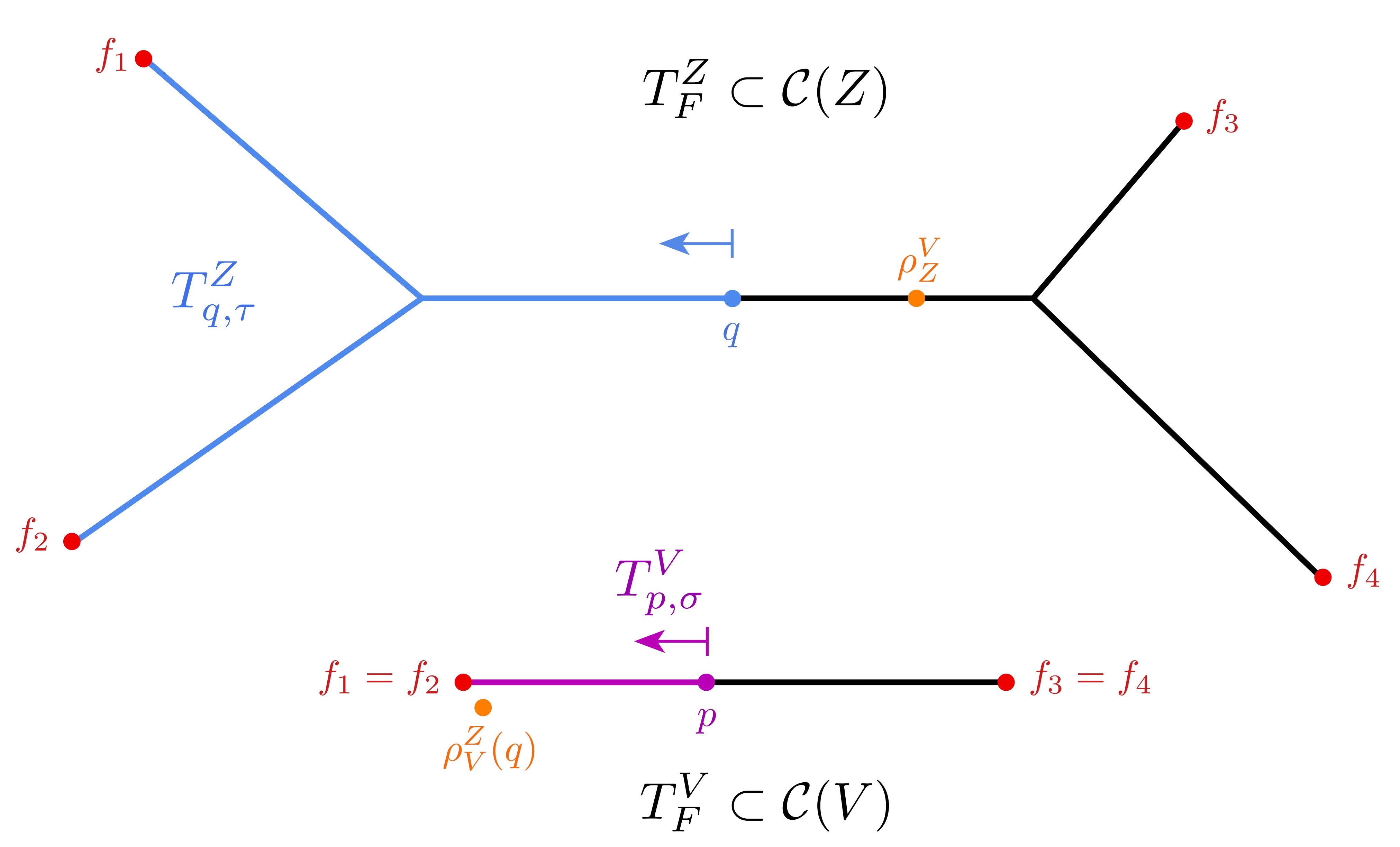}
\caption{A cartoon for case (5), when $V \propnest Z$ and $T^V_{p, \sigma}$ contains $\beta^V_F(\rho^Z_V(q))$.}\label{fig:VnestZ2}
\end{figure}
\par\smallskip
\end{proof}

\subsection{Refining the subdivisions}\label{subsec:comparing divisions}

In this subsection, we analyze the difference between $(M,M')$-evenly spaced subdivisions $\pt=\wp(F)$
and $\pt'=\wp(F')$ (using the fixed subdivision mechanism $\wp$ of Section \ref{subsec:subdivision}).
The main result here, Proposition \ref{prop:almost bijection of p}, is that there are bounded refinements $\pt_0$ and $\pt'_0$ which are related by an ``order-preserving'' bijection $j: \pt_0 \rightarrow \pt'_0$ (in the sense specified below) which only moves points a bounded distance.  In the next subsection, we prove that these refined subdivisions determine isomorphic cube complexes when put through our cubulation machine.

\subsubsection*{Coarse separation.}
We will need the following short discussion about coarse separation in quasi-isometrically embedded
trees. Suppose $j:T\to X$ is a $(\chi,\chi)$-quasi-isometric embedding of a tree $T$ into a
$\delta$-hyperbolic space $X$, and $p\in T$ is a point inside an edge, distance more than
$\mu>0$ from the endpoints. Let  $T_{p,\pm}(\mu)$ be the two components of the complement
of a $\mu$-neighborhood of $p$. For any $\ep$ there is a $\mu(\ep,\delta,\chi)$ such that the
images $j(T_{p,\pm}(\mu))$
are at least $\ep$ apart from each other. We want to compare this separation for two
nearby trees:

\begin{lemma}{lem:half-trees}
Let $T_1,T_2$ be trees with $(\chi,\chi)$-quasi-isometric embeddings $u_i:T_i \to X$ into a
$\delta$-hyperbolic space $X$, whose images are within Hausdorff distance $\epsilon$. There
exists $\mu_0 = \mu_0(\delta,\chi,\epsilon) > 10\ep$ so that for all $\mu\geq m_0$ the following holds. 

If $p_i\in T_i$ are points with
$d(u_1(p_1),u_2(p_2)) \le \mu$ and each $p_i$ is in a segment $e_i$ contained in an edge of $T_i$ and so that $u_i|_{e_i}$ is a $(1,\chi)$--quasi-isometric embedding. Moreover, assume that $p_i$ lies at distance more
than $2\mu$ from the endpoints of $e_i$. Then the labels of the components can be chosen so that
\begin{itemize}
 \item $u_1( (T_1)_{p_1,+}(2\mu))$ is in an
    $\ep$-neighborhood of $u_2( (T_2)_{p_2,+})$,
    
 \item $u_1( (T_1)_{p_1,+}(2\mu))$ is $\mu$-far from $u_2( (T_2)_{p_2,-})$, and
 \item the same holds swapping $+$ and $-$.
\end{itemize}
\end{lemma}

In what follows we apply this to trees $T^V_F$, and so that half-trees that here would be
$(T^V_F)_{p,\sigma}$ correspond to what we wrote as $T^V_{p,\sigma}$ in
Section \ref{subsec:subdivision}. We will use both notations, and hope  not to confuse the
reader.

\subsubsection*{Common refinements from close components}

We are now ready to prove the refinement statement.  We note that this is the main point at which we use the full power of Theorem \ref{thm:stable tree}, which provides not only that each of the pairs of trees $T^V_F, T^V_{F'}$ have Hausdorff close images in $\CC(V)$, but that they are identical on almost every component, and that their different components can be cut into large pieces where they are coarsely identical. 

\begin{proposition}{prop:almost bijection of p}
There exists $M_2 = M_2(|F|, \mathfrak S)>0$, such that if
$M > M_2$, there exists $R>0$ and 
\begin{enumerate}
\item Subsets $\pt_0 \subset \pt$ and $\pt'_0 \subset \pt'$ with
\begin{itemize}
\item $\pt_0(V)$ or $\pt'_0(V)\neq\emptyset$ only if $V\in\UU(F)\cap \UU(F')$, and
\item $|\pt\ssm \pt_0|, |\pt' \ssm \pt'_0| < R$, and
\end{itemize}
\item A bijection $j: \pt_0 \rightarrow \pt'_0$ with $d_V(p, j(p)) < 2M/3$ for any $p \in \pt_0(V)$.
\item For every $p\in\pt_0$ a bijection $j_p$ between the half-trees defined by $p$ and those defined by $j(p)$ with the following property. For any $V$ and $p,q\in\pt_0(V)$, if the half-tree $(T^V_F)_{p,\sigma}$ contains $q$, then $j_p((T^V_F)_{p,\sigma})$ contains $j(q)$.
\item Moreover, let $f,f'$ be so that either $f\in F$, $f'\in F'$, and $d_{\XX}(f,f')\leq 1$, or $f=f'=\rho^U_V$ for some $U, V\in\UU(F)\cap \UU(F')$ with $U\propnest V$ or $U \pitchfork V$. If $\beta^V_F(f)\in (T^V_F)_{p,\sigma}$, then $\beta^V_{F'}(f')\in j_p((T^V_F)_{p,\sigma})$.
\end{enumerate}
\end{proposition}

\begin{proof} 
We work in each $V\in \mathfrak S$ separately, constructing $\pt_0(V)$ and $\pt'_0(V)$ and
the bijection, and combine the results.

First, if $V \in \UU(F) \hspace{.025in} \symdiff  \hspace{.05in} \UU(F')$, then $\diam_V(F)$ and $\diam_V(F')$ are uniformly bounded, and we set $\pt_0(V) = \pt'_0(V) = \emptyset$.  By Proposition \ref{prop:stable family}, there are boundedly many such $V$, and hence this involves deleting at most boundedly many subdivision points.

Next, if $V \in \UU(F) \cap \UU(F')$ is not involved in the transition from $F$ to $F'$
(see Definition \ref{defn:involved}), then all of the relevant data is constant.  Hence,
by our formalism for choosing subdivisions (see the definition of $\wp$ in Subsection \ref{subsec:subdivision}), we
have $\pt(V) = \pt'(V)$, and we set $\pt_0(V) = \pt'_0(V) = \pt(V)$. 

Hence we may restrict our attention to a fixed $V \in \UU(F) \cap \UU(F')$ for which either $\pi_V(F) \neq \pi_V(F')$ or $\UU^V(F) \neq \UU^V(F')$ (or both).  We note that Proposition \ref{prop:bounded involved} bounds the number of such $V$ solely in terms of $n$ and $\mathfrak S$.

Fix such a $V$. Recall that within $\CC(V)$  we have $F^V = \pi_V(F)$ and $\YY^V =
\{\rho^Y_V| Y \in \UU^V(F)\}$, and similarly for ${F'}^V$ and ${\YY'}^V$. 
By Proposition \ref{prop:bounded involved},  $\# \left(\YY^V \symdiff {\YY'}^V\right) < N_1$, where $N_1= N_1(\mathfrak S,k)>0$.

Recall that Theorem \ref{thm:stable tree} provides a constant $L = L(k, \mathfrak S)>0$ and the following:
\begin{enumerate}
\item Stable trees with decompositions $T^V_F = T_c(F^V, \YY^V) \cup T_e(F^V, \YY^V)$ and
  $T^V_{F'} = T_c({F'}^V, {\YY'}^V) \cup T_e({F'}^V, {\YY'}^V)$; we write these as
  $T^V_c(F), T^V_e(F)$ and $T^V_c(F'), T^V_e(F')$ for short. 
\item Two stable subsets $T_s \subset T^V_e(F)$ and $T'_s \subset T^V_e(F')$ so that
\begin{enumerate}
\item $T_s$ and $T'_s$ are contained in the interiors of the edges of $T^V_e(F)$ and $T^V_e(F')$, respectively;
\item The complements $T^V_e(F) \ssm T_s$ and $T^V_e(F') \ssm T'_s$ each have at most $L$ components, each of which has diameter at most $L$;
\item There is a bijective correspondence between the components of $T_s$ and $T'_s$;
\item This bijective correspondence is the identity on all but at most $L$ components, with identical components of $T_s, T'_s$ coming from identical components of $T_e^V(F)$ and $T_e^V(F')$; 
\item The remaining components of $T_s$ are within Hausdorff distance $L$ of their corresponding components in $T'_s$.
\end{enumerate}
\end{enumerate}

We assume  $M >  \max(4L,4\mu_0, 4\kappa)$, where
$\mu_0$ is given by Lemma \ref{lem:half-trees} when $\epsilon = L$
and the quasi-isometry constants of the trees match those for $T^V_F$ and $T^V_{F'}$, 
and $\kappa$ is the constant in Proposition \ref{prop:rhos_close_to_leaves}.

Consider the sets of subdivision points $\pt_1(V) = \pt(V) \cap T_s$ and $\pt'_1(V) =
\pt'(V) \cap T'_s$ which are contained in the stable subsets.  Since $\pt(V) \subset
T^V_e(F)$ and $\pt'(V) \subset T^V_e(F')$, items (2a) and (2b) and our choice of subdivision width $M>4L$ imply that
$\pt(V) \ssm \pt_1(V)$ and $\pt'(V) \ssm \pt'_1(V)$ both have cardinality bounded above by
$L$.

By item (2d), the induced subdivisions $\pt_1(V)$ and $\pt'_1(V)$ agree on all but at most
$L$ components of $T_s$ and $T'_s$, respectively, as these components are segments in the
components of $T^V_e(F)$ and $T^V_e(F')$ which are equal and hence have
the same subdivisions by our setup (see Subsection \ref{subsec:subdivision}).

On the remaining $L$ components we can make bounded adjustments.
Let $e$ and $e'$ be edge components of $T_s$ and $T'_s$ related by the correspondence.
The closest subdivision point in $e$ to an endpoint is at most $M'/2$ away and at least $M$,
by definition of the subdivisions, and similarly for $e'$. 
Since $d_{Haus}(e,e') < L$ by item (2e), and $e$ and $e'$ are quasi-isometrically embedded with multiplicative constant 1 and additive constant depending on $(\XX,\mathfrak S)$ and $n$ (Theorem \ref{thm:stable tree} \ref{item:qi2}), the difference between the number of subdivision points in $e$ and $e'$ is bounded in terms of $M,M',L, (\XX,\mathfrak S)$, and $n$.

We denote $\pt_1(e)=\pt_1(V)\cap e$, and similarly for $\pt'_1(e')$. Note that these sets are naturally ordered once we choose endpoints of $e$ and $e'$, and we order $e$ and $e'$ by declaring those endpoints to be minimal. We choose endpoints of $e$ and $e'$ within distance bounded in terms of $L$ and the constants of the quasi-isometric embeddings of $e$ and $e'$ in $\CC(V)$. We will assume that $M$ is larger than this bound.

After deleting from $\pt_1(e) \cup \pt'_1(e')$
a number of points bounded by this constant (all of which occur near the endpoints of $e,e'$) and using the fact that $\pt_1(e), \pt'_1(e')$ are $M$-evenly spaced, we can obtain refinements $\pt_0(e)$ and $\pt'_0(e')$, which admit
an order-preserving (with respect to the aforementioned order) bijection $j_e:\pt_0(e) \rightarrow
\pt'_0(e')$ which satisfies  
$$d_V(p, j_e(p)) < L + M/2 +\zeta< 2M/3,$$
where $e' \to \CC(V)$ is a $(1,\zeta)$-quasi-isometric embedding by Theorem \ref{thm:stable tree}\ref{item:qi2}, and we have chosen $M$ sufficiently large to guarantee the inequality.  Indeed, for each $p$ we find a nearest point in $e'$, which is at most $L$ away, and then
move along $e'$ at most $M/2 + \zeta$ to a point of $\pt'_0(e')$. For later purposes, we can assume that points in $\pt_0(e)$ do not lie within $2M$ of the endpoints of $e$, and similarly for $\pt'_0(e')$. 

If we set $ \pt_0(V) = \bigcup_{e \in \pi_0(T_s)} \pt_0(e)$ and define $\pt'_0(V)$
similarly, then the $j_e$ maps combine to give a bijection $j_V: \pt_0(V) \rightarrow
\pt'_0(V)$ which moves points by distance at most $2M/3$, as required for item (2) of the
proposition.

To define the map $j_p$ between half-trees we use Lemma \ref{lem:half-trees}. Namely, we pair the half-tree $(T^V_F)_{p,\sigma}$ with the half-tree $(T^V_{F'})_{j_V(p),\tau}$ that contains $(T^V_F)_{p,\sigma}(2M)$ in its $L$-neighborhood.

For item (3) of the proposition, let $p,q\in \pt_0(V)$.  There are two
cases.

Suppose first that $p,q$ lie in the same edge-component $e$ of $T_s$. Recall that we chose
the bijection $j_e$ to be order-preserving, with respect to the order along $e$ and $e'$
determined by choosing endpoints $e^-\in e$ and $(e')^-\in e'$ which are a small distance
(less than $M$) apart to be minimal in the orders. Let $(T^V_F)_{p,\sigma}$ denote
the half-tree of $p$ containing $e^-$ and let $(T^V_{F'})_{j_V(p),\tau}$ denote
the half-tree of $j_V(p)$ containing $(e')^-$. Since we have $2M$ spacing now between $p$
and the endpoints of $e$, $e^-$ is in $(T^V_F)_{p,\sigma}(2M)$, and Lemma
\ref{lem:half-trees} says that there is exactly one half-tree at $j_V(p)$ which comes within $M$ of
$(T^V_F)_{p,\sigma}(2M)$. Since $e^-$ is within $M$ of $(e')^-$, it follows that 
$(T^V_{F'})_{j_V(p),\tau}$ is in fact the paired half-tree provided by Lemma
\ref{lem:half-trees}, which is our $j_p((T^V_F)_{p,\sigma})$.  We conclude that $q<p$ in the
order along $e$ if and only if $q\in (T^V_F)_{p,\sigma}$ and
$j_V(q) < j_V(p)$ along $e'$ if and only if $j_V(q)\in j_p((T^V_F)_{p,\sigma})$. Since
$j_e$  is order-preserving, (3) follows in this case.

Suppose now that $p,q$ do not lie in the same edge-component of $T_s$.
Suppose $q$ is contained in the half-tree $(T^V_F)_{p,\sigma}$, and let $(T^V_{F'})_{q,\tau}=j_p((T^V_F)_{p,\sigma})$.
In this case we have that $q$ lies in $(T^V_F)_{p,\sigma}(2M)$, rather than just in
$(T^V_F)_{p,\sigma}$. By Lemma \ref{lem:half-trees} there is only one half-tree of
$T^V_{F'}$ at $j_V(p)$ that comes within $M$ of $(T^V_F)_{p,\sigma}(2M)$, and said half-tree
must be $(T^V_{F'})_{j_V(p),\tau}$. Since $j_V(q)$ lies within $M$ of $q$, we have $j_V(q)\in
(T^V_{F'})_{j_V(p),\tau}$, as required.  

The argument for part (4) of the proposition is essentially the same as the argument for the second case of part (3), since all we used there is that the point $q$ of $T^V_F$ is not close to $p$, but it is close to a corresponding point in $T^V_{F'}$, and the analogous properties hold in all the listed cases.

\end{proof}

\subsection{Refinements give isomorphic cube complexes}

Consider the refined subdivisions $\pt_0$ and $\pt'_0$ for $F$ and $F'$, respectively,
that are produced by Proposition \ref{prop:almost bijection of p}.  These are $(M,M'')$
spaced subdivisions (though no longer evenly spaced) so each of the sets of data $(F,
\pt_0)$ and $(F', \pt'_0)$---and their associated collections of stable trees---can be
plugged into our cubulation machine to produce cube complexes $\QQ_{F, \pt_0}$ and
$\QQ_{F', \pt'_0}$, respectively.  We also assume that $M > \max{M_1, M_2}$, where $M_1, M_2$ are the constants from Lemma \ref{lem:halfspaces_intersect} and Proposition \ref{prop:almost bijection of p}, respectively, along with our other base assumptions about $M$.

Our next result says that these cube complexes are abstractly isomorphic and admit coarsely compatible quasi-isometric embeddings into $\XX$.  Using Proposition \ref{point deletion}, we will be able to conclude that the right hand side of Diagram \ref{eq:Phi diagram} from Theorem \ref{thm:stable_cubulations} commutes. 

\begin{proposition}{prop:isomorphism of cc}
There exists $M_3 = M_3(|F|, \mathfrak S)>0$, such that if
$M > M_3$, there exists $B>0$ and a cubical isomorphism $\hat{h}:\QQ_{F, \pt_0} \rightarrow \QQ_{F', \pt'_0}$ so that the diagram
     \begin{equation}\label{Phi diagram}
  \begin{tikzcd}
   \QQ_{F} \arrow[ddrr,"\Phi_{F}", bend left=40] \arrow[dr,"h \hspace{.075in}" left] &  \\
    &\QQ_{F, \pt_0} \arrow[dr,"\Phi_0 \hspace{.2in}" below]\arrow[dd, "\hat{h}"] \\
    & & \XX\\
    & \QQ_{F', \pt'_0} \arrow[ur,"\Phi'_0"] \\
    \QQ_{F'}\arrow[uurr,"\hspace{.2in} \vspace{.1in} \Phi_{F'}" below, bend right=40] \arrow[ur,"h'"] & \\
  \end{tikzcd}
  \end{equation}
commutes up to error $B$.
\end{proposition}

\begin{proof}

We will define the required cubical isomorphism $\hat{h}$ and then prove that the middle
triangle commutes up to bounded error.  This suffices for the proposition because
Proposition \ref{point deletion} says that the top and bottom triangles commute up to bounded
error. 

In order to use Lemma \ref{lem:halfspaces_bijection} to define an isomorphism between
$\QQ_{F,\pt_0}$ and $\QQ_{F',\pt'_0}$, we need a bijection between the corresponding
collections of half-spaces which preserves complements and disjointness.

Let $\WW_{F, \pt_0}$ and $\WW_{F', \pt'_0}$ be the sets of half-spaces and construct a bijection $\iota: \WW_{F, \pt_0}\rightarrow \WW_{F', \pt'_0}$ as follows.  
By Proposition \ref{prop:almost bijection of p}, we need only consider
$V \in \UU(F) \cap \UU(F')$.  Any $p\in \pt_0(V)$ is contained in an edge of $T^V_F$ (in fact in
$T^V_e$) and is at least $M$ from its endpoints.

Let $j:\pt_0 \rightarrow \pt'_0$ be the bijection provided by
Proposition \ref{prop:almost bijection of p}, together with the corresponding maps $j_p$ pairing the half-trees at $p\in\pt_0$ with those at $j(p)$.

To define $\iota$, let $p\in\pt_0$ and $\sigma\in\{\pm\}$, and if
$j_p(T^V_{p,\sigma})=T^V_{j_V(p),\tau}$ we define $\iota(W^V_{p,\sigma})=W^V_{j_V(p),\tau}$.
It is straight-forward to confirm that $\iota$ respects complementation as in condition
(1) of Lemma \ref{lem:halfspaces_bijection}.

We now confirm condition (2) of Lemma \ref{lem:halfspaces_bijection} for $\iota$ by using the various characterizations of half-space intersections given in Lemma \ref{lem:halfspaces_intersect}.  We remark that the figures from that proof are again relevant.

Let $p \in \pt_0(V)$ and $q \in \pt_0(Z)$ for $Z, V \in \UU(F) \cap \UU(F')$, and suppose
that the half-spaces corresponding to the half-trees $T^V_{p, \sigma}$ and $T^Z_{q, \tau}$
intersect non-trivially, where $\sigma, \tau \in \{\pm\}$.  There are five cases, up to switching the roles of $V$ and $Z$.
\par\smallskip

\textbf{Case $Z \orth V$:} This case follows immediately from the construction and Lemma \ref{lem:halfspaces_intersect}(1), since all relevant pairs of half-spaces intersect in this case.
\par\smallskip

\textbf{Case $Z=V$:} In this case, Lemma \ref{lem:halfspaces_intersect}(2) implies that
$T^V_{p, \sigma} \cap T^V_{q, \tau} \neq \emptyset$ (recall Figures \ref{fig:V=Z1} and
\ref{fig:V=Z2}).

In particular, up to switching the roles of $p$ and $q$, we have $q\in T^V_{p, \sigma}$. But then, in view of (3) of Proposition \ref{prop:almost bijection of p}, $j(q)\in j_p(T^V_{p, \sigma})$. In particular we have $j_p(T^V_{p, \sigma}) \cap j_q(T^V_{q, \tau}) \neq \emptyset$, so that the corresponding half-spaces intersect, again by Lemma \ref{lem:halfspaces_intersect}(2).

\par\smallskip

\textbf{Case $Z \pitchfork V$:} In this case, Lemma \ref{lem:halfspaces_intersect}(3)
implies that, up to switching the roles of $Z$ and $V$, that $(T^V_F)_{p,\sigma}$ contains
$\beta^V_F(\rho^Z_V)$ (recall Figure \ref{fig:VtransZ}).  It follows immediately from part (4) of Proposition \ref{prop:almost bijection of p} that $\beta^V_{F'}(\rho^Z_V) \in j_p(T^V_{p,\sigma})$, as required.

\par\smallskip

\textbf{Case $V \propnest Z$:} By Lemma \ref{lem:halfspaces_intersect}, there are two subcases, up to switching the roles of $V$ and $Z$: (a) when $T^Z_{q, \tau}$ contains $\beta^Z_F(\rho^V_Z)$, and (b) when $T^V_{p, \sigma}$ contains $\beta^V_F(\rho^Z_V(q))$.

In case (a), since $V\in \UU(F)\cap\UU(F')$ part (4) of Proposition \ref{prop:almost bijection of p} immediately gives that $\beta^Z_{F'}(\rho^V_Z)$ lies in $j_q(T^Z_{q, \tau})$.

Suppose now that (b) holds. We prove that $\beta^V_{F'}(\rho^Z_V(j(q)))$ is contained in $j_p(T^V_{p, \sigma})$ (recall Figure \ref{fig:VnestZ2}).  Recall from Proposition \ref{prop:rhos_close_to_leaves} that $\rho^Z_V(q)$ lies close to some $\pi_V(f)$ with bound in terms of $\mathfrak S$ and $|F|$.  Since $d_V(q,j(q))< 2M/3$ and both $q$ and $j(q)$ are at least $M-2K$ from $\rho^V_Z$, with the $2K$ coming from the facts that $T_c^Z(F)$ is within Hausdorff distance $K$ of the $\YY^Z$ and the edge of $T_e^Z(F)$ containing $q$ is $(1,K)$-quasi-isometrically embedded in $\CC(Z)$ (Theorem \ref{thm:stable tree}\ref{item:bijection}).  Choosing $M$ sufficiently large, we can guarantee that any geodesic in $\CC(V)$ between $p$ and $j_V(p)$ avoids the $\kappa_0$-neighborhood of $\rho^V_Z$, for $\kappa_0$ the constant of the bounded geodesic image property, which then bounds $d_V(\rho^Z_V(q),\rho^Z_V(j_Z(q))) < \kappa_0$.  Hence $\rho^Z_V(j_Z(q))$ is close to both $\rho^Z_V(q)$ and a leaf of $T^V_{F'}$, and thus $\beta^V_{F'}(\rho^Z_V(j(q))) \subset j_p(T^V_{p, \sigma})$, as required.

Since the wallspace map $\iota$ satisfies the conditions of Lemma \ref{lem:halfspaces_bijection}, we obtain a cubical isomorphism $\hat{h}: \QQ_{F, \pt_0} \rightarrow \QQ_{F', \pt'_0}$.

\par\smallskip

\begin{figure}
\includegraphics[width=1\textwidth]{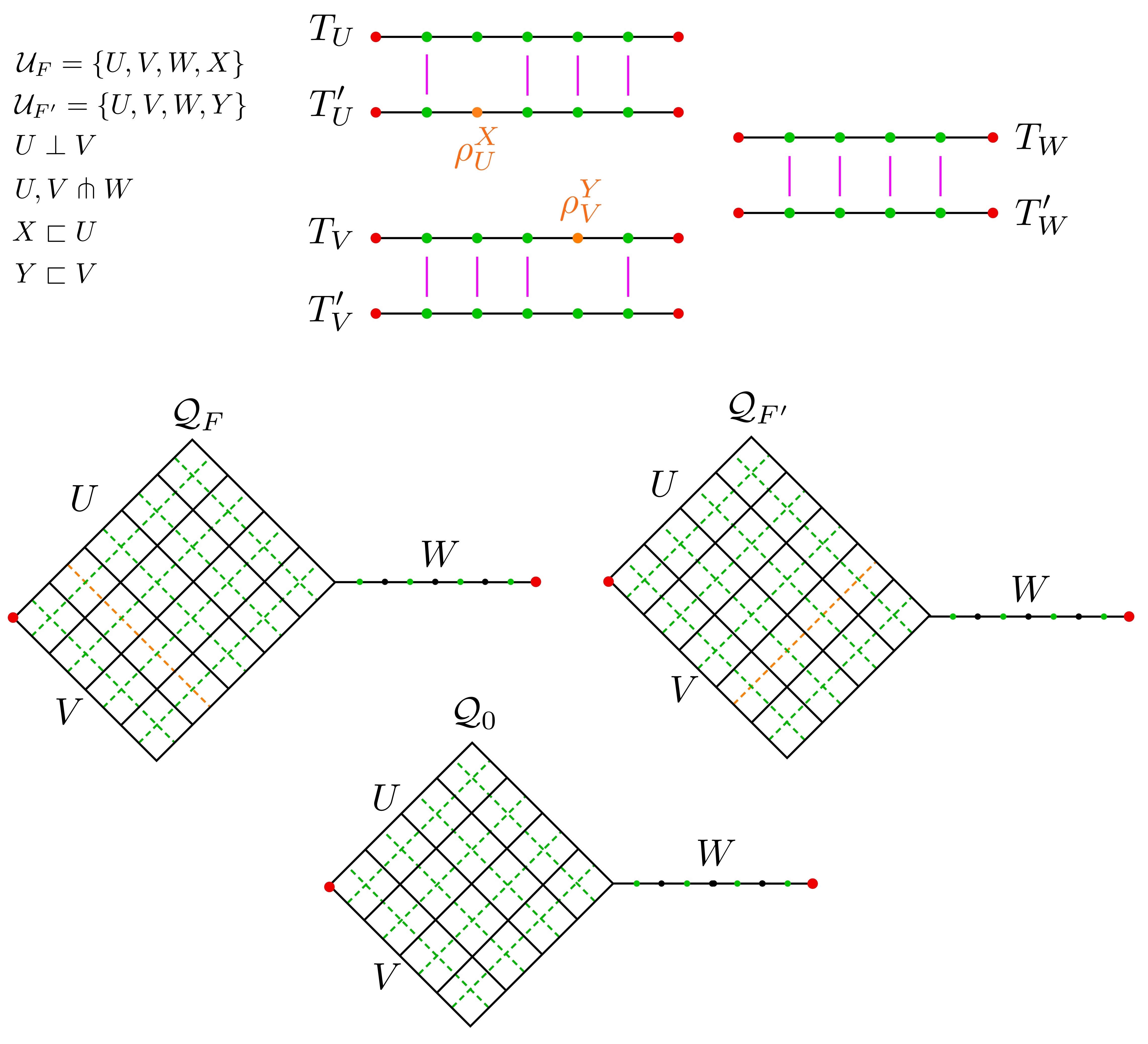}
\caption{A simple example of how subdivision bijections determine hyperplane deletions.  The relevance of $X$ for $F'$ but not $F$ requires deleting a subdivision point to obtain the bijection $j_U: \pt_U \to \pt'_U$ (indicated in pink), and similarly with $Y$ for $j_V:\pt_V \to \pt'_V$.  Deleting subdivision points results in hyperplane deletions when passing from $\QQ_F$ and $\QQ_{F'}$ to $\QQ_0$.  Note that since $X, Y \notin \UU_F \cap \UU_{F'}$, neither domain determines any subdivision points.}\label{fig:trees to cubes}

\end{figure}

\textbf{Commutativity:} It remains to prove that $\Phi_{F, \pt_0}: \QQ_{F, \pt_0} \rightarrow \XX$ and $\Phi_{F', \pt'_0} \circ h: \QQ_{F, \pt_0} \rightarrow \XX$ are the same up to a bounded error depending only on $k$ and the ambient HHS structure.

By the distance formula (Theorem \ref{thm:distance_formula}), it suffices to show that for each $0$-cube $x \in \QQ_{F, \pt_0}$, its respective images $\Phi_{F, \pt_0}(x)$ and $\Phi_{F', \pt'_0} \circ h(x)$ have coarsely the same projections to $\CC(V)$ for each $V \in \UU(F) \cap \UU(F')$. 

Recall from the end of Subsection \ref{subsec:constr} that the maps $\Phi_{F, \pt_0}$ and $\Phi_{F', \pt'_0}$ are defined domain-wise by intersecting certain collections of half-trees of $T^V_F$ and $T^V_{F'}$ for each $V \in \mathfrak S$, and hence the same is true for $\Phi_{F', \pt'_0} \circ h$.  By chasing the relevant definitions, of $h$ and $j$ especially, we see that the collections of  half-trees involved in the definition of $\Phi_{F,\pt_0}(x)$ and $\Phi_{F', \pt'_0} \circ h(x)$ are in bijection with each other, with corresponding half-trees lying within bounded Hausdorff distance depending only on $k$ and the ambient HHS structure.  Hence their intersections in $T^V_{F}$ and $T^V_{F'}$ coarsely coincide, as required.  This completes the proof of the proposition.
\end{proof}

\subsection{Proof of Theorem \ref{thm:stable_cubulations}}\label{subsec:finish cubulations}

Let $F, F'$ be as in the statement.

The CAT(0) cube complexes $\QQ_F, \QQ_{F'}$ are constructed in Subsection \ref{subsec:constr}, using subdivisions $\pt=\wp(F), \pt'=\wp(F')$ within each of the relevant stable trees produced in Section \ref{sec:stable trees}. Items (1) and (2) are proven in \cite{BHS:quasi}, as recalled in Theorem \ref{thm:qi_to_cube_cplx} above.

For (3), we may define a map $\psi_F:F \to \QQ_F$ as in \cite[Proof of Theorem 2.1]{BHS:quasi}.  For each $f \in F$, let $\psi_F(f)$ be the orientation of the walls on hull $H_{\theta}(F)$ obtained by choosing, for each wall $(W^V_{p,+},W^V_{p,-})$, the halfspace containing $f$.  We define $\psi_{F'}$ similarly.  That $\psi_F, \psi_{F'}$ satisfies (3) now follows from Theorem \ref{thm:qi_to_cube_cplx}.

We now prove the stability statements. We will only consider the case that $g$ is the identity for the same reason as in Remark \ref{rem:equivariance}, namely that what we have to prove is that the choices that we made along the way only affect the output in the way predicted by the statement, and such choices can be ``translated'' by automorphisms. The choices we are referring to are those of the trees $T^V_F$, of the evenly spaced subdivisions, and of points projecting coarsely in specified places in the various $\mathcal C(Y)$, as in Theorem \ref{thm:qi_to_cube_cplx}.\ref{item:realization}.

Let $\pt_0 \subset \pt$ and $\pt'_0 \subset \pt'$ be the refinements provided by Proposition \ref{prop:almost bijection of p}, where at most $N$ subdivision points are deleted.

The cube complex $\QQ_0$ as in the statement of Theorem \ref{thm:stable_cubulations} can
be taken to be either of the cube complexes $\QQ_{F,\pt_0}$ or $\QQ_{F', \pt'_0}$ that are
produced by Proposition \ref{prop:isomorphism of cc}.  So take $\QQ_0 = \QQ_{F',\pt'_0}$ and
let $\Phi_0: \QQ_0 \rightarrow \XX$ be the map $\Phi_{F', \pt'_0}:\QQ_{F',\pt'_0} \to \XX$
given by Proposition \ref{prop:isomorphism of cc}.
Finally, let $\eta' \equiv h':\QQ_{F'} \to \QQ_0$ be as
in Proposition \ref{prop:isomorphism of cc} and define $\eta: \QQ_{F} \to
\QQ_0$ by
$\eta = \hat{h} \circ h$.

By Proposition \ref{prop:isomorphism of cc}, these maps each satisfy the required properties and the right part of Diagram \ref{eq:Phi diagram} commutes up to bounded error, as required.

To see exact commutativity of the left part of the diagram, it remains to prove that $\eta \circ \psi_F \circ
\iota_F = \eta' \circ \psi_{F'} \circ \iota_{F'}$.  Recall that $\psi_F: F \rightarrow \QQ_F$
is defined by sending $f$ to the coherent orientation on the wallspace defined by $\pt$
which, for each $p \in \pt(V)$, chooses the half-tree of $T^V_F\ssm p$ containing $\beta^V_F(f)$, 
for each $V \in \UU(F)$. Since $h$ is a hyperplane deletion map, $h\circ \psi_F$ makes, for $p\in
\pt_0$, the same choice as $\psi_F$.

Fix $f\in F$ (the argument for $f\in F'$ is similar).  Then the two sides of the equation are coherent orientations on the wallspace
defining $\QQ_{F',\pt'_0}$, and we have to check that they coincide on every
halfspace. Pick $p\in \pt_0$ and let $p' = j_V(p)$, the map defined in Proposition
\ref{prop:almost bijection of p}. 
As above the orientation 
of $h \circ \psi_F \circ \iota_F(f) = h \circ \psi_F (f)$
on the wall associated to $p$
is the one that chooses the half-tree of
$T^V_{F}\ssm p$, call it $(T^V_F)_{p,+}$,  that contains $\beta^V_F(f)$.
The map $\hat h$, by the construction in Proposition \ref{prop:isomorphism of cc},
takes this to the orientation that chooses the half-tree  of $p'$ in $T^V_{F'}$
given by $j_p((T^V_F)_{p,+})$ (where $j_p$ is the bijection of half-trees provided by Proposition \ref{prop:almost bijection of p}).
Letting $f'=\iota_{F'}(f)$ we have
 $d_\XX(f',f)\leq 1$, so part (4) of
Proposition \ref{prop:almost bijection of p} tells us that
$\beta^V_{F'}(f') \in j_p((T^V_F)_{p,+})$, which means that
$j_p((T^V_F)_{p,+})$ is the half-tree selected by $\eta'\circ\psi_{F'}(f')$.
Thus we conclude that $\eta(\psi_F(\iota_F(f))) = \eta'(\psi_{F'}(\iota_{F'}(f)))$. 

This completes the proof of Theorem \ref{thm:stable_cubulations}.\qed

\section{Generalizing normal paths to find barycenters} \label{sec:normal paths}

In this section, we describe a variation on the ``normal paths'' construction due to Niblo-Reeves \cite{NibloReeves}.  We remark that they used these normal paths to build a biautomatic structure for any cubical group, with the paths playing the central role of the bicombing in that structure.

For the case of two points, our construction gives a ``symmetrized'' version of the Niblo-Reeves construction. The main difference with their construction, however, is that ours is adapted to allow for multiple points, which we need for our barycenter application (Theorem \ref{thm:barycenter main}).

\par\smallskip

The reader may want to refer to Subsection \ref{subsec:cube complexes} for the various notions and notations relating to cube complexes that we will use throughout this section

Let $X$ be a  CAT(0) cube complex, $\HH$ its set of hyperplanes, and $f:P \rightarrow X^{(0)}$ a (not necessarily injective) map from a
finite set $P$ into the vertices of $X$.  Roughly, we will find a barycenter for the set $f(P)$ in $X$ by an iterative sequence
of contractions which behaves stably under hyperplane deletions. The
main result of this section, and the only statement from this section that we need to prove our main theorems, is the following:

\begin{theorem}{thm:stable contraction}
Let $X$ be a  CAT(0) cube complex, and let $\mathcal H$ be its set of hyperplanes.

For each $f:P\to X^{(0)}$, where $P$ is a finite set, there is a finite sequence $\{f_i:P \to
X^{(0)}: i=0,\ldots,n=n_f\}$ with the following properties
\begin{enumerate}
 \item\label{item:ending_diam} $f_0=f$ and $\diam_\infty(f_n(P))\le 1$,
 \item\label{item:step_size} For each $p\in P$ and $0\leq i\leq n-1$ we have
   $d_\infty(f_i(p),f_{i+1}(p))\le 1$,

 \item\label{item:contraction_geod} For each $p$, there is an $\ell^1$-geodesic going through the vertices $f_0(p)$, $f_1(p)$, $\dots, f_n(p)$ in this order,
  \item\label{item:bary_separation} No hyperplane separates every point of $f(P)$ from a point in $f_n(P)$,
\item\label{item:same_image} if $g:Q\to P$ is surjective, then $f_i\circ g=(f\circ g)_i$ for all $i$,
 \item\label{item:deletion_fellow} if $G$ is any
hyperplane of $X$ then the hyperplane deletion map $\res_{\HH\ssm G}:X\rightarrow X(\HH \ssm G)$ satisfies 
$$|n_f-n_{\res_{\HH\ssm G}\circ f}|\leq 1 \indent \text{and} \indent 
d_\infty (  \res_{\HH\ssm G}(f_i(p)), (\res_{\HH\ssm G}\circ f)_i(p) ) \le 1,
$$
\end{enumerate}
\end{theorem}
Recall that $d_{\infty}$ is the metric generated by the sup-metric on each cube in the ambient complex.

The proof of Theorem \ref{thm:stable contraction} occurs in parts over the remainder of this section.  We tie them together in Subsection \ref{subsec:proof of stable contraction} below.

\medskip

We will mostly ignore the ambient cube complexes and focus on the hyperplane set $\HH$ and regard maps $f$ as above as maps associating to $p\in P$ an
orientation on $\HH$.  For each pair $(f,\HH)$ we consider a number of
operations.

\par\smallskip

First, let $\HH_f$ be the set of hyperplanes of $X$ that \emph{separate}
$f(P)$.  That is, $\HH_f$ is the set of hyperplanes $H \in \HH$ for which there exist $p, p' \in P$ so that $f(p)$ and $f(p')$ are on different sides of $H$.

Let
$$\Trim(f,\HH) = (\res_{\HH_f}(f),\HH_f)$$
be the restriction to $\HH_f$.  Note that
the quotient complex $X(\HH_f)$ actually embeds in
$X(\HH)$, and that it is finite even if $X(\HH)$ is not.

If $G$ is a collection of mutually crossing hyperplanes we let
$$
\delete_G(f,\HH) = (\res_{\HH\ssm G}(f),\HH\ssm G)
$$
This ``deletion map'' corresponds to composing $f$ with the quotient
by $G$, that is,
$$\res_{\HH\ssm G} \circ f: P \rightarrow X(\HH \ssm G).$$
We will also write $\res_{\HH\ssm G}(f)$ as $f|_{\HH\ssm G}$,
in a slight abuse of notation.

\subsection{Extremal and transitional hyperplanes}

As stated above, our generalized normal paths give a series of contractions of our set $f(P)$ to a bounded diameter set in $X$.  This is accomplished by iteratively jumping the points of $f(P)$ (and their subsequent contracted images) over a sequence of hyperplanes in $\HH_f$.  Thus we are lead to understand which hyperplanes are next in line to be jumped.

\begin{definition}{}
Let $X$ be a cube complex and $\HH$ its set of hyperplanes.

\begin{itemize}

\item A point $p \in X$ is \textbf{\emph{adjacent}} to a hyperplane $H \in \HH$ if there are no hyperplanes separating $p$ from $H$.

\item A hyperplane $H\in \HH_f$ is \textbf{\em extremal} if on one side of $H$
every point of $f(P)$ is adjacent to $H$ in $X$.  We let $\Ext(f,\HH)\subset \HH_f$ denote the set of extremal
hyperplanes.
\item A hyperplane $H \in \Ext(X,f)$ is \textbf{\em transitional} if it 
is extremal and on one side of $H$ {\em not} every point of $f(P)$ is
adjacent to $H$; we let $\Trans(f,\HH)\subset \Ext(f,\HH)$ be the set
of transitional hyperplanes.
\end{itemize}
\end{definition}

\begin{figure}
\includegraphics[width=0.25\textwidth]{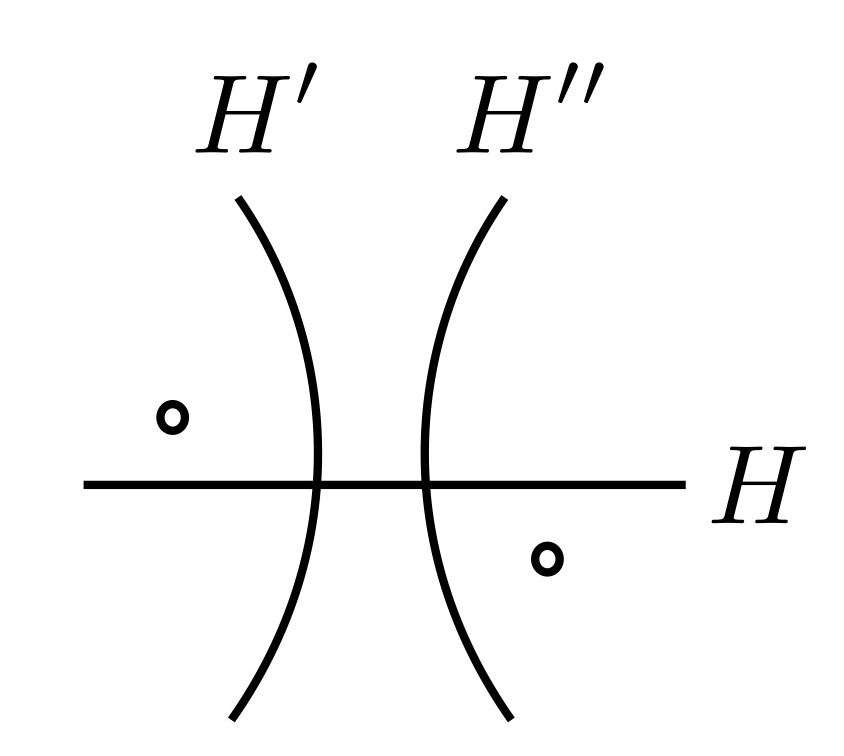}
\caption{$H$ is extremal but not transitional, while $H'$ and $H''$ are both transitional. }\label{fig:extnottrans}
\end{figure}

If $H$ is transitional, we write $P = P_0(H) \sqcup P_1(H) = P_0 \sqcup P_1$ where $f(P_0)$ is
adjacent to $H$ on one side, $f(P_1)$ is on the other side,  and at least one point of $f(P_1)$ is not
adjacent to $H$.

We note that $\Ext(f,\HH)$ is always nonempty when $P$ is nonempty. In fact, for any $p\in P$, it is readily shown that $H$ is extremal whenever $H$ is a hyperplane in $\HH_f$ so that the number of hyperplanes separating $f(p)$ from $H$ is maximal.

Moreover, for any hyperplane $H \in \Ext(f,\HH) \ssm \Trans(f,\HH)$, every point of $f(P)$ must be adjacent to $H$.

\begin{lemma}{lem:trans hp}
The set $f(P)$ is contained in a single cube of $X$ if and only if $\Trans(f,\HH) = \emptyset$.
\end{lemma}

\begin{proof}
First of all, we observe that $f(P)$ being contained in a single cube is equivalent to $\HH_f$ being mutually crossing.  Moreover, it is clear from the definitions that if $\HH_f$ is mutually crossing, then $\Trans(f,\HH) = \emptyset$.

\par\smallskip 

In the other direction, we suppose that $f(P)$ is not contained in a single cube and then produce a transitional hyperplane. The fact that $f(P)$ is not contained in a single cube implies that there exists $p\in P$ and a hyperplane $H\in \HH_f$ with $f(p)$ not adjacent to $H$.

Choose now $p\in P$ and $H\in\HH_f$ so that the number of hyperplanes separating $f(p)$ from $H$ is maximal; the argument above implies that the number of these hyperplanes is positive, so that $f(p)$ is not adjacent to $H$. As observed before the lemma, it follows that $H$ is extremal, and since $f(p)$ is not adjacent to $H$, we have that $H$ is in fact transitional, as required.
 \end{proof}

\subsection{The move sequence}

We now build our sequence of contractions, which we call moves.

\par\smallskip

Roughly speaking, a \textbf{{move}} is an operation on $(f,\HH)$ in which, for every $H\in
\Trans(f,\HH)$, the points of $P_0(H)$ cross $H$ to the opposite
side. The resulting pair $(f_1,\HH) = \move(f,\HH)$ is a map for which
$\HH_{f_1} = \HH_f \ssm \Trans(f,\HH)$, so that the image of the new map $f_1(P)$  is
strictly contained within the subcomplex spanned by $f(P)$.

\realfig{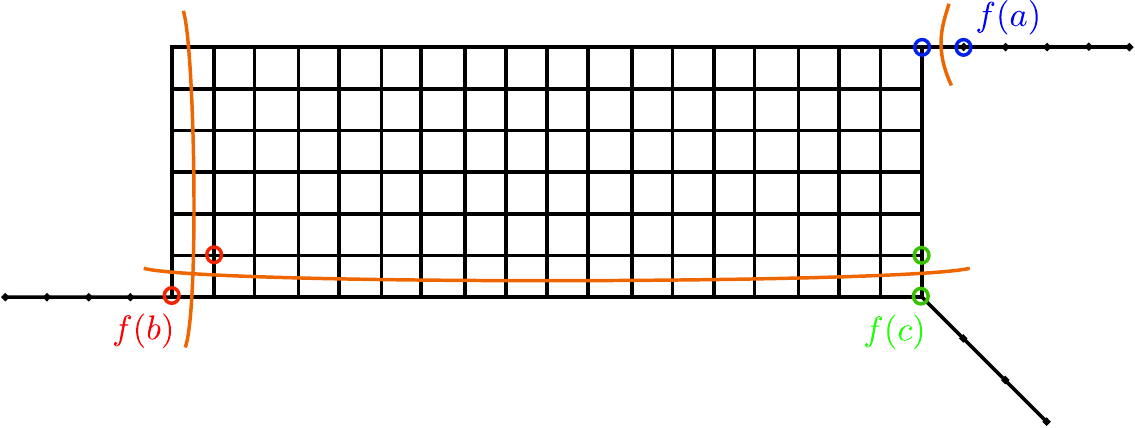}{An example of a single move. $f(P)$ consists of three points and $f_1(P)$, in matching
  colors, are on the other side of the transitional hyperplanes, which are indicated in orange.}

To define $f_1$, we need the following observation:

\begin{lemma}{}\label{lem:adjacent to hyperplane}
 
For each $p\in P$, the set 
$$J(p) = \{H\in \Trans(f,\HH): p\in
P_0(H)\}$$
is mutually crossing.
\end{lemma}

\begin{proof}
Suppose by contradiction that $H_1,H_2\in J(p)$ and that
$H_1$ does not cross $H_2$. Let $s_1$ be the side of $H_1$ on which
$f(P_0(H_1))$ lies (and hence is adjacent) and let $s_2$ be the 
corresponding  side for $H_2$. Since $p\in P_0(H_1)\cap P_0(H_2)$, we have that $s_1$ and $s_2$ intersect. Since $H_1$ does not cross $H_2$, up to swapping indices we have that $H_2$ lies in $s_1$ and separates $H_1$ from $f(P_0(H_2))$, which contains $p$ and thus contradicts the choice of $s_1$.
\end{proof}

We can now define $f_1(p)$.  By Lemma \ref{lem:adjacent to hyperplane}, $f(p)$ is the corner of a unique maximal cube each of whose midcubes is contained in some hyperplane of $J(p)$, and we can choose $f_1(p)$ to be the diagonally opposite
corner.  Equivalently, $f_1(p)$ can be defined as the point obtained by flipping the orientation that
$f(p)$ gives to all the hyperplanes in $J(p)$.  Either definition gives the new map $f_1$.

Note that there may be $p$ for which $J(p)=\emptyset$, and for these $f_1(p) = f(p)$.  Moreover, since the
operation moves all points of $P_0(H)$ across $H$ for each $H\in
\Trans(f,\HH)$, we have the following consequence:

\begin{lemma}{} \label{lem:separators after move}
We have $\HH_{f_1}  = \HH_f \ssm \Trans(f, \HH)$.
\end{lemma}

Now we are ready to consider the {\em move sequence} of $(f,\HH)$. 
For $i\ge 0$ define
$$
(f_i,\HH) = \move^i(f,\HH).
$$

\begin{lemma}{lem:move maps stabilize}
The move sequence $\{f_i\}$ is eventually constant, with constant image all contained in one cube.
\end{lemma}

\begin{proof} By Lemma \ref{lem:separators after move}, the set $\HH_{f_i}$ becomes strictly smaller for each $i$
  as long as $\Trans(f_i,\HH) \ne \emptyset$, so that after a finite number of
  steps we must have $\Trans(f_n,\HH)=\emptyset$, and thereafter $f_i$ are all the same, and their images are all contained in a cube by Lemma \ref{lem:trans hp}. 
\end{proof}

\realfig{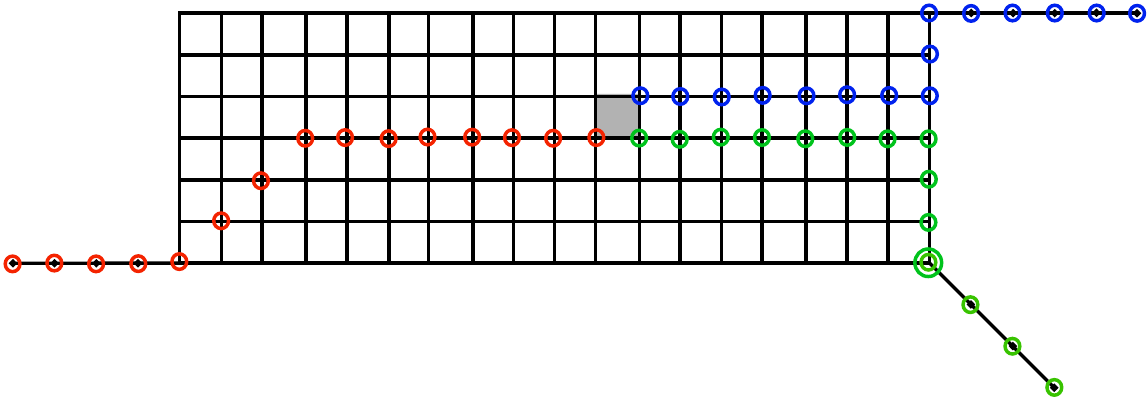}{An example of a move sequence terminating in all points lying in a
  single cube. Concentric circles around a vertex indicate, in this case, a point $p$ with
  $f_i(p) = f_{i+1}(p)$. }

Our main result about this sequence is its stability under the
operation $\delete_G$. This is part of Theorem \ref{thm:stable contraction}, which
we rephrase here in our new language:

\begin{proposition}{fellow travel}
  Let $G$ be a mutually crossing set in $\HH_f$.
  Let 
  $$(f_i,\HH) = \move^i(f,\HH)$$
  and
  $$(f'_i,\HH\ssm G) = \move^i(\delete_G(f,\HH))$$
  be the move sequences for $(f,\HH)$ and $\delete_G(f,\HH)$.
 Then for each $p\in P$,
$$
d_\infty(\res_{\HH\ssm G}\circ f_i(p),f'_i(p)) \le 1.
$$

\end{proposition}

The proposition is in fact a generalization of what we need, since it deals with a mutually crossing set.

\realfig{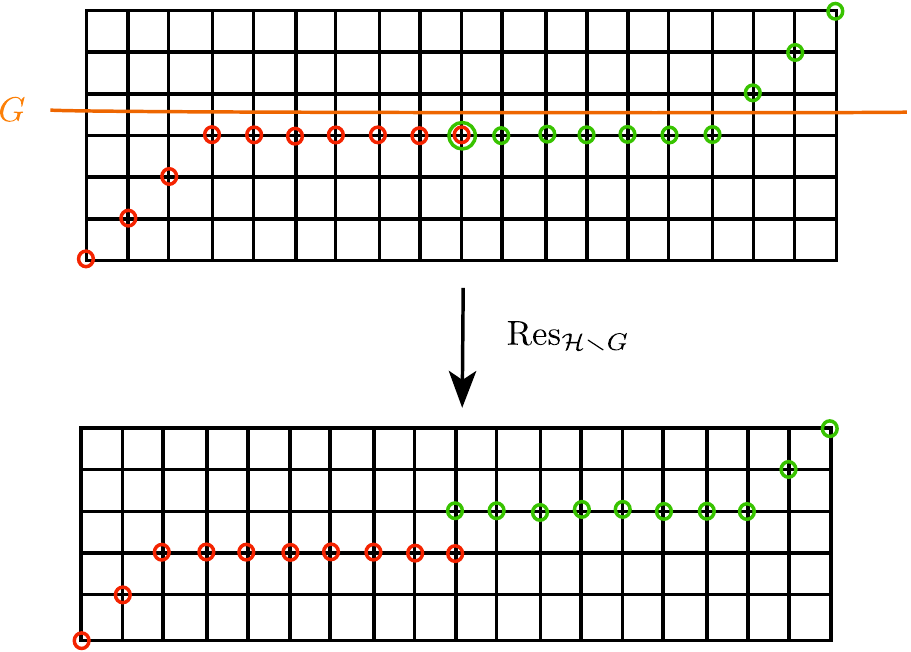}{An illustration of Proposition \ref{fellow travel}, where $P$ has
  two points.}

\subsection{Proof of stability of move sequences}
We begin by studying the structure of the extremal and transitional
hyperplanes for a pair $(f,\HH)$, and the way in which they are
affected by hyperplane deletions.

\begin{lemma}{Ext Trans structure}
Every $J\in \Ext(f,\HH) \ssm \Trans(f,\HH)$ crosses every
hyperplane in $\HH_f\ssm \{J\}$.
\end{lemma}
\begin{proof}
Since $J$ is extremal but not transitional, $f(P)$ is
adjacent to it on both sides (see Figure \ref{fig:extnottrans}). This means that no other hyperplane
can separate any $f(p)$ from $J$, and this implies that every $H\in \HH_f\ssm\{J\}$
crosses $J$.
\end{proof}

The next lemma explains how the extremal and transitional hyperplanes
change after a deletion step:

\begin{lemma}{Ext Trans Delete}
Let $G$ be a  mutually crossing hyperplane set in $\HH_f$. Then
\begin{equation}\label{ext to ext}
\Ext(f,\HH) \ssm G  \  \subset \ \Ext(\delete_G(f,\HH)),
\end{equation}
and
\begin{equation}\label{same difference}
\Ext(\delete_G(f,\HH)) \ssm \Ext(f,\HH) = \Trans(\delete_G(f,\HH)) \ssm \Trans(f,\HH) 
\end{equation}
Moreover, if $G \intersect \Ext(f,\HH) = \emptyset$, then
\begin{equation}\label{ext equality}
  \Ext(f,\HH) = \Ext(\delete_G(f,\HH))
\end{equation}
and
\begin{equation}\label{trans equality}
  \Trans(f,\HH) = \Trans(\delete_G(f,\HH)).
\end{equation}
\end{lemma}

\begin{proof}
The inclusion \eqref{ext to ext} is clear from the definitions.

For \eqref{same difference}, if $J\in \Ext(\delete_G(f,\HH)) \ssm \Ext(f,\HH)$, then $f(P)$ is not
adjacent to $J$ on either side, but on at least one side the only
hyperplanes separating $J$ from $f(P)$ are in $G$. In fact this
happens on exactly one side  since $G$
is a mutually crossing set and its members cannot be separated by
$J$. In particular this means $J\in \Trans(\delete_G(f,\HH))$, and
therefore $J\in\Trans(\delete_G(f,\HH)) \ssm \Trans(f,\HH)$ since
$\Trans(f,\HH) \subset \Ext(f,\HH)$.
This situation is indicated in Figure \ref{Ext-Trans}(a).

Conversely if $J\in \Trans(\delete_G(f,\HH)) \ssm \Trans(f,\HH)$, then either
$f(P)$ is not adjacent to $J$ on either side, in which case we are in
the same situation as above, or $f(P)$ is adjacent to $J$ on both
sides. But in the latter case this adjacency remains true after
deletion of $G$, which contradicts
$J\in\Trans(\delete_G(f,\HH))$.

\realfig{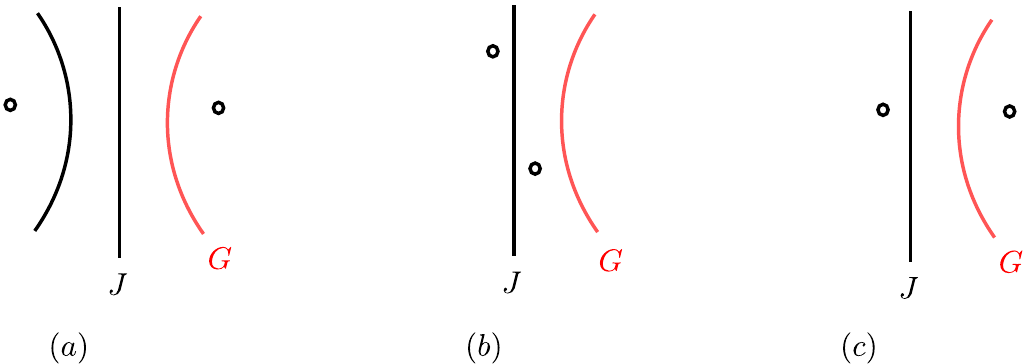}{Changes caused by deleting $G$. In (a), $J$ is not
  in $\Ext(f,\HH)$ but is in $\Trans(\delete_G(f,\HH))$. In (b), $J$
  is in $\Ext\ssm\Trans$ both before and after. In (c), $J$ is in
  $\Trans(f,\HH)$ but not in $\Trans(\delete_G(f,\HH))$. }

We conclude that 
$J\in \Ext(\delete_G(f,\HH))\ssm \Ext(f,\HH)$
if and only if 
$J\in \Trans(\delete_G(f,\HH)) \ssm \Trans(f,\HH)$, which gives
(\ref{same difference}).

In the description of
$J\in \Ext(\delete_G(f,\HH))\ssm \Ext(f,\HH)$, we note that the
hyperplanes of $G$ separating $J$ from $f(P)$ cannot themselves be
separated from $f(P)$ (on the side not containing $J$) by any other
hyperplanes, because this would contradict
$J\in\Ext(\delete_G(f,\HH))$. Thus those hyperplanes are themselves
extremal. We conclude that, if $G\intersect \Ext(f,\HH) = \emptyset$,
then 
$ \Ext(\delete_G(f,\HH))\ssm \Ext(f,\HH) = \emptyset$, giving
(\ref{ext equality}).

 Finally to show (\ref{trans equality}) when $G\intersect \Ext(f,\HH) =
 \emptyset$, note first that (\ref{same difference}) and (\ref{ext
  equality}) imply that $\Trans(\delete_G(f,\HH)) \subset
 \Trans(f,\HH)$. Now if $J\in \Trans(f,\HH) \ssm \Trans(\delete_G(f,\HH))$, then on the side of $J$ where
 $f(P)$ is not adjacent there must only be hyperplanes of $G$
 separating $J$ from $f(P)$, whose deletion makes $f(P)$ adjacent on
 that side (see Figure \ref{Ext-Trans}(c)). But this contradicts the assumption that $G \cap \Ext(f,\HH) = \emptyset$.
 \end{proof}

We can now obtain the following lemma, which in the simplest case
shows that moves and deletions commute.
\begin{lemma}{easy commute}
 In the notation of Proposition \ref{fellow travel}, if $G \cap \Ext(f, \HH) = \emptyset$, then the following
 diagram commutes:
$$  
\begin{tikzcd}
(f,\HH) \arrow[r,"\delete_G"] \arrow[d,"\move"] &  (f',\HH\ssm G) \arrow[d,"\move"]
\\
(f_1,\HH) \arrow[r,"\delete_{G}"] & (f_1',\HH\ssm G)
\end{tikzcd}
$$
 
\end{lemma}
\begin{proof}
By Lemma \ref{Ext Trans Delete}, we have
$    \Trans(f,\HH) = \Trans(\delete_G(f,\HH)).$ This means
that the $\move$ operation
on both $(f,\HH)$ and $(f',\HH\ssm G)$ affects exactly the same set of
hyperplanes, and in exactly the same way. That is,
for $J\in \Trans(f,\HH)$, the subset $P_0\subset P$ whose $f$-image is
on the adjacent side of $J$ is also the subset whose $f'$-image is on
the adjacent side, since the deletion of $G$  does not affect this.
The lemma follows. 
\end{proof}

Now we consider the general situation, where some hyperplanes of $G$
may be in $\Ext(f,\HH)$. 

\begin{lemma}{long commute}
  Let $G$ be a mutually crossing set in $\HH$ and define

$$
  G' = \left(G \union \Trans(\delete_G(f,\HH))\right) \ssm \Trans(f,\HH)
$$
and
$$
K = \Trans(f,\HH) \ssm (G \union \Trans(\delete_G(f,\HH))).
$$
Then $G'$ is a mutually crossing set,  every hyperplane $H\in K$ crosses all hyperplanes of
 $\HH_f\ssm (G\cup \{H\})$, and there exists a map $g:P \rightarrow X\left(\HH_f \ssm (G \cup G' \cup K)\right)$ so that the following diagram commutes:
 $$
 \begin{tikzcd}
   (f,\HH) \arrow[rr,"\delete_G"] \arrow[d,"\move"] & & (f',\HH\ssm G) \arrow[d,"\move"]
   \\
   (f_1,\HH) \arrow[r,"\delete_{G'}"] &     (f_1|_{\HH\ssm G'},
   \HH\ssm G') \arrow[bend right=30,ddr,"\Trim"]
       & (f_1',\HH\ssm G) \arrow[d,"\delete_K"] \\
& &     (f_1'|_{\HH\ssm (G\union K)}, \HH\ssm (G\union K))
   \arrow[d,"\Trim" ] \\
   & & (g,\HH_f \ssm (G\union G' \union K) )
 \end{tikzcd}
 $$
\end{lemma}

\begin{proof}
First we prove that $G'=\{G'_1,\ldots,G'_k\}$ is a mutually crossing set. Suppose that
$G'_i$ and $G'_j$  do not cross.
Then they cannot both be in $G$ by hypothesis.

\realfig{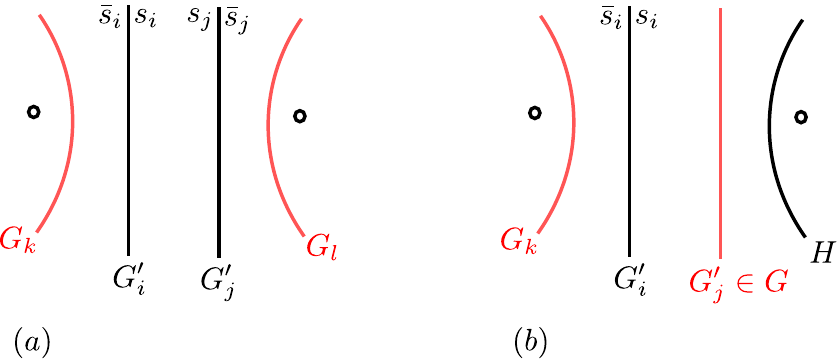}{The contradictions arising from hyperplanes in $G'$
  not being mutually crossing.}

Assume first that $G'_i,G'_j$ are in 
$\Trans(\delete_G(f,\HH))\ssm\Trans(f,\HH)$, which is the same as
$\Ext(\delete_G(f,\HH))\ssm\Ext(f,\HH)$ by Lemma \ref{Ext Trans
  Delete}.
Let $s_i$ be the side of $G'_i$ that contains $G'_j$, and
define $s_j$ similarly (Figure \ref{Gprime}(a)). Then $f'(P)$ cannot be adjacent to $G'_i$ on
the side $s_i$, since part of $f'(P)$ is separated from $G'_i$ by $G'_j$. Therefore,
since $G'_i \in \Ext(\delete_G(f,\HH))$, $f'(P)$ must be adjacent to
$G'_i$ on the opposite side, $\bar s_i$. Similarly 
$f'(P)$ must be adjacent to
$G'_j$ on the opposite side, $\bar s_j$.
Note that $\bar s_i$ and $\bar s_j$ are disjoint. On the other
hand since $G'_i$ and $G'_j$ are not in $\Ext(f,\HH)$, there must be
some $G_k\in G$ in $\bar s_i$ separating $G'_i$ from some point of $f(P)$,
and similarly 
$G_l\in G$ in $\bar s_j$ separating $G'_j$ from some point of $f(P)$. But
this is not possible since $G_k$ crosses $G_l$.

Now assume $G'_i\in \Ext(\delete_G(f,\HH))\ssm \Ext(f,\HH)$ and $G'_j\in G\ssm\Trans(f,\HH)$
 (Figure \ref{Gprime}(b)). If $G'_j\in \Ext(f,\HH) \ssm \Trans(f,\HH)$
then it crosses $G'_i$ by Lemma \ref{Ext Trans structure}, so
we may assume $G'_j\notin \Ext(f,\HH)$.
Define $s_i$ as before. Now since $G'_j\notin \Ext(f,\HH)$, on the
side of $G'_j$ contained in $s_i$, there must be another hyperplane $H$
separating $G'_j$ from a point of $f(P)$. This $H$ cannot be in $G$
since $G'_j$ crosses $H$, so it is not deleted and hence $f'(P)$ is
not adjacent to $G'_i$ on the $s_i$ side. As above there must
therefore be a $G_k\in G$ on the $\bar s_i $ side. This $G_k$ cannot
cross $G'_j$, again a contradiction.

To see that each hyperplane of $K$ crosses all other hyperplanes of $\HH_f\ssm G$,
note that $K$ is contained in $ \Ext(\delete_G(f,\HH)) \ssm
\Trans(\delete_G(f,\HH))$ by part (\ref{ext to ext}) of Lemma \ref{Ext Trans Delete},
and then use Lemma \ref{Ext Trans structure}.

To finish the argument, we claim that all we have to check
is that the set of hyperplanes that are either transitional for a
$\move$ operation or deleted along each side
of the diagram is the same. This is because of the following relations
which follow directly from the definitions: 
$$
\Trim\circ\move (f,\HH) = \delete_{\Trans(f,\HH)} \circ \Trim (f,\HH)
$$
and
$$
\Trim\circ\delete_G = \delete_G\circ\Trim. 
$$
With these relations, we can simplify each side of the diagram to a single $\Trim
\circ \delete_{V}$ where $V$ is the union of hyperplanes from all the
deletion and Move steps on that side. 

Thus, comparing the left side of the diagram with the top arrow and
right side, it remains to check that
\begin{equation}\label{hyperplane sets}
\Trans(f,\HH) \union G' \stackrel{?}{=} G \union \Trans(\delete_G(f,\HH)) \union K.
\end{equation}
But, using the definitions of $G'$ and $K$, we see that both sides are
equal to $G \union \Trans(\delete_G(f,\HH)) \union \Trans(f,\HH)$.
\end{proof}

Every hyperplane $H \in \HH_f$ becomes extremal at some point along the move sequence.  We want to understand how deletions affect when this occurs.  Note that when a hyperplane $H \in \HH_f$ becomes extremal, it need not become transitional, and what happens can change with the deletion of a nearby hyperplane.

For a hyperplane $H\in \HH_f$, define
$$
e_H(f,\HH)
$$
to be the first index $i$ such that $H\in \Ext(f_i,\HH)$.

\begin{lemma}{time shift}
  For any mutually crossing set $G\subset \HH_f$, and any
  $H\in \HH_f\ssm G$, we have
\begin{equation}\label{time shift eq}
  e_H(f,\HH) = e_H(\delete_G(f,\HH)) + \delta
\end{equation}
  for some $\delta \in \{0,1\}$. 
\end{lemma}

\begin{proof}
If $e_H(f,\HH) = 0$ then $H$ is already in $\Ext(f,\HH)$, which
implies $H\in\Ext(\delete_G(f,\HH))$ by Lemma \ref{Ext Trans Delete}
(\ref{ext to ext}). The equality (\ref{time shift eq}) follows with $\delta=0$. 

Thus we may assume $e_H(f,\HH) > 0$. Suppose that
$e_H(\delete_G(f,\HH))=0$. This means that $H \in
\Ext(\delete_G(f,\HH)) \ssm \Ext(f,\HH)$,
which means (as in the proof of Lemma \ref{long commute}) that there are some elements $G_i$
of $G$ which separate $H$ from $f(P)$ on one side $s_i$, so that
$f(P)$ is adjacent to $G_i$ on the $s_i$ side. But $f(P)$ is not
adjacent to $G_i$ on the other side because $H$ is there, which means
$G_i\in \Trans(f,\HH)$.  But this implies that, in $\move(f,\HH)$,
all $G_i$ as above are no longer in the set of separating hyperplanes, and hence
$H\in \Ext(\move(f,\HH))$, so $e_H(f,\HH) = 1$. This gives (\ref{time
  shift eq}) with $\delta = 1$.

From now on we can assume $e_H(f,\HH)>0$ and $e_H(\delete_G(f,\HH))>0$, 
and prove the statement by induction on the cardinality of $\HH_f$
(the case $|\HH_f|=2$ is easy, and already covered by the previous paragraphs).

Let $(f',\HH') = \delete_G(f,\HH)$, $(f_1,\HH) = \move(f,\HH)$,
and $(f'_1,\HH') = \move (f',\HH')$.
By definition (since $e_H(f,\HH)>0$ and $e_H(f',\HH)>0$) we have
$$e_H(f,\HH)= e_H(f_1,\HH)+1$$
and
$$e_H(f',\HH') =  e_H(f'_1,\HH') +1.$$
Thus it will suffice to prove
\begin{equation}\label{time shift step}
e_H(f_1,\HH) = e_H(f'_1,\HH') + \delta.
\end{equation}

  Consider the warmup case when $G\intersect \Ext(f,\HH) = \emptyset$.  
By Lemma \ref{easy commute},
$$
(f'_1,\HH') = \delete_G(f_1,\HH).
$$
Since $|\HH_{f_1}| < |\HH_f|$, the inductive hypothesis gives us
$$
e_H(f_1,\HH) = e_H(f'_1,\HH') +\delta
$$
for $\delta = 0$ or $1$, proving (\ref{time shift step}) and hence
(\ref{time shift eq}).

Now in the general case, we use the diagram of Lemma \ref{long commute}.
Note first that the value of $e_H$ is not affected by a $\Trim$
operation. This is because $\Trim$ does not affect the set $\HH_f$, or
the membership in $\Ext$ or $\Trans$.

The value of $e_H$ is also unaffected by the $\delete_K$ arrow on the
right side. This is because each hyperplane $H \in K$ crosses
every hyperplane in $\HH_f \ssm (G \cup \{H\})$, which implies that any hyperplane in $K$ cannot affect the membership in $\Ext$ or $\Trans$ of any other hyperplane in $\HH_f \ssm G$. Therefore, we see that $\delete_K$ commutes with
the $\move$ sequence on $(f'_1,\HH_f\ssm G)$.

The remaining arrow is labeled by $\delete_{G'}$, and $G'$ is a
mutually crossing set. Thus by induction we know
$$
e_H(f_1,\HH) = e_H(\delete_{G'}(f_1,\HH)) + \delta
$$
for $\delta = 0$ or $1$. Again the equality (\ref{time shift step})
follows. 
\end{proof}

We are now ready to prove the stability result for move sequences.

\subsubsection*{Proof of Proposition \ref{fellow travel}}

We need to prove the following statement: For each $i$ and $p \in P$, the set of $H\in \HH_f\ssm
G$ such that $H$ separates $f_i(p)$ from $f'_i(p)$ is mutually
crossing.

Note that when a sequence $(f_i(p))$ crosses a hyperplane $H$, it can
only happen in the transition from $f_j$ to $f_{j+1}$ where
$j=e_H(f,\HH)$. Moreover it must be that $H\in \Trans(f_j,\HH)$, and
that $f_j(p)$ is on the side of $H$ where $f_j(P)$ is adjacent.
If the sequence $(f_i(p))$ does not cross $H$, then $H$ will either remain in $\HH_{f_i}$ for
all $i\ge j$, or it will be crossed only by points on the other side
and not by $f_i(p)$.   The same holds for the sequence $(f'_i)$. 

Now suppose that $H_1,H_2$ separate $f_i(p)$ from $f'_i(p)$, for some $i$, and
$H_1$ does not cross $H_2$. Since $f_0(p)$ and $f'_0(p)$ are on the same side of
both hyperplanes we may assume that $H_1$ separates them from
$H_2$. Hence there exists a (different) time $i$ at which one of them is still on the original side of
$H_1$, whereas the other has crossed both hyperplanes; fix such $i$.

Suppose first that $f'_i(p)$ is the one which lies on the other side of $H_2$. Let
$j=e_{H_1}(f',\HH)$. Then since $H_1$ does not cross $H_2$, $f'_{j+1}(p)$ has not
yet crossed $H_2$. Thus we must have
$$
j = e_{H_1}(f',\HH) < e_{H_2}(f',\HH) < i
$$
and $H_1\in \Trans(f'_j,\HH)$, with $f'_j(p)$ on the side where
$f'_j(P)$ is adjacent to $H_1$.

Now $k=e_{H_1}(f,\HH) \le j+1$ by Lemma \ref{time shift}.  Since $j+1 < i$ and
$f_i(p)$ has not crossed $H_1$, it must be that the side of $H_1$ where
$f_k(P)$ is adjacent is the one opposite from $f_k(p)$, the one
containing $H_2$.

But this means that $H_2$ can no longer be in $\HH_{f_k}$, which can
only be if $e_{H_2}(f,\HH) \le k-1 \le j$. Thus, again by
Lemma \ref{time shift}, $e_{H_2}(f',\HH) \le j$, which is a
contradiction.

We conclude that $H_1$ does not cross $H_2$, which is what we wanted.

The case where $f_i(p)$ crosses the hyperplanes and $f'_i(p)$ does not
is handled similarly.  The main difference is that in this case, instead of using the inequality $e_{H_1}(f,\HH) \le e_{H_1}(f',\HH)+1$, we use $e_{H_1}(f',\HH) \le e_{H_1}(f,\HH)$, which creates a ``$+1$'' there, which is then lost in the other application of Lemma \ref{time shift}.

\qed

\subsection{Completing the proof of Theorem \ref{thm:stable contraction}}\label{subsec:proof of stable contraction}

Property (\ref{item:ending_diam}) follows from Lemma \ref{lem:trans hp}, while property (\ref{item:step_size}) follows from the construction, where for both properties we use the fact that cubes have diameter $1$ in the $d_\infty$ metric. Property (\ref{item:same_image}) is also easily seen to hold by construction, and more specifically it follows from the fact that the sets $\Ext(f,\HH)$ and $\Trans(f,\HH)$ only depend on the image of $f$ (and hence that this will remaining true throughout the move sequence).

For property (\ref{item:contraction_geod}), observe that for a fixed $p \in P$, that no hyperplane $H$ separates $f_i(p)$ from $f_{i+1}(p)$ for two different values of $i$.  It follows that any combinatorial path obtained by concatenating a choice of geodesics from $f_i(p)$ to $f_{i+1}(p)$ for $0 \leq i < n$ is an $\ell^1$-geodesic in $X$.  This proves (\ref{item:contraction_geod}).

For property (\ref{item:bary_separation}), by definition of the contraction sequence, the hyperplanes that separate any given $f(p)$ from $f_n(q)$ for $p,q \in P$ are contained in $\HH_f$.  On the other hand, a hyperplane $H \in \HH_f$ cannot separate every point in $f(P)$ from any fixed vertex of $X$, because there are elements of $f(P)$ on both sides of $H$, by definition of $\mathcal H_f$. Hence, there can be no hyperplane separating $f(P)$ from a point in $f_n(P)$.  This gives property (\ref{item:bary_separation}).

Property (\ref{item:deletion_fellow}) is a direct consequence of Proposition \ref{fellow travel} and Lemma \ref{lem:move maps stabilize}.  This completes the proof.\qed

\section{Proofs of main theorems} \label{sec:main proofs}

We are now almost ready to prove our main theorems, Theorem \ref{thm:barycenter main} and Theorem \ref{thm:bicombing main}.  The main bit of work here is Proposition \ref{prop:stable contractions in HHS}, which compiles our preceding stability results into a useful form for our current purposes.

\begin{proposition}{prop:stable contractions in HHS}
Let $(\cuco X, \mathfrak S)$ be a $G$-colorable HHS for $G <  \mathrm{Aut}(\mathfrak S)$.
 For any $k \in \mathbb N$, there exists $K_3 = K_3(k, \mathfrak S)>0$ so that the following holds:

Suppose that $F, F' \subset \cuco X$ are finite subsets satisfying $|F|, |F'| \leq k$, let $g\in G$,
 and suppose that $d_{Haus}(gF,F')\leq 1$. Choose any map $\iota_{F}:F\sqcup F'\to F$ so that $\iota_{F}(f)=f$ if $f\in F$ and $d_\XX(g\iota_F(f),f)\leq 1$ if $f\in F'$. Also, choose a map $\iota_{F'}:F\sqcup F'\to F'$ such that $\iota_{F'}(f)=f$ if $f\in F'$ and $d_\XX(gf,\iota_{F'}(f))\leq 1$ if $f\in F$. Consider
\begin{itemize}
\item The cube complexes $\QQ_F, \QQ_{F'}$ produced by Theorem \ref{thm:stable_cubulations} with associated maps $\Phi_F, \Phi_{F'}$ to $\cuco X$, and $\psi_F,\psi_{F'}$ from $F,F'$ to $\QQ_F, \QQ_{F'}$ ;
\item The sequences of contractions $\{(\psi_F)_i=\psi_i\}_{i\leq n_{\psi_F}}$ and $\{(\psi_{F'})_i=\psi'_i\}_{i\leq n_{\psi_{F'}}}$ produced by Theorem \ref{thm:stable contraction}.  Set $n_{\psi_F}=n_F$ and $n_{\psi_{F'}}=n_{F'}$.
\end{itemize}

Then 
\begin{enumerate}
\item $|n_F - n_{F'}| < K_3$, and \label{item:bound on n difference}
\item For each $i \in \left\{1, \dots, \max\{n_F, n_{F'}\}\right\}$ and any $f \in F\sqcup F'$, we have \label{item:contraction distance}
$$d_{\cuco X}\left(g\circ\Phi_F \circ \psi_i(\iota_F(f)), \Phi_{F'} \circ \psi'_i(\iota_{F'}(f))\right) \leq K_3,$$
\item $\diam_{\cuco X}(\Phi_F(\psi_{n_F}(F)) < K_3$. \label{item:diameter bound}
\end{enumerate}
\end{proposition}

More visually, item \ref{item:contraction distance} says that the following diagram coarsely commutes:

  \begin{equation}
  \begin{tikzcd}
    & F\arrow[rr,"\psi_i=(\psi_F)_i"] & &\QQ_F \arrow[dr,"g\circ\Phi_F"]  &  \\
    F\sqcup F' \arrow[ur,"\iota_{F}"]\arrow[dr,"\iota_{F'}\ \ \ " below] & & & & \XX \\
    & F'\arrow[rr,"\psi'_{i}=(\psi_{F'})_i"]& &\QQ_{F'}\arrow[ur,"\ \ \ \Phi_{F'}" below] & \\
  \end{tikzcd}
  \end{equation}

\begin{proof}
We will use the output and notation of Theorem \ref{thm:stable_cubulations}, and in particular the CAT(0) cube complex $\mathcal Y_0$ obtained from both $\mathcal Y_F$ and $\mathcal Y_{F'}$ by collapsing at most $N = N(k, \mathfrak S)>0$ hyperplanes, with hyperplanes collapse maps $h,h'$.

We will also use the notation of Theorem \ref{thm:stable contraction}, in particular the notation $\{f_i:i\leq n_f\}$ for the sequence of maps starting with $f$ and ending with a map with bounded image.

We have $\psi:=h\circ\psi_F\circ\iota_F=h'\circ\psi_{F'}\circ\iota_{F'}$, as stated in Theorem \ref{thm:stable_cubulations}. By Theorem \ref{thm:stable contraction}, composing, say, $\psi_F$ with a hyperplane deletion map affects the length of the corresponding sequence of maps by at most 1. In particular, we have $|n_F-n_\psi|\leq N$ and, similarly, $|n_{F'}-n_\psi|\leq N$ (notice that $n_{\psi_F}=n_{\psi_F\circ \iota_F}$ by Theorem \ref{thm:stable contraction}(\ref{item:same_image}), and a similar statement holds for $F'$). Hence, conclusion \ref{item:bound on n difference} holds for any $K_3$ larger than $2N$.

\par\smallskip

We now prove conclusion \eqref{item:contraction distance}. By Theorem \ref{thm:stable_cubulations}, Diagram \eqref{Phi diagram} commutes with error at most $K = K(k, \mathfrak S)$. For convenience, we reproduce the diagram here:

  \begin{equation}
  \begin{tikzcd}
    & F\arrow[r,"\psi_F"]&\QQ_F \arrow[dr,"g\circ\Phi_F"] \arrow[d,"\eta" left] &  \\
    F\sqcup F' \arrow[ur,"\iota_{F}"]\arrow[dr,"\iota_{F'}\ \ \ " below] & &\QQ_0 \arrow[r,"\Phi_0"] & \XX \\
    & F'\arrow[r,"\psi_{F'}"]&\QQ_{F'}\arrow[ur,"\ \ \ \Phi_{F'}" below] \arrow[u,"\eta'"] & \\
  \end{tikzcd}
  \end{equation}

For any $f\in F\sqcup F'$ we have
$$d_{\cuco X}(g\circ\Phi_F\circ(\psi_F)_i(\iota_F(f)),\Phi_0\circ \eta\circ(\psi_F)_i(\iota_F(f)))\leq K.$$
By Theorem \ref{thm:stable contraction}, we have $d_{\infty}(\eta\circ(\psi_F)_i(\iota_F(f)), (\eta\circ\psi_F)_i(\iota_F(f)))\leq N$, and hence $$d_{\cuco X}(\Phi_0\circ \eta\circ(\psi_F)_i(f),\Phi_0\circ (\eta\circ\psi_F)_i(f))\leq K'=K'(k, \mathfrak S)$$
since $\Phi_0$ is a quasi-isometric embedding with controlled constants (and the dimension of $\mathcal Y_0$ is bounded in terms of $\mathfrak S$ by Theorem \ref{thm:stable_cubulations}, so that the $\ell^\infty$ and $\ell^1$ metrics on it are uniformly quasi-isometric).  The triangle inequality then gives

$$d_{\cuco X}(g\circ \Phi_{F}\circ(\psi_{F})_i(\iota_{F}(f)),\Phi_0\circ(\eta\circ\psi_{F})_i(\iota_{F}(f)))\leq K'+K.$$

Similarly, we get
$$d_{\cuco X}(\Phi_{F'}\circ(\psi_{F'})_i(\iota_{F'}(f)),\Phi_0\circ(\eta'\circ\psi_{F'})_i(\iota_{F'}(f)))\leq K'+K.$$
By Theorem \ref{thm:stable contraction}(\ref{item:same_image}), we have $(\eta'\circ\psi_{F'})_i(\iota_{F'}(f))=(\eta'\circ\psi_{F'}\circ\iota_{F'})_i(f)= (\eta\circ\psi_F\circ \iota_F)_i(f)$, and hence conclusion \eqref{item:contraction distance} holds for any $K_3$ larger than $2(K'+K)$.

Finally, to prove \eqref{item:diameter bound}, we now bound the diameter of $\Phi_F(\psi_{n_F}(F)$. Similarly to above, $\Phi_F$ is a quasi-isometric embedding with constants controlled in terms of $k,\mathfrak S$ even when we endow $\QQ_F$ with the $\ell^\infty$ metric. Since $\diam_{\QQ_F}(\psi_{n_F}(F)) \leq 1$ (in the $\ell^\infty$-metric) by Theorem \ref{thm:stable contraction}, this gives the required bound on $\diam_{\cuco X}(\Phi_F(\psi_{n_F}(F))$.
\end{proof}

\subsection{Barycenters: Proof of Theorem \ref{thm:barycenter main}}\label{subsec:barycenters thm}

Our next goal is to prove Theorem \ref{thm:barycenter main}.  To do so, we'll need the precise definition of stable barycenter:

\begin{definition}{barycenter}
  For a metric space $X$ a \textbf{\em stable barycenter map} for $k$ points is a map $\tau:X^k
  \to X$ which is
  \begin{itemize}
    \item  \textbf{\em Permutation invariant}, meaning  $\tau\circ\pi = \tau$ for any $\pi:X^k\to X^k$
      that is a permutation of the factors.
    \item \textbf{\em Coarsely Lipschitz}, meaning there exists $\kappa_1>0$ such that for $x,x'\in X^k$
      $$ d_X(\tau(x),\tau(x')) \le \kappa_1 d_{X^k}(x,x') + \kappa_1.$$
  \end{itemize}
  We further say that $\tau$ is \textbf{\em coarsely equivariant} with respect to a group
  $\Gamma$ acting on $X$ by isometries if there exists $\kappa_1>0$ such that for all $g\in \Gamma$
  $$
  d_X(g\tau(x),\tau(gx)) \le \kappa_1,
  $$
  where $\Gamma$ acts on $X^k$ diagonally.
\end{definition}

We now prove that colorable HHSes admit stable coarsely equivariant barycenters, with the following version slightly more general than Theorem \ref{thm:barycenter main}:

\begin{theorem}{thm:barycenter general}
Let $(\cuco X, \mathfrak S)$ be a $G$-colorable HHS for $G < \mathrm{Aut}(\mathfrak S)$.  Then $\cuco X$ admits coarsely
$G$-equivariant stable barycenters for $k$ points, for any $k\ge 1$. Moreover, the coarse barycenter of a set $F$ is contained in the hierarchical hull of $F$.
\end{theorem}

\begin{proof}
We use the notation from the statement of Proposition \ref{prop:stable contractions in HHS}.

To define a barycenter $\tau(f_1,\dots,f_k)$, we consider $F=\{f_i\}$, set $x_F = \Phi_F(\psi_{n_F}(F))$, and let $\tau(f_1,\dots,f_k)$ be an arbitrary point in $x_F$; we make the choice depending on the set $F$ only, so that permutation invariance is achieved.

This choice does not matter for our purposes since $\diam_{\cuco X}(\Phi_F(\psi_{n_F}(F)) < K_3$ by Proposition \ref{prop:stable contractions in HHS}.

\par\smallskip

Now suppose that $(f'_1,\dots,f'_k)$ is so that $d_{Haus}(\{f'_i\},\{f_i\})\leq 1$, and set $F'=\{f'_i\}$. Without loss of generality, assume that $n_F \geq n_{F'}$, where we note that $n_F - n_{F'} < K_3$ by part \eqref{item:bound on n difference} of Proposition \ref{prop:stable contractions in HHS}.  Part \eqref{item:contraction distance} of Proposition \ref{prop:stable contractions in HHS} now implies that for any $f\in F\sqcup F'$ we have

$$d_{\cuco X}(\Phi_F(\psi_{n_F}(\iota_F(f)), \Phi_{F'}(\psi'_{n_{F}}(\iota_{F'}(f))) < K_3.$$

But $\psi'_{n_{F'}} = \psi'_{n_{F}}$ since $n_F \geq n_{F'}$, so we can conclude that

$$d_{\cuco X}(\Phi_F(\psi_{n_F}(\iota_F(f)), \Phi_{F'}(\psi'_{n_{F'}}(\iota_{F'}(f)))) < K_3.$$

Finally, the fact that $\diam_{\cuco X}(\Phi_F(\psi_{n_F}(F)) < K_3$ and $\diam_{\cuco X}(\Phi_{F'}(\psi'_{n_{F'}}(F')) < K_3$ gives that

$$\diam_{\cuco X}(x_F \cup x_{F'}) < 3K_3.$$

Setting $\kappa_1 = 3K_3$, we get that $\tau$ is $\kappa_1$-coarsely Lipschitz.

\par\smallskip

Finally, coarse equivariance follows similarly, applying Proposition \ref{prop:stable contractions in HHS} with $F'=gF$, as follows. First, as above we can assume $n_F\geq n_{gF}$, for otherwise we can swap the roles of $F$ as $gF$, by considering the automorphism $g^{-1}$. We still have $n_F-n_{gF}<K_3$. Part \eqref{item:contraction distance} of Proposition \ref{prop:stable contractions in HHS} implies that for any $f\in F\sqcup gF$ we have
$$d_{\cuco X}(g\circ\Phi_F(\psi_{n_F}(\iota_F(f)), \Phi_{gF}(\psi'_{n_{gF}}(\iota_{gF}(f))) < K_3.$$
As above, we conclude that
$$\diam_{\cuco X}(g (x_F) \cup x_{gF}) < 3K_3,$$
which completes the proof.
\end{proof}

\subsection{Bicombability: Proof of Theorem \ref{thm:bicombing main}} \label{subsec:bicombings thm}

We begin with the formal definition of bicombing which is appropriate for our context; see \cite{AB95}.  In the following definition, we adopt the convention that if $\phi:[0,a]\rightarrow X$ is a map, then we trivially extend $\phi$ by $\phi(t) = \phi(a)$ for all $t >a$.

\begin{definition}{}\label{defn:bicombing}
A \textbf{\emph{discrete, bounded, quasi-geodesic bicombing}} of a metric space $X$ consists of a family of discrete paths $\{\Omega_{x,y}\}_{x,y \in X}$ and a constant $\kappa_2>0$ satisfying the following:
\begin{enumerate}
\item \textbf{\emph{Quasi-geodesic}}: For any $x,y \in X$ with $d = d_{X}(x,y)$, there exists $n_{x,y} \leq \kappa_2d + \kappa_2$ so that the path $\Omega_{x,y}:\{0, \dots, n_{x,y}\} \rightarrow \cuco X$ is an $(\kappa_2,\kappa_2)$-quasi-isometric embedding with $\Omega_{x,y}(0)=x$ and $\Omega_{x,y}(n_{x,y}) = y$; and
\item \textbf{\emph{Fellow-traveling}}: If $x',y' \in X$ with $d' = d_X(x',y')$ and $d_X(x,x'), d_X(y,y') \leq 1$, then for all $t \in \{0, \dots, \max\{n_{x,y}, n_{x',y'}\}\}$, we have
$$d_X(\Omega_{x,y}(t), \Omega_{x',y'}(t)) \leq \kappa_2.$$
\end{enumerate}

\begin{itemize}

\item In addition, we say that $\{\Omega_{x,y}\}_{x,y \in X}$ is \textbf{\emph{$\Gamma$-coarsely equivariant}} with respect to a group $\Gamma < \mathrm{Isom}(X)$ if for any $g \in \Gamma$ and $x,y \in X$ and $t \in \{0, \dots, \max\{n_{x,y}, n_{x',y'}\}\}$, we have
$$d_{X}(g \cdot \Omega_{x,y}(t), \Omega_{g \cdot x, g\cdot y}(t)) < \kappa_2.$$
\end{itemize}
\end{definition}

Finally, we recall the following definition from \cite{HHS_II}, which was inspired by the the paths constructed in \cite{MM_II}:

\begin{definition}{defn:hierarchy path}
For $D\geq 1$, a path $\gamma$ in $\cuco X$ is a \emph{\textbf{$D$--hierarchy path}} if
 \begin{enumerate}
  \item $\gamma$ is a $(D,D)$--quasigeodesic,
  \item for each $W\in\mathfrak S$, $\pi_W\circ\gamma$ is an unparametrized $(D,D)$--quasigeodesic.
  \end{enumerate}
  \end{definition}

We can now prove that colorable HHSes admit discrete, bounded, quasi-geodesic, coarsely equivariant bicombings by hierarchy paths. 

\begin{figure}
\includegraphics[width=1\textwidth]{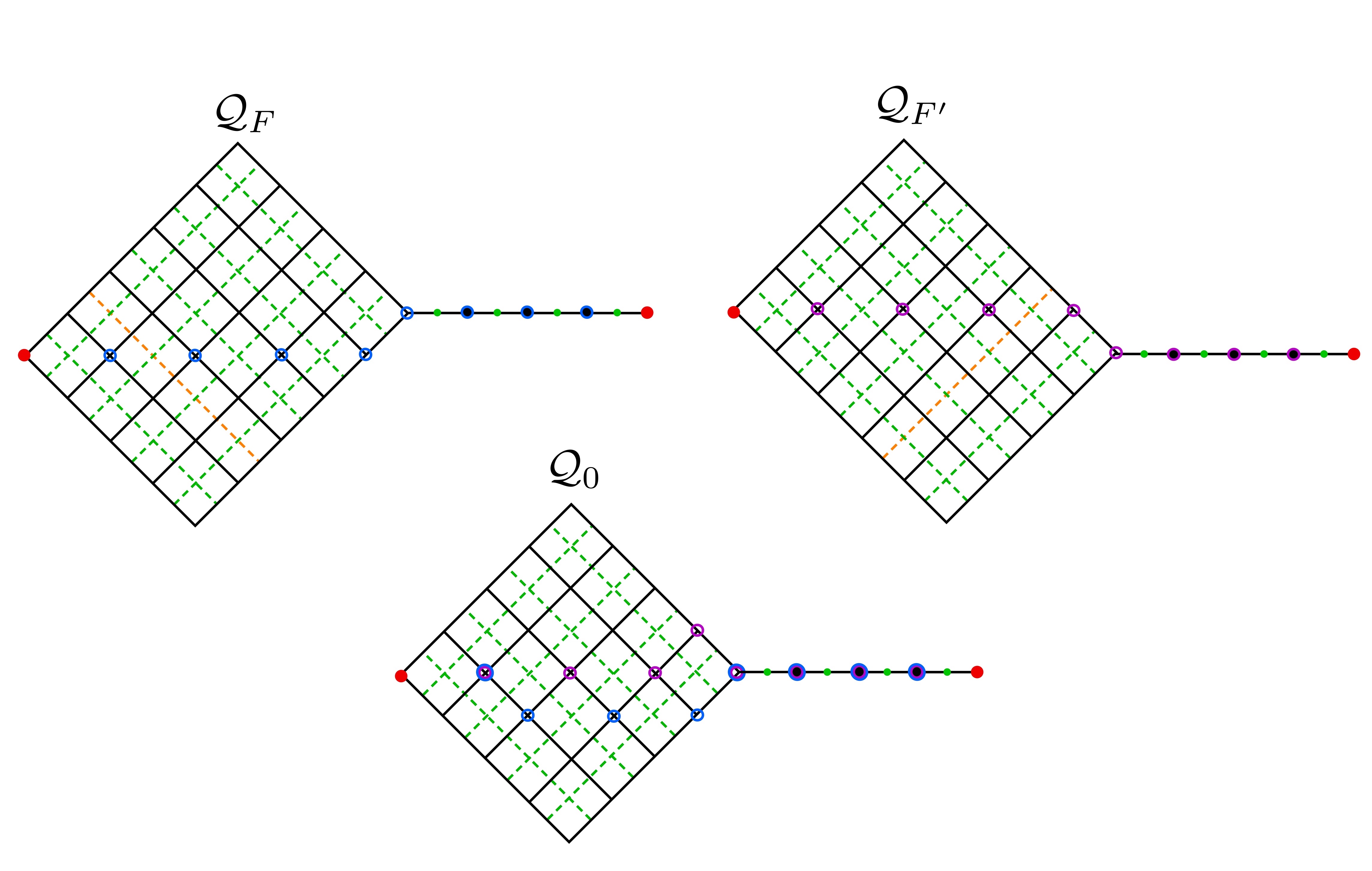}
\caption{A simple example of bicombing paths, building on the hierarchical setup from Figure \ref{fig:trees to cubes}.  Deleting the orange hyperplanes from $\QQ_F$ and $\QQ_{F'}$ results in perturbing the contraction paths in $\QQ_0$.}\label{fig:bicombing example}
\end{figure}

\begin{theorem}{thm:bicombing general}
Let $(\cuco X, \mathfrak S)$ be a $G$-colorable HHS with $G < \mathrm{Aut}(\mathfrak S)$.  Then there exists $D>0$ so that $(\cuco X, \mathfrak S)$ admits a coarsely $G$-equivariant, discrete, bounded, quasi-geodesic bicombing by $D$-hierarchy paths.
\end{theorem}

\begin{proof}
Let $(\cuco X, \mathfrak S)$ be a colorable HHS.  We again use the notation from the statement of Proposition \ref{prop:stable contractions in HHS}, where now $F= \{x,y\}$ and $F' = \{x',y'\}$ with $d_{\cuco X}(x,x') \leq 1$ and $d_{\cuco X}(y,y') \leq 1$.  We make a blanket observation that $k=2$ and so the constant $K_3$ in Proposition \ref{prop:stable contractions in HHS} depends only on $(\cuco X, \mathfrak S)$.

Coarse equivariance can be obtained using the argument below, setting $F'=g F$.  We omit the details for readability.
\par\smallskip

{\bf Construction of the bicombing paths.} Let $\psi = \psi_0$ and define a map $\omega_{x,y}: \{0, \dots, 2n_{x,y}\} \rightarrow \QQ_F$ by

$$\omega_{x,y}(i) = \begin{cases}
\psi_i(x) & i \in \{0, \dots, n_{x,y}\}\\
 \psi_{2n_{x,y}-i}(y) \ & i \in \{n_{x,y}+1,\dots,  2n_{x,y}\}
\end{cases}$$

We claim that $\omega_{x,y}$ is a $(C,C)$-quasi-geodesic in the $\ell^1$-metric on $\QQ_F$, for some uniform $C$. First, the points $\omega_{x,y}(0),\dots, \omega_{x,y}(n_{x,y})$ appear on an $\ell^1$-geodesic $\gamma_1$ from $\omega_{x,y}(0)$ to $\omega_{x,y}(n_{x,y})$ in the given order by Theorem \ref{thm:stable contraction}(\ref{item:contraction_geod}), and the same holds for $\omega_{x,y}(n_{x,y}+1),\dots, \omega_{x,y}(2n_{x,y})$ for some $\ell^1$-geodesic $\gamma_2$ from $\omega_{x,y}(n_{x,y}+1)$ to $\omega_{x,y}(2n_{x,y})$. Moreover, consecutive $\omega_{x,y}(i)$ are uniformly close to each other, since they are at distance at most $1$ in the $\ell^\infty$-metric, which is uniformly quasi-isometric to the $\ell^1$-metric with constant only depending on the dimension of $\QQ_F$, which in turn only depends on $\mathfrak S$.

Let $\gamma$ be the concatenation of $\gamma_1$, an $\ell^1$-geodesic from $\omega_{x,y}(n_{x,y})$ to $\omega_{x,y}(n_{x,y}+1)$, and $\gamma_2$.  Since no hyperplane can separate $\{x,y\}$ from $\psi_{n_{x,y}}(x)$ or $\psi_{n_{x,y}}(y)$, again by Theorem \ref{thm:stable contraction}, we see that $\gamma$ crosses each hyperplane at most once, and is therefore an $\ell^1$-geodesic. Since $\omega_{x,y}(n_{x,y})$ and $\omega_{x,y}(n_{x,y}+1)$ are just opposite corners of a cube, we see that the $\omega_{x,y}(i)$ appear along an $\ell^1$-geodesic in the given order, and with uniformly spaced gaps. This shows that $\omega_{x,y}$ is a $(C,C)$-quasi-geodesic in the $\ell^1$-metric, for $C$ depending only on the dimension of $\QQ_F$ and hence only on $\mathfrak S$.

It follows then that the composition 
$$\Omega_{x,y} = \Phi_F \circ \omega_{x,y}: \{0, \dots, 2n_{x,y}\} \rightarrow \cuco X$$
is a $(K_4,K_4)$-quasi-geodesic in $\cuco X$ with $K_4 = K_4(\mathfrak S)$. We can perturb it a uniformly bounded amount at the endpoints to make sure that the endpoints are $x$ and $y$; with a slight abuse of notation we still denote the perturbation $\Omega_{x,y}$ and the quasi-isometry constants $K_4$.

\par\smallskip

This proves that Definition \ref{defn:bicombing}(1) holds for the family $\left\{\Omega_{x,y}\right\}_{x,y \in \cuco X}$.

\par\smallskip

{\bf Fellow-travelling.} We now prove the fellow-traveling condition in Definition \ref{defn:bicombing}(2) holds.  Once again adopting our previous notation, we want to prove that there exists $\kappa_2 = \kappa_2(\cuco X, \mathfrak S)>0$ so that for any $t \in \{0, \dots, \max\{n_{x,y}, n_{x',y'}\}\}$, we  have
\begin{equation}\label{eqn:bicombing inequality}
d_{\cuco X}(\Omega_{x,y}(t), \Omega_{x',y'}(t)) < \kappa_2.
\end{equation}

Without loss of generality, suppose that $n_{x,y} \geq n_{x',y'}$ and recall that Proposition \ref{prop:stable contractions in HHS}(1) gives that $\delta = n_{x,y} - n_{x',y'} < K_3$, where $K_3$ depends only on $(\cuco X, \mathfrak S)$.  There are four cases to consider:

\begin{itemize}
\item[(i)] When $0 \leq i \leq n_{x',y'}$, where $\Omega_{x,y}$ and $\Omega_{x',y'}$ are defined using $x$ and $x'$, respectively;
\item[(ii)]When $n_{x',y'} < j \leq n_{x,y}$, where $\Omega_{x,y}$ is defined using $x$ whereas $\Omega_{x',y'}$ is defined using $y'$;
\item[(iii)] When $n_{x,y} < q \leq 2n_{x',y'}$, where both $\Omega_{x,y}$ and $\Omega_{x',y'}$ are nonconstant and defined using $y$ and $y'$, respectively;
\item[(iv)] When $2n_{x',y'} <r \leq 2n_{x,y}$, when $\Omega_{x,y}$ is nonconstant but $\Omega_{x',y'}(j) = \Omega_{x',y'}(2n_{x',y'})$ for all such $j$.
\end{itemize}

In what follows, we will repeatedly use the fact that $\Omega_{x,y}$ and $\Omega_{x',y'}$ are $(K_4, K_4)$-quasi-geodesics.  Also, set $K_5= K_4 \cdot (2\delta) + K_4$. 

In case (i), equation \eqref{eqn:bicombing inequality} follows immediately from Proposition \ref{prop:stable contractions in HHS}(2) with $\kappa_2 = K_3$.
\par\smallskip

In case (ii), we have
$$d_{\cuco X}(\Omega_{x,y}(n_{x',y'}), \Omega_{x,y}(j)) < K_5$$
and
$$d_{\cuco X}(\Omega_{x',y'}(n_{x',y'}), \Omega_{x',y'}(j)) < K_5,$$
while Proposition \ref{prop:stable contractions in HHS}(2) gives 
$$d_{\cuco X}(\Omega_{x,y}(j+2\delta), \Omega_{x',y'}(j)) < K_3,$$
so the triangle inequality implies that \eqref{eqn:bicombing inequality} holds in this case with $\kappa_2 = 2K_5 + K_3$.
\par\smallskip

In case (iii), we have
$$d_{\cuco X}(\Omega_{x,y}(q), \Omega_{x,y}(q+2\delta)) < K_5$$
and Proposition \ref{prop:stable contractions in HHS}(2) provides
$$d_{\cuco X}(\Omega_{x,y}(q+2\delta), \Omega_{x',y'}(q)) < K_3$$
so that the triangle inequality implies that \eqref{eqn:bicombing inequality} holds with $\kappa_2 = K_3+ K_5$.
\par\smallskip

Finally, in case (iv), we have
$$d_{\cuco X}(\Omega_{x,y}(2n_{x',y'}), \Omega_{x,y}(2n_F)) < K_5$$
and Proposition \ref{prop:stable contractions in HHS}(2) provides
$$d_{\cuco X}(\Omega_{x,y}(2n_{x',y'}), \Omega_{x',y'}(2n_{x',y'})) < K_3.$$
Since $\Omega_{x',y'}(2n_F) = \Omega_{x',y'}(2n_{x',y'})$ by convention, the triangle inequality implies that \eqref{eqn:bicombing inequality} holds with $\kappa_2 = K_3 + K_5$.

\par\smallskip
Hence we may set $\kappa_2 = 2K_5 + K_3$ to complete the proof of the fellow traveling condition in Definition \ref{defn:bicombing}(2).  This completes the proof that these paths gives a bicombing.

\par\smallskip

{\bf Hierarchy paths.} To finish the proof, we now show that $\Omega_{x,y}$ is a $D$--hierarchy path for some $D = D(\mathfrak S)>0$ (Definition \ref{defn:hierarchy path}). We will use that $\Phi_F$ is a $K$--median map (Theorem \ref{thm:stable_cubulations}-(2));
we now recall what this means.

In a CAT(0) cube complex $\QQ$ one can define a map $m_\QQ:\QQ^3\to\QQ$ (called median), and the only property of this map that we need here is that if $x,y,z$ appear in this order along an $\ell^1$ geodesic, then $m(x,y,z)=y$. Also, in an HHS $\cuco X$, one can define a map $m_{\cuco X}:\cuco X^3\to \cuco X$ with the property that there is a constant $M= M(\mathfrak S)>0$ so that for each $x,y,z\in\cuco X$ and $V\in\mathfrak S$ we have that $\pi_V(m(x,y,z))$ lies $M$--close to all geodesics connecting pairs of distinct elements of $\{\pi_V(x),\pi_V(y),\pi_V(z)\}$; see \cite{HHS_II} for details. Finally, $\Phi_F$ being $K$--median means that for all $x,y,z\in\QQ_F$ we have
$$d_{\cuco X}(\Phi_F(m_{\QQ_F}(x,y,z))\ ,\ m_{\cuco X}(\Phi_F(x), \Phi_F(y),\Phi_F(z)))\leq K.$$ From this, we see that there exists $M' = M'(\mathfrak S)>0$ so that for all $x,y\in\cuco X$, $0\leq i\leq j\leq k\leq 2n_{x,y}$, and $V\in\mathfrak S$, we have that $\pi_V(\Omega_{x,y}(j))$ lies $M'$--close to geodesics connecting $\pi_V(\Omega_{x,y}(i))$ and $\pi_V(\Omega_{x,y}(k))$. This suffices to prove that $\Omega_{x,y}$ is a hierarchy path.  This completes the proof.
\end{proof}

\bibliographystyle{alpha}
\bibliography{bary}

\end{document}